\documentclass[11pt,reqno]{amsart}

\usepackage[leqno]{amsmath}
\usepackage{amsthm}
\usepackage{amsfonts, tikz-cd}
\usepackage{amssymb}
\usepackage{eucal}
\usepackage[all]{xy}
\usepackage{mathtools}
\usepackage{stmaryrd}
\usepackage{csquotes}
\usepackage{enumitem}
\usepackage{upgreek}

\setenumerate{itemsep=2pt,parsep=2pt,before={\parskip=2pt}}

\usepackage[colorlinks=true,hyperindex, linkcolor=magenta, pagebackref=false, citecolor=cyan,pdfpagelabels]{hyperref}

\newcommand{\xrightarrowdbl}[2][]{%
  \xrightarrow[#1]{#2}\mathrel{\mkern-14mu}\rightarrow
}

\usepackage{xspace}
\usepackage{epsfig,epic,eepic,latexsym,color}
\usepackage[all]{xy}
\usepackage{comment} 
\numberwithin{equation}{section}

\usepackage[backend=biber, style=alphabetic, maxalphanames=5, maxnames=99, useprefix=true]{biblatex}
\usepackage[english]{babel}
\usepackage{latexsym}

\newcommand{\nc}{\newcommand}
\nc{\rnc}{\renewcommand}
\rnc{\P}{\mathbf P}
\nc{\R}{\mathbf R}
\rnc{\rm}{\mathrm}
\rnc{\bf}{\mathbf}
\nc{\cal}{\mathcal}

\nc{\C}{\mathbf C}
\nc{\Q}{\mathbf Q}
\nc{\Z}{\mathbf Z}
\nc{\A}{\mathbf A}

\nc{\rR}{\mathrm{R}}

\nc{\an}{\operatorname{an}}
\nc{\perfd}{\operatorname{perfd}}
\nc{\perf}{\operatorname{new, perf}}
\nc{\diam}{\diamondsuit}

\nc{\htt}{\operatorname{ht}}
\nc{\Nm}{\operatorname{Nm}}
\nc{\Ker}{\operatorname{Ker}}
\nc{\mmod}{\operatorname{mod}}
\nc{\End}{\operatorname{End}}
\nc{\Tor}{\operatorname{Tor}}
\nc{\coker}{\operatorname{Coker}}
\nc{\dR}{\mathrm{dR}}
\nc{\crys}{\mathrm{crys}}
\nc{\dcrys}{\mathrm{dcrys}}
\nc{\cris}{\mathrm{cris}}
\nc{\Fil}{\mathrm{Fil}}
\nc{\gr}{\mathrm{gr}}
\nc{\conj}{\mathrm{conj}}

\nc{\Aut}{\operatorname{Aut}}
\nc{\cont}{\text{cont}}
\nc{\sep}{\text{sep}}
\nc{\Hom}{\mathrm{Hom}}
\nc{\Gal}{\mathrm{Gal}}
\nc{\Spec}{\text{Spec}\,}
\nc{\RZ}{\operatorname{RZ}}
\nc{\Syn}{\mathrm{Syn}}
\nc{\ProSyn}{\mathrm{ProSyn}}
\nc{\psyn}{\mathrm{psyn}}
\nc{\Psyn}{\mathrm{pSyn}}
\nc{\aff}{\mathrm{aff}}
\nc{\hocolim}{\operatorname{hocolim}}

\rnc{\t}{\tau}
\nc{\mm}{\pmb{\mu}}
\rnc{\a}{\alpha}
\nc{\n}{\mathfrak n}
\nc{\m}{\mathfrak m}
\nc{\mfs}{\mathfrak s}
\nc{\mf}{\mathfrak f}
\nc{\e}{\varepsilon}
\nc{\dd}{\delta}
\nc{\Imm}{\operatorname{Im}}
\rnc{\sp}{\operatorname{sp}}
\nc{\Supp}{\operatorname{Supp}}

\nc{\p}{\mathfrak p}
\nc{\q}{\mathfrak q}

\nc{\Sym}{\operatorname{Sym}}
\nc{\codim}{\operatorname{codim}}
\nc{\rk}{\operatorname{rk}}
\nc{\GL}{\operatorname{GL}}
\nc{\SL}{\operatorname{SL}}
\nc{\Lie}{\operatorname{Lie}}
\nc{\Ind}{\operatorname{Ind}}
\nc{\Div}{\underline{Div}}
\nc{\Pic}{\mathbf{Pic}}
\nc{\red}{\mathrm{red}}
\nc{\uPic}{\underline{ \mathbf{Pic}}}
\nc{\rH}{\mathrm{H}}
\nc{\Spf}{\operatorname{Spf}}
\nc{\Frac}{\operatorname{Frac}}
\nc{\colim}{\operatorname{colim}}
\nc{\Spa}{\operatorname{Spa}}
\rnc{\Spec}{\operatorname{Spec}}
\nc{\Alg}{\operatorname{Alg}}
\nc{\Poly}{\operatorname{Poly}}
\nc{\PShv}{\operatorname{PShv}}
\nc{\Shv}{\operatorname{Shv}}
\nc{\Fun}{\operatorname{Fun}}
\nc{\op}{\mathrm{op}}
\nc{\id}{\mathrm{id}}
\nc{\tr}{\mathrm{tr}}
\nc{\IC}{\mathrm{IC}}
\nc{\Ann}{\mathrm{Ann}}
\nc{\Ass}{\mathrm{Ass}}
\nc{\WeakAss}{\mathrm{WeakAss}}

\rnc{\an}{\operatorname{an}}
\nc{\et}{\text{\'et}}
\nc{\Et}{\text{\'Et}}

\rnc{\et}{\text{\'et}}
\nc{\proet}{\text{pro\'et}}
\nc{\syn}{\text{syn}}
\nc{\prosyn}{\text{prosyn}}

\nc{\xr}{\xrightarrow}

\nc{\eps}{\epsilon}
\nc{\ov}{\overline}
\nc{\ud}{\underline}
\nc{\wdh}{\widehat}

\nc{\bG}{\mathbf G}
\nc{\bZ}{\mathbf Z}

\nc{\F}{\mathcal F}
\nc{\G}{\mathcal G}
\nc{\E}{\mathcal E}
\nc{\K}{\mathcal K}
\nc{\I}{\mathcal I}
\nc{\sQ}{\mathcal Q}

\nc{\cY}{\mathcal Y}
\nc{\cU}{\mathcal U}
\nc{\cV}{\mathcal V}

\nc{\X}{\mathcal X}
\nc{\T}{\mathfrak T}

\nc{\LL}{\mathcal{L}}
\rnc{\S}{\mathcal S}
\nc{\M}{\mathcal M}
\nc{\sU}{\mathfrak U}
\nc{\V}{\mathfrak V}
\nc{\N}{\mathrm N}

\nc{\rV}{\mathrm{V}}

\nc{\ra}{\rangle}

\nc{\os}{\overset}

\rnc{\O}{\mathcal O}
\nc{\J}{\mathcal J}
\theoremstyle{definition}

\newtheorem{thm1}{Theorem}[section]
\newtheorem{lemma1}[thm1]{Lemma}

\newtheorem{prop1}[thm1]{Proposition}

\newtheorem{rmk1}[thm1]{Remark}

\newtheorem{cor1}[thm1]{Corollary}

\newtheorem{conjecture1}[thm1]{Conjecture}
\newtheorem{question1}[thm1]{Question}
\newtheorem{example1}[thm1]{Example}

\newtheorem{thm}{Theorem}[subsection]
\newtheorem{lemma}[thm]{Lemma}
\newtheorem{defn}[thm]{Definition}
\newtheorem{prop}[thm]{Proposition}
\newtheorem{construction}[thm]{Construction}
\newtheorem{conjecture}[thm]{Conjecture}

\newtheorem{rmk}[thm]{Remark}

\newtheorem{warning}[thm]{Warning}
\newtheorem{cor}[thm]{Corollary}
\newtheorem{variant}[thm]{Variant}

\newtheorem{question2}[thm]{Question}
\newtheorem{notation}[thm]{Notation}

\begin{document}

\title[Arithmetic properties of \'etale cohomology and nearby cycles]{Arithmetic properties of $\ell$-adic \'etale cohomology and nearby cycles of rigid analytic spaces}
\author{David Hansen}
\author{Bogdan Zavyalov}
\maketitle
\begin{abstract} We prove a number of results on the \'etale cohomology of rigid analytic varieties over $p$-adic non-archimedean local fields. Among other things, we establish bounds for Frobenius eigenvalues, show a strong version of Grothendieck's local monodromy theorem, prove mixedness of the nearby cycle sheaf, and show that for any formal model, the IC sheaf on the special fiber is captured by the nearby cycles of the IC sheaf on the generic fiber. We also prove a local version of Deligne's weight-monodromy conjecture, by a novel perfectoid analysis of nearby cycles. 

Along the way, we develop the theory of ``constructible $\ell$-adic complexes on Deligne's topos'' (six operations, perverse t-structure, a notion of mixedness, etc.), which is prerequisite to a precise discussion of the Galois action on nearby cycles for algebraic and rigid analytic varieties over non-archimedean fields.
\end{abstract}

\tableofcontents

\section{Introduction}

This paper centers around three main results, all dealing with $\ell$-adic cohomology groups of quasi-compact and quasi-separated rigid-analytic varieties over $p$-adic non-archimedean fields. We briefly list these results here and then discuss each result in more detail. The first result is a proof of (a slightly weakened version of) a conjecture of Bhatt--Hansen (see \cite[Conjecture 4.15]{Bhatt-Hansen}). The second result is a strong uniform-in-$\ell$ version of Grothendieck's local monodromy theorem for rigid-analytic varieties. The last result concerns a local version of the weight-monodromy conjecture for nearby cycle sheaves, which was proved in the equal characteristic case by Gabber; this paper makes the first progress in mixed characteristic. \smallskip

Besides these three main points, we also develop the general theory of $\Z_\ell$- and $\Q_\ell$-constructible sheaves on Deligne's topos (see Definition~\ref{defn:deligne-topos}) in significant detail. These results are crucial even to give a correct formulation of \cite[Conjecture 4.15]{Bhatt-Hansen} and the local weight-monodromy conjecture for nearby cycles. 

\subsection{Deligne's topos}\label{intro-deligne}

Our main initial goal was to prove  \cite[Conjecture 4.15]{Bhatt-Hansen}. However, it quickly turned out that even to formulate \cite[Conjecture 4.15]{Bhatt-Hansen} correctly (or the nearby cycle version of the weight-monodromy conjecture), we have to use sheaves on Deligne's topos and their structure theory. \smallskip

Let us briefly explain the main source of this necessity. According to \cite[Conjecture 4.15]{Bhatt-Hansen}, for an admissible formal $\O_K$-scheme $\X$, the nearby cycles complex $\rm{R}\Psi_\X \rm{IC}_{\X_\eta, \Q_\ell}$ should be a mixed perverse sheaf on the special fiber. However, this claim does not quite make sense, since the nearby cycles is {\it not} a complex of sheaves on the special fiber $\X_s$. What it is, rather, is a 
\begin{center}
    ``complex of $\Q_\ell$-sheaves on the geometric special fiber $\X_{\ov{s}}$ with a continuous action of $G_K$ compatible with the action of $G_K$ on $\X_{\ov{s}}$''.
\end{center} 
This definition, however, is rather difficult to make precise by hand, and the additional problem of defining the six functors for such sheaves suggests we should take a more conceptual approach. Also, since the nearby cycles do not have any preferred descent to a complex of sheaves on the special fiber, we might instead try to adapt the notion of mixedness to this situation. \smallskip

We resolve both issues in Appendix~\ref{appendix:deligne}, Appendix~\ref{appendix:adic}, and Section~\ref{section:deligne}. The results of Appendix~\ref{appendix:deligne} and Appendix~\ref{appendix:adic} are (mostly) not new, but they seem very difficult to find explicitly stated in the literature. Therefore, we decided to present these results in the generality needed for this paper. The material of Section~\ref{section:deligne} seems to be somewhat known to the experts, but we were not able to find any rigorous discussion of these results in the literature. In particular, even a precise definition of a mixed sheaf on Deligne's topos seems not to be present in the existing literature. \smallskip

We now briefly summarize the main results of each of these sections in more detail. Throughout this discussion, we fix a non-archimedean field $K$ with ring of integers $\O_K$ and residue field $k$. We also fix a prime number $\ell$.  \smallskip

In Appendix~\ref{appendix:deligne}, we follow \cite{SGA7_2} and introduce the notion of Deligne's topos $X\times_s \eta$ for a finite type $k$-scheme $X$ (see Definition~\ref{defn:deligne-topos}). Although general product topoi are quite abstract, Deligne's topos is very concrete, and gives a precise meaning to the intuition of a sheaf on $X_{\ov{s},\et}$ with a continuous compatible $G_K$-action. In particular, we show that there is a morphism of topoi $\pi_X\colon X_{\ov{s}, \et} \to X\times_s \eta$ (see Lemma~\ref{lemma:properties-Deligne}), where intuitively $\pi_X^{\ast}$ corresponds to forgetting the $G_K$-action. As evidence for this intuition, we show that for any object $\F\in D(X\times_s\eta; \Z/\ell^n)$, there is a functorial ``action of $G_K$'' on the pullback $\pi^*_X\F$ compatible with the action of $G_K$ on $X_{\ov{s}}$ (see Construction~\ref{construction:action} and Construction~\ref{construction:action-adic} for a precise formulation). The rest of Appendix~\ref{appendix:deligne} is devoted to defining six functors for Deligne's categories and the (analytic and algebraic) nearby cycles; here we closely follow some ideas and constructions of Lu--Zheng \cite{Lu-Zheng}. Appendix~\ref{appendix:adic} is devoted to extending these results to $\Z_\ell$ and $\Q_\ell$-coefficients. We also show that the ``derived category of constructible $\Q_\ell$-sheaves'' $D^b_c(X\times_s \eta; \Q_\ell)$ admits both standard and perverse $t$-structures (see Corollary~\ref{cor:constructible-t-structure-rationally} and Lemma~\ref{lemma:perverse-t-structure-rationally}). \smallskip


Now we discuss the content of Section~\ref{section:deligne}. Throughout this section, we fix a non-archimedean field $K$ which is \emph{arithmetic} (see Definition~\ref{defn:arithmetic}), and a prime number $\ell$ invertible in $\O_K$. Any continuous section $\sigma\colon G_k\to G_K$ of the canonical projection $G_K\to G_k$ of Galois groups defines a morphism of topoi $\sigma_X: X_{\et} \to X\times_s \eta$ with an associated conservative pullback functor
\[
\sigma^*_X\colon D^b_c(X\times_s \eta; \Q_\ell) \to D^b_c(X; \Q_\ell)
\]
for any finite type $k$-scheme $X$. We show that, for any $\F\in D^b_c(X\times_s \eta; \Q_\ell)$, mixedness of $\sigma^*_X \F$ is independent of $\sigma$:

\begin{thm}\label{intro-1}(Lemma~\ref{lemma:mixed-of-weight-d} and Corollary~\ref{cor:pure-independent}) Let $K$ be an arithmetic non-archimedean field, $X$ a finite type $k$-scheme, and
\[
    \sigma, \sigma'\colon G_k \to G_K
\]
two continuous sections, and $\F\in D^b_{c}(X\times_s \eta; \Q_\ell)$. Then $\sigma_X^*\F\in D^b_c(X; \Q_\ell)$ is pure of weight $w$ (resp. mixed of weights $\leq w$, resp. mixed of weights $\geq w$) if and only if ${\sigma}_{X}'^{*}\F\in D^b_c(X; \Q_\ell)$ is pure of weight $w$ (resp. mixed of weights $\leq w$, resp. mixed of weights $\geq w$).
\end{thm}

Theorem~\ref{intro-1} allows us to define mixed and pure sheaves on $X\times_s \eta$: we simply say that $\F\in D^b_c(X\times_s \eta; \Q_\ell)$ is mixed (resp. pure) if $\sigma_X^*\F$ is mixed (resp. pure) for a(ny) choice of a continuous section $\sigma$ (see Definition~\ref{defn:pure-mixed-Deligne}).  We also show that for mixed perverse sheaves, the weight filtration can be constructed on the level of Deligne's categories, recovering the usual weight filtration after applying $\sigma_X^*$ for any continuous section $\sigma$:

\begin{thm}\label{intro-2}(Theorem~\ref{thm:weight-filtration}) Let $K$ be an arithmetic non-archimedean field, $X$ a finite type $k$-scheme, and $\F\in \rm{Perv}(X\times_s \eta; \Q_\ell)$ a mixed perverse sheaf (see Definition~\ref{defn:perverse}). Then there is a unique functorial increasing weight filtration
\[
\rm{Fil}_{\rm{W}}^n\F \subset \F
\]
such that 
\begin{enumerate}
    \item each $\rm{Fil}_{\rm{W}}^n\F$ is a perverse sheaf;
    \item $\rm{Gr}^n_{\rm{W}} \F$ is zero or a pure sheaf of weight $n$;
    \item $\rm{Fil}_{\rm{W}}^{-n}\F = 0$ and $\rm{Fil}_{\rm{W}}^n\F = \F$ for a large $n\gg 0$.
\end{enumerate}
Furthermore, the weight filtration satisfies the following properties:
\begin{enumerate}
    \item any morphism of mixed perverse sheaves $f\colon \F\to \G$ is strictly compatible with the weight filtrations, i.e.  $f(\rm{Fil}_{\rm{W}}^\bullet \F) =\rm{Fil}_{\rm{W}}^\bullet \G \cap f(\F)$;
    \item for any continuous section $\sigma\colon G_k \to G_K$ of the projection $r\colon G_K \to G_k$, there is an equality of filtrations
    \[
    \sigma_X^*\rm{Fil}_{\rm{W}}^\bullet \F = \rm{Fil}_{\rm{W}}^\bullet \sigma_X^*\F,
    \]
    where $\rm{Fil}_{\rm{W}}^\bullet \sigma_X^*\F$ is the weight filtration from \cite[Th\'eorem\`e 5.3.5]{BBD}.
\end{enumerate}
\end{thm}

We also show that all complexes in $D^b_c(X\times_s \eta; \Q_\ell)$ automatically satisfy a version of the Grothendieck quasi-unipotence theorem and admit a canonical nilpotent monodromy operator $N$.

\begin{thm}\label{intro-3}(Corollary~\ref{cor:quasi-unipotent}, Lemma~\ref{lemma:monodromy-operator}, and Definition~\ref{defn:monodromy-operator-perverse}) Let $K$ be an arithmetic non-archimedean field, $X$ a finite type $k$-scheme, and $\F\in D^b_c(X\times_s \eta; \Q_\ell)$. Then
\begin{enumerate}
    \item there is an open subgroup $I_1\subset I$ such that the action of $I_1$ on $\pi^*_X\F$ is unipotent;
    \item there is a unique (independent of $I_1$) nilpotent morphism 
    \[
    N\colon \pi_X^*\F\to \pi_X^*\F(-1)
    \]
    in $D^b_c(X_{\ov{s}}; \Q_\ell)$ such that 
    \[
    \rho_g=\exp(Nt_\ell(g))
    \]
    for $g\in I_1$;
    \item if $\F$ is perverse, then $N$ descends to a morphism $N\colon \F \to \F(-1)$.
\end{enumerate}
\end{thm}

\subsection{Mixedness of the nearby cycles}

Using the machinery discussed in Section~\ref{intro-deligne}, we can formulate and prove the corrected (and slightly weakened) version of the $\ell$-adic conjecture from \cite{Bhatt-Hansen}:

\begin{thm}\label{intro-4} Let $X$ be quasi-compact quasi-separated rigid-analytic variety over a $p$-adic local field\footnote{Local fields are defined in Section~\ref{section:terminology}. A $p$-adic local field is always a finite extension of $\Q_p$.} $K$, and $\X$ an admissible formal $\O_K$-model of $X$ with special fiber $\X_s$, so $X=\X_\eta$. Then  
\begin{enumerate}
\item the nearby cycles $\rm{R}\Psi_{\X} \rm{IC}_{X, \Q_\ell}$ is a mixed perverse sheaf. Moreover, if $\X_s$ is of pure dimension $d$, then $\rm{IC}_{\X_s\times_s \eta, \Q_\ell}$ (see Definition~\ref{defn:IC-Deligne}) is a direct summand of the $d$-th graded piece of the weight filtration on $\rm{R}\Psi_\X \rm{IC}_{X, \Q_\ell}$ (see Theorem~\ref{thm:weight-filtration});
\item\label{2} for any $g\in G_K$ projecting to the geometric Frobenius in $G_k$ and any integer $i$, the eigenvalues of $g$ acting on $\rm{IH}^i(X_{\wdh{\ov{\eta}}}; \Q_\ell)$ are $q$-Weil numbers of weights $\geq i$;
\item If $X$ is smooth or the $\ell$-adic Decomposition theorem for rigid-analytic varieties holds (see \cite[Conjecture 4.17]{Bhatt-Hansen}), then weights in (\ref{2}) are $\geq \rm{max}(0, i)$. 
\end{enumerate}
\end{thm}

\begin{rmk} We note that the original formulation of \cite[Conjecture 4.15]{Bhatt-Hansen} contains a typo: over a non-algebraically closed base field, the functor $\rm{R}\lambda_{\X*}$ used in \cite[Conjecture 4.15]{Bhatt-Hansen} is different from the nearby cycles functor and does not preserve perverse sheaves. For a precise comparison between $\rm{R}\Psi_{\X}$ and $\rm{R}\lambda_{\X_{\O_C}\ast}$, see Lemma \ref{lemma:compute-nearby-cycles}.(1) and Remark \ref{rmk:adic-good}.
\end{rmk}

\begin{rmk} Theorem~\ref{intro-4} is weaker than \cite[Conjecture 4.15]{Bhatt-Hansen} since the latter predicts that the weights of the geometric Frobenius action are all non-negative\footnote{We use the same normalization for the intersection cohomology groups as the one used in \cite{Bhatt-Hansen}. In particular, intersection cohomology groups of a smooth space live in degrees $[-\dim X, \dim X]$.}. We can prove this either for smooth $X$ or under the assumption that the $\ell$-adic decomposition theorem holds for a resolution of singularities of $X$.
\end{rmk}

\begin{rmk} A version of Theorem~\ref{intro-4} for algebraic $\X$ has been previously obtained in \cite[\textsection 10. Appendix]{Gortz-Haines}. However, we note that the proof in {\it loc.\,cit.\,}only shows that $\rm{R}\Psi_{\X} \rm{IC}_{X, \Q_\ell}$ is a mixed perverse sheaf for {\it some} choice of a continuous section of $G_K \to G_k$, while our result ensures that this holds for {\it any} such section. In fact, in order to run our argument in the analytic situation, we have to prove the result for all sections simultaneously. This level of generality is possible thanks to Theorem~\ref{intro-1}, whose proof relies on a general careful analysis of sheaves on Deligne's topos. 
\end{rmk}

The essential idea of the proof of Theorem~\ref{intro-4} is to use perverse exactness of the nearby cycles and resolution of singularities to reduce to the smooth case. In this case, one can use Elkik's algebraization and the comparison of analytic and algebraic nearby cycles to reduce to an analogous result in the algebraic world. Then one can use de Jong's alterations to reduce to the strictly semi-stable case, where the nearby cycles were explicitly computed by T.\,Saito. \smallskip

As a byproduct of our methods, we also show that the nearby cycles of the constant sheaf are mixed and give some estimates on the weights of the Frobenius action on (ordinary and compactly supported) cohomology of $X$.

\begin{thm}\label{intro-5}(Lemma~\ref{lemma:trivial-action-any-model} and Theorem~\ref{thm:global-weights}) Let $K$ be a $p$-adic local field, and $\X$ an admissible formal $\O_K$-scheme with generic fiber $X=\X_\eta$. Then
\begin{enumerate}
    \item the nearby cycles $\rm{R}\Psi_{\X} \Q_\ell\in D^b_c(\X_s\times_s \eta; \Q_\ell)$ are mixed;
    \item For any $g\in G_K$ projecting to the geometric Frobenius in $G_k$ and any integer $i\geq 0$, the eigenvalues of $g$ acting on $\rm{H}^i(X_{\wdh{\ov{\eta}}}; \Q_\ell)$ are $q$-Weil numbers of weights $\geq 0$;
    \item For any $g\in G_K$ projecting to the geometric Frobenius in $G_k$ and any integer $i\geq 0$, the eigenvalues of $g$ acting on $\rm{H}^i_c(X_{\wdh{\ov{\eta}}}; \Q_\ell)$ are $q$-Weil numbers;
    \item For any $g\in G_K$ projecting to the geometric Frobenius in $G_k$ and any integer $i\geq 0$, the eigenvalues of $g$ acting on $\rm{IH}^i_c(X_{\wdh{\ov{\eta}}}; \Q_\ell)$ are $q$-Weil numbers of weights $\leq 2d+i$;
    \item For any $g\in G_K$ projecting to the geometric Frobenius in $G_k$ and any integer $i\geq 0$, the eigenvalues of $g$ acting on $\rm{IH}^i(X_{\wdh{\ov{\eta}}}; \Q_\ell)$ are $q$-Weil numbers of weights $\geq i$.
\end{enumerate}
\end{thm}

\subsection{Grothendieck's local monodromy theorem}

The methods used in the proof of Theorem~\ref{intro-4} can also be adapted to show the Grothendieck Local Monodromy Theorem for rigid-analytic varieties: 

\begin{thm}\label{intro-6}(Theorem~\ref{thm:local-monodromy-cohomology}) Let $K$ be a discretely valued $p$-adic non-archimedean field, $\ell\neq p$ a prime number, $\Lambda$ a ring $\Z/\ell^n\Z$, $\Z_\ell$, or  $\Q_\ell$, and $X$ a quasi-compact quasi-separated rigid-analytic variety over $K$. Then there is an open subgroup $I_1\subset I$ and an integer $N$ (both independent of $\ell\neq p$ and $\Lambda$) such that, for each $g\in I_1$, $(g-1)^N$ acts trivially on 
\[
\rm{H}^i(X_{\wdh{\ov{\eta}}}, \Lambda), \rm{H}^i_c(X_{\wdh{\ov{\eta}}}, \Lambda), \rm{IH}^i(X_{\wdh{\ov{\eta}}},\Lambda), \text{ and }\rm{IH}^i_c(X_{\wdh{\ov{\eta}}},\Lambda)
\]
for each integer $i$.
\end{thm}

The idea of the proof of Theorem~\ref{intro-6} is similar to that of Theorem~\ref{intro-5}: we reduce the general case to the case of an algebraic strictly semi-stable formal model, where the result is well-known. A more careful analysis of the proof leads us to a stronger version of Theorem~\ref{intro-6} in case of usual cohomology groups:

\begin{thm}\label{thm:intro-7}(Theorem~\ref{thm:unipotent-action-strong}) Under the same assumptions as in Theorem~\ref{intro-6}, there is an open subgroup $I_1\subset I$, independent of $\ell$ and $\Lambda$, such that for all $g\in I_1$ and all integers $i$, $(g-1)^{i+1}=0$ on  $\rm{H}^i(X_{\wdh{\ov{\eta}}}, \Lambda)$.
\end{thm}

\subsection{Weight-monodromy conjecture for the nearby cycles}

For the rest of this section, we fix a $p$-adic local field $K$ and a prime number $\ell\neq p$. \smallskip

Let $X$ be a smooth and proper $K$-scheme. Then its geometric \'etale cohomology groups $\rm{H}^i(X_{\ov{K}}, \Q_\ell)$ come equipped with the monodromy filtration $\rm{Fil}_{\rm{M}}^\bullet \rm{H}^i(X_{\ov{K}}, \Q_\ell)$. The following famous conjecture is due to P.\,Deligne, and is motivated by analogy with properties of limit mixed Hodge structures; we refer to \cite{Illusie-autor} for a beautiful overview of this circle of ideas.

\begin{conjecture}\label{conj:weight-monodromy}(Weight-Monodromy Conjecture) Let $X$ be as above. Then the eigenvalues of any geometric Frobenius lift on $\rm{gr}^j_{\rm{M}} \rm{H}^i(X_{\ov{K}}, \Q_\ell)$ are $q$-Weil numbers of weight $i+j$ for every pair of integers $i, j$.
\end{conjecture}

Conjecture~\ref{conj:weight-monodromy} is a global statement which is specific to proper algebraic varieties. In this paper, we prove (under some small additional assumption) the following local version of this conjecture, which seems to be folklore:

\begin{conjecture}\label{conj:intro-local-weight-monodromy-nearby} (Local Weight-Monodromy Conjecture) Let $\X$ be an admissible formal $\O_K$-scheme with smooth generic fiber $\X_{\eta}$. Then the nearby cycles $\rm{R}\Psi_{\X} \Q_\ell \in D^{b}_c(\X_s \times_s \eta; \Q_\ell)$ are monodromy-pure of weight $0$ (see Definition~\ref{defn:monodromy-pure}).
\end{conjecture}

\begin{thm}\label{thm:intro-8}(Theorem~\ref{thm:local-weight-monodromy}) Let $\X$ an admissible formal $\O_K$-scheme with smooth generic fiber $\X_\eta$. Suppose that each point $x\in \X$ admits a pointed \'etale morphism $(\sU, u) \to (\X, x)$ such that $\sU_\eta$ admits an \'etale morphism to a closed unit disk $\bf{D}^d_K$. Then the nearby cycles $\rm{R}\Psi_\X \Q_\ell$ is a monodromy-pure sheaf of weight $0$. 
\end{thm}

\begin{rmk} Unfortunately, Theorem~\ref{thm:intro-8} (or even Conjecture~\ref{conj:intro-local-weight-monodromy-nearby}) does not imply the global weight-monodromy conjecture. We hope that Theorem~\ref{thm:intro-8} can be combined with some other ideas to give a new approach to Conjecture~\ref{conj:weight-monodromy}. In particular, Theorem~\ref{thm:intro-8} allows to reduce the weight-monodromy conjecture to a question about varieties over finite fields. 
\end{rmk}

\begin{cor}\label{cor:intro} A smooth rigid-analytic $K$-variety $X$ admits a cofinal family of admissible formal models $\{\X_i\}_{i\in I}$ such that $\rm{R}\Psi_{\X_i} \Q_\ell$ is monodromy-pure of weight $0$.
\end{cor}

\begin{rmk} Theorem~\ref{thm:intro-8}, in particular, proves that $\rm{R}\Psi_\X\Q_\ell$ is monodromy-pure of weight $0$ for any semi-stable formal $\O_K$-scheme. Previously, it was known in the algebraic semi-stable case by an explicit calculation of the nearby cycles. Our proof is completely different: it is quite soft and does not require any explicit computations.
\end{rmk}

If $K$ is a characteristic $p$ local field (so $K\cong \bf{F}_q(\!(T)\!)$), then Theorem~\ref{thm:intro-8} holds for any admissible formal model of $X$. This result is essentially due to O.\,Gabber in the algebraic case. The analytic case can be easily deduced from this using Elkik's algebraization (see Theorem~\ref{lemma:char-p-local-weight-monodromy}). \smallskip

The proof of Theorem~\ref{thm:intro-8} is inspired by Scholze's proof of Conjecture~\ref{conj:weight-monodromy} for smooth proper varieties which can be realized as set-theoretic complete intersections in a projective space (see \cite[Theorem 1.14]{Sch0}). We briefly recall the strategy used in \cite{Sch0}. P.\,Scholze uses the embedding into a projective space to reduce Conjecture~\ref{conj:weight-monodromy} to an analogous claim for a perfectoid covering of $X$, then the tilting equivalence and subtle approximation and algebraization results (this is where the complete intersection assumption becomes necessary) allows to reduce the question to the Weight-Monodromy Conjecture in equal characteristic $p>0$. This was already proven by P.\,Deligne and T.\,Ito (see \cite{Weil2} and \cite{Ito}). \smallskip

Our idea is somewhat similar: we use the \'etale morphism $\X_\eta\to \bf{D}^d_K$ (which is assumed to exist locally on $\X$) to reduce the original claim for $X$ to an analogous claim for a suitable perfectoid covering, obtained by pullback from a canonical perfectoid covering of $\bf{D}^d_K$. Then using the tilting equivalence and the algebraization and approximation results of R.\,Elkik and Gabber--Ramero, we can eventually reduce to the equi-characteristic version already proven by O.\,Gabber. In particular, our proof does not require any explicit computations. \smallskip

These results suggest the following generalization of the weight-monodromy conjecture. \smallskip

\begin{conjecture}\label{conj:intro} Let $K$ be a $p$-adic local field, $X$ a smooth proper rigid-analytic $K$-variety, and $\ell\neq p$ a prime number. Suppose that $X$ admits an admissible formal $\O_K$-model $\X$ with a projective special fiber $\X_s$. Then the eigenvalues of any geometric Frobenius lift on $\rm{gr}^j_{\rm{M}} \rm{H}^i(X_{\wdh{\ov{\eta}}}, \Q_\ell)$ are $q$-Weil numbers of weight $i+j$ for every integers $i, j$.
\end{conjecture}

\begin{rmk} Conjecture~\ref{conj:intro} has no chances to hold for {\it all} smooth and proper rigid-analytic varieties $X$ over $K$ because it is already false for the Hopf surface. However, the condition that $X$ admits a formal model with projective reduction is very strong. This condition was first singled out by S.\,Li \cite{Li-projective}, and Hansen--Li then suggested that it might have consequences in $p$-adic Hodge theory, and in particular that it might imply Hodge symmetry \cite{Hansen-Li}. This hope was then dispelled by some explicit counterexamples constructed by A.\,Petrov \cite{Petrov-hodge}. However, $\ell$-adic cohomology has rather different formal properties than Hodge cohomology in this setting, and the spaces constructed by Petrov do satisfy Conjecture \ref{conj:intro}.
\end{rmk}

\begin{question2}\label{intro:question:super-Weil} Let $K$ be a $p$-adic local field, $\ell\neq p$ a prime number, $f\colon X \to Y$ a projective morphism of finite type $k$-schemes, and $\F\in D^b_c(X\times_s \eta; \Q_\ell)$ monodromy pure of weight $w$. Is $\rm{R}(f\times_s \eta)_*\F\in D^b_c(Y\times_s \eta; \Q_\ell)$ monodromy-pure of weight $w$?
\end{question2}

Question~\ref{intro:question:super-Weil} together with Corollary~\ref{cor:intro} imply both Conjecture~\ref{conj:intro-local-weight-monodromy-nearby} and Conjecture~\ref{conj:intro} (and also Conjecture~\ref{conj:weight-monodromy}). In particular, Corollary~\ref{cor:intro} allows us to reduce the Weight-Monodromy conjecture to a (probably very hard) question purely on the special fiber. We hope that this could help to shed some new insights on the general version of this conjecture. 

\subsection{Terminology}\label{section:terminology}

A {\it non-archimedean} field $K$ is a complete rank-$1$ valued field. A non-archimedean field $K$ is {\it $p$-adic} if $K$ is a non-archimedean field of mixed characteristic $(0, p)$. A non-archimedean field $K$ is {\it local} if it is discretely valued non-archimedean field with finite residue field. We denote ring of integers of $K$ by $\O_K$ and its residue field by $k$. \smallskip

In this paper, we always write {\it qcqs} as a shortcut for quasi-compact quasi-separated. It applies to adic spaces, formal schemes, and schemes. \smallskip

A {\it rigid-analytic variety} over a non-archimedean field $K$ is a locally finite type adic space over $\Spa(K, \O_K)$. An {\it admissible} formal $\O_K$-scheme is a (topologically) finitely presented flat formal $\O_K$-scheme. If $\X$ is an admissible formal $\O_K$-scheme, we denote by $\X_\eta$ its adic generic fiber, and by $\X_s$ its special fiber. Likewise, we denote by $\X_{\wdh{\ov{\eta}}}$ its geometric generic fiber, and by $\X_{\ov{s}}$ its geometric special fiber. More generally, if $X$ is a rigid-analytic space over $K$ and $C=\wdh{\ov{K}}$ is a completed algebraic closure of $K$, we denote the base change $X_C$ by $X_{\wdh{\ov{\eta}}}$. \smallskip

If $\cal{A}$ is a Grothendieck abelian category, we denote by $\cal{D}(\cal{A})$ its associated {\it $\infty$-derived category}. Its homotopy category is denoted by $D(\cal{A})$ and it coincides with the usual triangulated derived category of $\cal{A}$. \smallskip

We denote by $\cal{T}$ the {\it $2$-category of topoi} and by $\rm{Pith}(\cal{T})$ its pith, i.e. the $(2, 1)$-category obtained from $\cal{T}$ by removing the non-invertible $2$-morphisms (see \cite[\href{https://kerodon.net/tag/00AL}{Tag 00AL}]{kerodon}). Likewise, $\cal{C}at$ denotes the {\it $2$-category of categories} and $\rm{Pith}(\cal{C}at)$ denotes its pith.

\subsection{Acknowledgements} We learned a lot about product topoi and nearby cycles from Lu--Zheng's article \cite{Lu-Zheng}, and the influence of this paper will hopefully be clear to the reader. We would like to thank Ko Aoki, Bhargav Bhatt, Sasha Petrov, and Peter Scholze for many fruitful and inspiring discussions. We also thank Tom Haines for drawing our attention to the paper \cite{Gortz-Haines}, and for some clarifying discussions about that article. We are grateful to Weizhe Zheng and the anonymous referee for many helpful comments on the previous version of this paper. The authors gratefully acknowledge funding through the Max Planck Institute for Mathematics in Bonn, Germany, during the preparation of this work.

\section{Deligne's category}\label{section:deligne}

For the rest of this section, we fix the following notation. We fix a non-archimedean field $K$ with ring of integers $\O_K$ and residue field $k=k(s)$. In what follows, we denote by $G_s$ the absolute Galois group of $k$ and by $G_\eta$ the absolute Galois group of $K$. We also fix a prime number $\ell$ invertible in $\O_K$. \smallskip

We denote by $s$ (resp. $\eta$) the classifying topos of the pro-finite group $G_s$ (resp. $G_\eta$), or equivalently the \'etale topos of $\Spec k$ (resp. $\Spec K$ or $\Spa(K, \O_K)$); it consists of discrete sets equipped with a continuous action of $G_s$ (resp. $G_\eta$). The natural morphism $r\colon G_\eta \to G_s$ induces a canonical morphism of topoi $r\colon \eta \to s$. For each $g\in G_\eta$, we often denote its image $r(g)\in G_s$ simply by $\ov{g}$.\smallskip

We refer to Definition~\ref{defn:deligne-topos} for the definition of Deligne's topos $X\times_s \eta$ for a qcqs $k$-scheme $X$. And we refer to Definition~\ref{defn:constructible-deligne} for the definition of the ``constructible'' $\infty$-categories $\cal{D}^b_c(X\times_s \eta; \Z_\ell)$ and $\cal{D}^b_c(X\times_s \eta; \Q_\ell)$. Their homotopy categories are denoted by $D^b_c(X\times_s \eta; \Z_\ell)$ and $D^b_c(X\times_s \eta; \Q_\ell)$ respectively.

\subsection{Arithmetic fields}

The main goal of this section is to define the notion of an arithmetic field and verify its main properties. In what follows, we will mostly be interested in Deligne's category $D(X\times_s \eta; \Q_\ell)$ for an arithmetic field $K$.

\begin{defn}\label{defn:arithmetic} A non-archimedean field $K$ is {\it arithmetic} if there is a local field $L$ with ring of integers $\O_L$ and residue field $l$ such that
\begin{enumerate}
    \item there is an isomorphism $\varphi\colon l\xr{\sim} k$;
    \item there is an isomorphism of topological groups $\psi\colon G_K\xr{\sim} G_L$ compatible with $\varphi$. More precisely, the natural diagram
    \[
    \begin{tikzcd}
    G_K \arrow{r}{r_K}\arrow{d}{\psi} & G_k \arrow{d}{\varphi_*} \\
    G_L \arrow{r}{r_L} & G_l
    \end{tikzcd}
    \]
    commutes, where $r_K$ and $r_L$ are the natural reduction morphisms and $\varphi_*$ is an isomorphism induced by $\varphi$. 
\end{enumerate}
\end{defn}

\begin{rmk}\label{rmk:finite-residue-field} The residue field of any arithmetic field is finite. 
\end{rmk}

We first discuss some examples of arithmetic fields. 

\begin{lemma}\label{lemma:perfectified-completion} Let $K\cong \bf{F}_q((T))$ be the field of Laurent power series, and let $\wdh{K_{\rm{perf}}}$ be its completed perfection. Then $\wdh{K_{\rm{perf}}}$ is an arithmetic field. 
\end{lemma}
\begin{proof}
    Since the residue field of $K$ is perfect, we conclude that the natural morphism $K \to \wdh{K_{\rm{perf}}}$ induces an isomorphism on residue fields. Therefore, it suffices to show that the natural morphism 
    \[
    G_{\wdh{K_{\rm{perf}}}} \to G_K
    \]
    is an isomorphism. By the invariance of \'etale site under universal homeomorphisms, we conclude that the natural morphism
    \[
    G_{K_{\rm{perf}}}\to G_K
    \]
    is an equivalence. So it suffices to show that 
    \[
    G_{\wdh{K_{\rm{perf}}}}\to G_{K_{\rm{perf}}}
    \]
    is an isomorphism. Now note that $\O_{K_{\rm{perf}}}$ is $T$-adically henselian (as a filtered colimit of $T$-adically complete rings). Therefore, it is henselian with respect to its maximal ideal by \cite[\href{https://stacks.math.columbia.edu/tag/09XJ}{Tag 09XJ}]{stacks-project} and the observation that $\rm{rad}(T)=\mathfrak{m}_{\O_{K_{\rm{perf}}}}$. Therefore, \cite[Proposition 2.4.3]{Berkovich-etale} ensures that $K_{\rm{perf}}$ is quasi-complete (in the sense of \cite[Definition 2.3.1]{Berkovich-etale}). And so \cite[Proposition 2.4.1]{Berkovich-etale} implies that the natural morphism
    \[
    G_{\wdh{K_{\rm{perf}}}} \to G_{K_{\rm{perf}}}
    \]
    is an isomorphism. 
\end{proof}

For the next definition, we fix an algebraic closure $\Q_p \subset \ov{\Q}_p$ and a choice $\left\{p^{1/p^n}\right\}_{n\in \bf{N}}$ of compatible $p$-power roots of $p$.

\begin{defn}\label{defn:standard-zp-extension} For a finite extension $\Q_p\subset K$, a {\it  $p^{1/p^\infty}$-Kummer extension} $K\subset K_{\infty}=K(p^{1/p^\infty})$ is the field $\wdh{K\left(\cup_{n=1}^\infty p^{1/p^n}\right)}$ is the $p$-adic completion of the field obtained by adding all $p$-power roots $p^{1/n}$. 
\end{defn}

\begin{warning} The definition of $K_\infty$ depends on a choice of an algebraic closure $\ov{\Q}_p$ and a sequence of compatible $p$-power roots of $p$.
\end{warning}

\begin{lemma} Let $\Q_p\subset \Q_p\left(p^{1/p^\infty}\right)$ be a $p^{1/p^\infty}$-Kummer extension. Then $\Q_p\left(p^{1/p^\infty}\right)$ is an arithmetic field. 
\end{lemma}
\begin{proof}
    Firsly, we note that $K$ is a perfectoid in the sense of \cite[Definition 1.2]{Sch0}. So essentially the claim follows from the tilting equivalence (see \cite[Theorem 3.7]{Sch0}). For the reader's convenience, we spell out the argument in more detail.\smallskip
    
    Namely, we first use that there is a unique perfectoid field $K^\flat$ of characteristic $p$ with a pseudo-uniformizer $T\in K^\flat$ and an isomorphism
    \[
    \O_{K}/p \simeq \O_{K^\flat}/T.
    \]
    This implies that residue field of $K$ and $K^\flat$ are canonically isomorphic, and \cite[Theorem 3.7]{Sch0} implies that $G_K\simeq G_{K^\flat}$. Moreover, the proof of \cite[Theorem 3.7]{Sch0} ensures that this isomorphism is compatible with the isomorphism on residue field in the sense of Definition~\ref{defn:arithmetic}).\smallskip
    
    Now we claim that $K^\flat \simeq \wdh{\bf{F}_p((T))_{\rm{perf}}}$. Indeed, this follows from the observation that $\wdh{\bf{F}_p((T))_{\rm{perf}}}$ is a perfectoid field of characteristic $p$ (that is equivalent to being perfect) and a sequence of isomorphisms
    \[
    \O_K/p\simeq \bf{Z}_p[p^{1/p^\infty}]/p \simeq \bf{F}_p[T^{1/p^\infty}] \simeq \O_{\wdh{\bf{F}_p((T))_{\rm{perf}}}}/T.
    \]
    Finally, (the proof of) Lemma~\ref{lemma:perfectified-completion} ensures that the natural morphism
    \[
    \bf{F}_p((T)) \to K^\flat
    \]
    induces an isomorphism on residue fields and Galois groups, so $G_K\simeq G_{\bf{F}_p((T))}$ and $k\simeq \bf{F}_p$ in a compatible way. 
\end{proof}

\begin{lemma}\label{lemma:finite-extensions} Let $K$ be an arithmetic field and let $K\subset K'$ be a finite separable extension. Then $K'$ is an arithmetic field. 
\end{lemma}
\begin{proof}
    By Galois theory, a finite separable extension $K\subset K'$ corresponds to an open subgroup $G'\subset G_K$. Using the isomorphism $G_K\simeq G_L$ for a local field $L$, we can transform $G'$ to an open subgroup $G''\subset G_L$ that defines a finite extension $L\subset L'$. Using that the isomorphism $G_K\simeq G_L$ is compatible with an isomorphism of residue field $\varphi\colon l\simeq k$, we conclude that the images of $G'$ and $G''$ coincide in $G_l$ under the isomorphism $\varphi_*\colon G_k\simeq G_l$. Then 
    \[
    G_{K'}\simeq G' \simeq G'' \simeq G_{L'}.
    \]
    One easily checks that this isomorphism is compatible with an isomorphism on residue fields, so $K'$ is arithmetic. 
\end{proof}

\begin{lemma}\label{lemma:standard-zp-extension} Let $\Q_p\subset K$ be a finite extension and let $K\subset K_\infty$ be a $p^{1/p^\infty}$-Kummer extension. Then $K_\infty$ is an arithmetic field.
\end{lemma}
\begin{proof}
    Firstly, we note that $K\left(\cup_{n=1}^\infty p^{1/p^n}\right)$ is sub-algebra of the tensor product
    \[
    K\otimes_{\Q_p}\Q_p\left(\cup_{n=1}^\infty p^{1/p^n}\right).
    \]
    Therefore, $K\left(\cup_{n=1}^\infty p^{1/p^n}\right)$ is a finite separable extension of $\Q_p\left(\cup_{n=1}^\infty p^{1/p^n}\right)$. Now (similarly to the argument used in Lemma~\ref{lemma:perfectified-completion}) $\Q_p(\cup_{n=1}^\infty p^{1/p^n})$ is quasi-complete in the sense of \cite[Definition 2.3.1]{Berkovich-etale}, and thus \cite[Proposition 2.4.1]{Berkovich-etale} ensures that
    \[
    \wdh{\Q_p(\cup_{n=1}^\infty p^{1/p^n})} \subset K_\infty
    \]
    is a finite separable extension. Thus, $K_\infty$ is an arithmetic field by Lemma~\ref{lemma:finite-extensions}.
\end{proof}

\begin{rmk}\label{rmk:non-standard-tilt} The proofs of Lemma~\ref{lemma:finite-extensions} and Lemma~\ref{lemma:standard-zp-extension} actually show slightly more. For any finite extension $\Q_p\subset K$, there is a unique pair of a characteristic $p$ local field $L$ and a morphism
\[
\a\colon L \to K^{\flat}_\infty
\]
such that $\a$ realizes $K^\flat_\infty$ as a completed perfection of $L$. In what follows, we call $L$ a {\it non-standard tilt} of $K$ and denote it by $K^\flat$. 
\end{rmk}

\subsection{Inertia action: the case of a point}

For the rest of this section, we fix a non-archimedean field $K$ of residue characteristic $p$, and a prime number $\ell\neq p$.  \smallskip

We refer to Definition~\ref{defn:deligne-topos} for the definition of Deligne's topos $X\times_s \eta$ for a qcqs $k$-scheme $X$, and to Definition~\ref{defn:constructible-deligne} for the definition of the constructible ``derived'' categories $D^b_c(X\times_s \eta; \Z_\ell)$, $D^b_c(X\times_s \eta; \Q_\ell)$. \smallskip

We recall that, for a finite type $k$-scheme $X$ and $\F \in D^b_c(X\times_s \eta; \Lambda)$ for $\Lambda\in \{\Z/\ell^n\Z, \Z_\ell, \Q_\ell\}$, we have a well-defined action
\[
\rho\colon I\to \rm{Aut}_{\Lambda}\left(\pi^*_X \F\right) 
\]
discussed in Construction~\ref{construction:action} and Construction~\ref{construction:action-adic}. More generally, for any $g\in G_\eta$ with an image $\ov{g}\in G_s$, we have a well-defined isomorphism
\[
\rho_g\colon \ov{g}^*\left(\pi_X^*\F\right) \xr{\sim} \pi_X^*\F
\]
such that these isomorphisms satisfy the cocyle condition 
\[
\rho_{gh}= \rho_g \circ \ov{g}^*(\rho_h)
\]
``up to a canonical identification $\ov{gh}^*\simeq \ov{g}^*\circ \ov{h}^*$''. We also note that if $X=\Spec k$ is the base point, there is an equivalence $D^b_c(X; \Lambda)\simeq D^b_c(\eta; \Lambda)$, so $\rho$ extends to a homomorphism
\[
\rho\colon G_\eta \to \rm{Aut}_{\Lambda}(\pi^*_{s}\F)
\]

The main goal of this and the next sections is to show that, for an arithmetic field $K$, the inertia action 
\[
\rho\colon I \to \rm{Aut}_{\Q_\ell}\left(\pi^*_X \F\right) 
\]
is always continuous for any $\F\in D^b_c(X\times_s\eta; \Q_\ell)$ and an explicitly specified topology on the automorphism group. Our argument will be somewhat roundabout: we first treat the case $X=\Spec k$ and then deduce the general case from this one. \smallskip

We start by defining topology on $\rm{Aut}_{\Lambda}(\pi^*_{s} \F)$. We will need the following well-known lemma:

\begin{lemma}\label{lemma:different-topologies} Let $\ov{X}$ be a finite type $\ov{k}$-scheme, let $\ell$ be a prime number invertible in $k$, and let $\F, \G \in D^b_c(\ov{X}; \Z_\ell)$. Suppose $\F\simeq \lim_n \F_{n}$, $\G=\lim_n \G_n$ with $\F_n, \G_n \in D^b_{ctf}(\ov{X}; \Z/\ell^n\Z)$. Then the natural morphism
\[
\rm{Hom}_{\Z_\ell}(\F, \G)\to \lim_n \rm{Hom}_{\Z/\ell^n\Z}(\F_n, \G_n)
\]
is an isomorphism, $\rm{Hom}_{\Z_\ell}(\F, \G)$ is finitely generated, and the limit topology on $\rm{Hom}_{\Z_\ell}(\F, \G)$ coincides with the $\ell$-adic topology. 
\end{lemma}
\begin{proof}
    We start with the first claim. By construction, we have Milnor's exact sequence computing Homs in the homotopy category of an $(\infty)$-limit of $\infty$-categories:
    \[
    0 \to \rm{R}^1\lim_n \rm{Ext}^{-1}_{\Z/\ell^n\Z}(\F_n, \G_n) \to \rm{Hom}_{\Z_\ell}(\F, \G) \to \lim_n \rm{Hom}_{\Z/\ell^n\Z}(\F_n, \G_n) \to 0.
    \]
    Now \cite[Theorem 9.5.3]{Lei-Fu} ensures that all groups $\rm{Ext}^{-1}_{\Z/\ell^n\Z}(\F_n, \G_n)$ are finite. So the Mittag-Leffler criterion implies vanishing of the $\rm{R}^1\lim$ term. Now  \cite[Proposition 0.7.2.11]{FujKato} guarantees that $\rm{Hom}_{\Z_\ell}(\F, \G)$ is finitely generated and the limit topology coincides with the $\ell$-adic topology. 
\end{proof}

For the next definition, we fix a finite type $\ov{k}$-scheme $\ov{X}$.

\begin{defn}\label{defn:topology} For $\F, \G\in D^b_c(\ov{X}; \Z_\ell)$, we {\it topologize} $\rm{Hom}_{\Z_\ell}(\F, \G)$ via the $\ell$-adic topology. \smallskip

For $\F, \G\in D^b_c(\ov{X}; \Q_\ell)$ with lattices $\F\simeq \F_0[\frac{1}{\ell}]$, $\G\simeq \G_0[\frac{1}{\ell}]$ with $\F_0, \G_0\in D^b_c(\ov{X};\Z_\ell)$, we {\it topologize} 
\[
\rm{Hom}_{\Q_\ell}(\F, \G)=\rm{Hom}_{\Z_\ell}(\F_0, \G_0)\left[\frac{1}{\ell}\right]\simeq \colim_{\times \ell} \rm{Hom}_{\Z_\ell}(\F_0, \G_0)
\]
via the colimit topology. 

Finally, for $\F \in D^b_c(\ov{X};\Q_\ell)$ (resp. $\F \in D^b_c(\ov{X}; \Z_\ell)$), we topologize 
\[
\rm{Aut}(\F) \subset \rm{End}(\F)
\]
via the subspace topology. 
\end{defn}

\begin{rmk} It is straightforward to check that, for $\F, \G\in D^b_c(X; \Q_\ell)$, the topology on
\[
\rm{End}_{\Q_\ell}(\F, \G)
\]
is independent of a choice of lattices $\F_0$ and $\G_0$. 
\end{rmk}

\begin{cor}\label{cor:continuous-action} Let $\F\in D^b_c(\eta; \Q_\ell)$. Then the homomorphism
\[
\rho\colon G_\eta \to \rm{Aut}_{\Q_\ell}(\pi_s^*\F)
\]
is continuous.
\end{cor}
\begin{proof}
By definition, it suffices to show that the composition
    \[
    G_\eta \to \rm{Aut}_{\Q_\ell}(\pi_{s}^*\F) \to \rm{End}_{\Q_\ell}(\pi^*_s\F)
    \]
is continuous. Now note that since $D^b_{c}(\ov{\eta}; \Q_\ell)=D^b_{\rm{coh}}(\Q_\ell)$ the bounded derived category of finite dimensional (equivalently, coherent) $\Q_\ell$-vector spaces, so there are no higher Ext groups. Therefore, 
\[
\rm{End}_{\Q_\ell}(\pi^*_s\F) \simeq \bigoplus_{i\in \bf{Z}} \rm{End}_{\Q_\ell}\left(\pi^*_s\left(\cal{H}^i\left(\F\right)\right)\right).
\]
So it suffices to show $\rho$ is continuous for $\F\in D^b_c(\eta; \Q_\ell)^{\heartsuit}$ with respect to the constructible $t$-structure on $D^b_c(\eta; \Q_\ell)$. \smallskip

To prove this, we choose a lattice $\F\simeq \F'[\frac{1}{\ell}]$ with $\F'\in D^b_c(\eta; \Z_\ell)^{\heartsuit}$ and $\ell$-torsionfree. In particular,  $\F'\otimes_{\Z_\ell}^L \Z/\ell^n\Z$ lies in $D^b_c(\eta; \Z/\ell^n\Z)^{\heartsuit}=\mathrm{Shv}_{c}(\eta; \Z/\ell^n\Z)$ for any integer $n\geq 0$. By definition of the topology on $\rm{End}_{\Q_{\ell}}(\pi^*_s\F)$, it suffices to show that the natural map
\[
G_\eta\to \rm{End}_{\Z_\ell}(\pi^*_s\F')
\]
is continuous. Then we write $\F'=\lim_n \F_n$ with $\F_n\in \rm{Shv}_c(\eta; \Z/\ell^n\Z)$. Lemma~\ref{lemma:different-topologies} ensures that it suffices to show that each map
\[
G_\eta \to \rm{End}_{\Z/\ell^n\Z}(\pi^*_s\F_n)
\]
is continuous for every $n\geq 0$. In this case, it suffices to show that there is an open subgroup of $G_\eta$ that acts trivially on $\pi^*_s\F_n$. \smallskip

Now we identify $\rm{Shv}_c(\eta; \Z/\ell^n\Z)$ with $\rm{Mod}^{\rm{disc}, \rm{coh}}_{\Z/\ell^n\Z[G_\eta]}$ the category of finite (equivalently, coherent) discrete $\Z/\ell^n\Z$-modules with a continuous $\Z/\ell^n\Z$-linear action of $G_\eta$. Say $\F$ corresponds to 
\[
V\in \rm{Mod}^{\rm{disc}, \rm{coh}}_{\Z/\ell^n\Z[G_\eta]}.
\]
Continuity of the action implies that the stabilizer of each point is an open subgroup. Thus, the finiteness assumption implies that there is an open subgroup $U\subset G$ that acts trivially on $V$ finishing the proof.
\end{proof}

\subsection{Inertia action: general case}

In what follows, we fix a non-archimedean arithmetic field  $K$ of residue characteristic $p$ and a prime number $\ell\neq p$. \smallskip

The main goal of this section is to prove an analogue of Corollary~\ref{cor:continuous-action} for an arbitrary finite type $k$-scheme $X$. Our argument will be somewhat indirect: we first show that the representation $\rho$ is quasi-unipotent, and then we deduce that $\rho$ is continuous. \smallskip

For this, we will need the structure theory for the Galois groups of an arithmetic field $K$. To see this, we note the definition of an arithmetic field $K$ there is a local field $L$ and an isomorphism $G_K\simeq G_L$ compatible with an isomorphism of residue field $k\simeq l$. In particular, it also induces an isomorphism of inertia subgroups $I_K\simeq I_L$. Therefore, it suffices to understand the Galois and inertia groups of a local field $K$. \smallskip

In what follows, we denote by $P \subset I$ the group of wild inertia. The structure of a Galois group of a local field is well-known: there is a short exact sequence
\[
0 \to P\to I \xr{t} \prod_{p'\neq p}\Z_{p'}(1)\to 0
\]
such that $P$ is pro-$p$ group. We denote by $t_\ell \colon I\to \Z_\ell(1)$ the composition of $t$ with the projection onto the $\ell$-factor. We also denote by $P_\ell$ the kernel of $t_\ell$. We recall that, for any $g\in G_\eta$ and $h\in I$, 
\[
t_\ell(ghg^{-1})=\chi_\ell(g) t_\ell(h),
\]
where $\chi_\ell\colon G_\eta \to \Z_\ell^\times$ is the cyclotomic character of $G_\eta$. 

\begin{rmk}\label{rmk:log-cont} Let $K$ be a local field. Then we note that, from the Galois-theoretic point of view, the morphism $t_\ell \colon I \to \Z_\ell(1)$ is a morphism $\rm{Gal}(K^{\rm{sep}}/K_{\rm{nr}}) \to \rm{Gal}(K_\ell/K_{\rm{nr}})$, where $K_{\rm{nr}}$ is the maximal unramified extension of $K$, and $K_\ell$ is the (pro)-Kummer extension $K_\ell=\cup K_{\rm{nr}}(\pi^{1/\ell^n})$ for a choice of a uniformizer $\pi\in K$. In particular, the target (even as an abelian group) of $t_\ell$ is canonically isomorphic to $\Z_\ell(1)=\lim_n \mu_{\ell^n}(\ov{K})$ and not to $\Z_\ell$. Of course, these groups are isomorphic after a choice of a compatible sequence of primitive $\ell$-power roots of unity $\zeta_{\ell^n}$, but we do not want to fix this choice. 
\end{rmk}

For the next definition, we fix a finite type $k$-scheme $X$ and $\F\in D^b_c(X\times_s \eta; \Q_\ell)$
.
\begin{defn} A subgroup $I_1\subset I$ {\it acts unipotenly} under an action $\rho\colon I\to \rm{Aut}_{\Q_\ell}(\pi^*_X\F)$ if there is an integer $N$ such that $\big(1-\rho(g)\big)^N=0$ for every $g\in I_1$.\smallskip

An action $\rho\colon I\to \rm{Aut}_{\Q_\ell}(\pi^*_X\F)$ is {\it quasi-unipotent} if there is an open subgroup $I_1\subset I$ that acts unipotently. \smallskip

An action $\rho\colon I\to \rm{Aut}_{\Q_\ell}(\pi^*_X\F)$  is {\it strongly quasi-unipotent} if it is quasi-unipotent and $\rho(P_\ell)$ is finite. 
\end{defn}

\begin{thm}\label{thm:grothendieck}(\cite[Proposition on p.515]{Serre-Tate}, Grothendieck) Let $K$ be an arithmetic non-archimedean field, and $\rho\colon G_\eta \to \rm{GL}(V)$ a continuous representation of $G_\eta$ on a finite dimensional $\Q_\ell$-vector space $V$. Then $\rho(P_\ell)$ is a finite group, and there is an open subgroup $I_1\subset I$ that acts unipotently.
\end{thm}

\begin{cor}\label{cor:quasi-unipotent} Let $K$ be an arithmetic non-archimedean field, let $X$ be a finite type $k$-scheme, let $\F\in D^b_c(X\times_s \eta; \Q_\ell)$, and let $\rho\colon I \to \rm{Aut}_{\Q_\ell}(\pi_X^*\F)$ the corresponding action of the inertia group. Then $\rho$ is quasi-unipotent.
\end{cor}
\begin{proof}
    {\it Step~$0$. Reduce to $\F\in D^b_c(X\times_s \eta; \Q_\ell)^{\heartsuit}$ for the ``standard'' $t$-structure from Corollary~\ref{cor:constructible-t-structure-rationally}.} Note that there are only finite number of ``cohomology sheaves'' $\cal{H}^i(\F)$ and $\pi_X^*$ is $t$-exact. Therefore, we conclude that an element $g\in I$ acts unipotently on $\pi_X^*\F$ if and only if it acts unipotently on each $\cal{H}^i(\pi_X^*\F)=\pi_X^*(\cal{H}^i(\F))$ (probably with a different exponent). Therefore, it suffices to show the claim for $\F\in D^b_c(X\times_s\eta; \Q_\ell)^{\heartsuit}$.\smallskip

    {\it Step~$1$. $X=\Spec k$.} In this case, $D^b_c(X\times_s \eta; \Q_\ell)^{\heartsuit}=D^b_c(\eta; \Q_\ell)^{\heartsuit}$ is the category of constructible \'etale $\Q_\ell$-sheaves on $\Spec K$. In this case, it suffices to show that the action of $I$ on $\F|_{\Spec\ov{K}}$ is quasi-unipotent. This action extends to a continuous action of $G_\eta$ by Corollary~\ref{cor:continuous-action}, and so the result follows from Theorem~\ref{thm:grothendieck}. \smallskip
    
    {\it Step~$2$. $X=\Spec k'$ for a finite extension $k\subset k'$.} Consider the morphism $f\colon \Spec k'\to \Spec k$. After passing to the algebraic closure, $X_{\ov{s}}$ becomes a disjoint union of finite copies of $\Spec k(\ov{s})$. Thus, an endomorphism of $\pi_X^*\F$ is zero if and only if it is zero on $f_{\ov{s}, *}\pi_X^*\F$. Therefore, Lemma~\ref{lemma:pushforward-pullback} (and Remark~\ref{rmk:adic-good}) ensures that it suffices to prove the claim for $X'=\Spec k$ and $\F'=(f\times_s\eta)_*\F$  that is already covered by Step~$1$. \smallskip
    
    {\it Step~$3$. $X$ is smooth and $\pi^*_X\F \in D^b_c(X_{\ov{s}}; \Q_\ell)^{\heartsuit}$ is lisse.} For each connected component $\{X_i\}_{i=1}^n$ of $X_{\ov{s}}$, we pick a closed point $\ov{x_i}\in X_i$ and a closed point $x_i\in X$ such that $\Spec k(x_i)\times_{k} \Spec k(\ov{s})$ contains $\ov{x_i}$. \smallskip
    
    Now we use the identification of the category of lisse $\Q_\ell$-sheaves on $X_i$ with the category of continuous $\pi_{1}(X_i, \ov{x}_i)$-representations (see \cite[Proposition 10.1.23]{Lei-Fu}) to conclude that an endomorphism of $\pi_X^*\F$ is zero if and only if it is zero on stalks at each $\ov{x_i}$. Therefore, we can replace $X$ with $X'=\sqcup_{i=1}^n \Spec k(x_i)$ and $\F$ with its pullback onto $X'\times_s \eta$. In this case, the result follows from Step~$2$. \smallskip
    
    {\it Step~$4$. General case.} Suppose $X=\sqcup_{i\in I} X_i$ is a finite stratification of $X$, then an automorphism of $\pi_X^*\F$ is unipotent if and only if it is unipotent on each $\pi^*_{X_i}\left(\F|_{(X_i\times_s \eta)}\right) \simeq \left(\pi^*_X\F\right)|_{X_{i, \ov{s}}}$. \smallskip
    
    Now we note that Lemma~\ref{lemma:lisse-stratification} we can find a stratification of $X=\sqcup_{i=1}^n X_i$ such that each $X_{i, \rm{red}}$ is smooth\footnote{Here we use that the residue field is perfect, so smoothness of $X_{i, \ov{k}, \rm{red}}$ implies smoothness of $X_{i, \rm{red}}$} and $\pi^*_X\F|_{X_{i,\ov{s}}}$ is lisse. Therefore, we can replace $X$ with each $X_{i, \rm{red}}$ to assume that  $\pi_X^*\F$ is lisse. Then the result follows from Step~$3$. 
\end{proof}

\begin{rmk}\label{rmk:normal-subgroup} We can always find an open subgroup $I_1\subset I$ such that $\rho|_{I_1}$ is unipotent and $I_1$ is a normal subgroup of $G$. Indeed, pick any open subgroup $I_1\subset I$ such that $\rho|_{I_1}$ is unipotent. Then there is an open subset $U_1\subset G$ such that $U_1\cap I = I_1$. Since $G$ is profinite, we can find an open normal subgroup $G'_1\subset G_1$ such that $G'_1\subset U_1$ (see \cite[Proposition (1.1.3)]{cohomology-of-number-fields}). Then $I'_1\coloneqq G'_1\cap I$ is normal in $G$ and $\rho|_{I'_1}$ is unipotent.
\end{rmk}

\begin{lemma}\label{lemma:no-homs} Let $G$ be a pro-(prime-to-$\ell$) group, and let $M$ be a finite $\ell^\infty$-torsion group. Then there are no non-trivial homomorphisms $G \to M$.
\end{lemma}
For the applications later in this paper, it is important that we do not make any continuity assumptions in the formulation of Lemma~\ref{lemma:no-homs}.
\begin{proof}
    Since $M$ is a finite group, there is an integer $N$ such that $m^{\ell^N}=e$ for any $m\in M$. Therefore, it suffices to show that the $\ell^N$-power map $(-)^{\ell^N}\colon G\to G$ is bijective\footnote{This map is not a group homomorphisms if $G$ is not abelian.}.\smallskip
    
    Let $U_i$ be a basis of open normal finite index subgroups in $G$, so $G=\lim G/U_i$ and the order of $|G/U_i|$ is co-prime to $\ell$ by our assumption on $G$. \smallskip
    
    {\it Step~$1$. $(-)^{\ell^N}\colon G/U_i\to G/U_i$ is bijective for each $U_i$.} Since $G/U_i$ is a finite group, it suffices to show that the map is surjective. Pick an element $x\in G/U_i$. Since $G/U_i$ is finite, there is an integer $d$ such that $x^d=e$. Since the order of $G/U_i$ is coprime with $\ell$, $d$ is also coprime with $\ell$. Therefore, there are integers $a$ and $b$ such that $da+\ell^Nb=1$. Therefore, we conclude that
    \[
    (x^b)^{\ell^N}=x^{\ell^Nb}=x^{da+\ell^nb}=x.
    \]
    
    {\it Step~$2$. $(-)^{\ell^N}\colon G\to G$ is bijective.} Now let $x\in G$ be an element, we denote by $x_i$ its image in $G/U_i$. And let $y_i$ be a unique element in $G/U_i$ such that $y_i^{\ell^N}=x_i$, its existence follows from Step~$1$. By uniquness, if $U_i\subset U_j$, the the image of $y_i$ under the natural projection map $G/U_i\to G/U_j$ is equal to $y_j$. Then the sequence $\{y_i\}_{i\in I}$ defines an element $y\in G$ such that $y^{\ell^N}=x$.
\end{proof}

\begin{cor}\label{cor:p1-doesnot-act} Let $K$ be an arithmetic non-archimedean field, let $X$ be a finite type $k$-scheme, and let $\F\in D^b_c(X\times_s \eta; \Q_\ell)$. Then $\rho\colon I \to \rm{Aut}_{\Q_\ell}(\pi_X^*\F)$ is strongly quasi-unipotent. More precisely, $\rho(P_{1, \ell})=\{\rm{Id}\}$ where $P_{1, \ell}\coloneqq I_1\cap P_{\ell}$ and $I_1\subset I$ is an(y) open subgroup such that $\rho|_{I_1}$ is unipotent.
\end{cor}
\begin{proof}
    Corollary~\ref{cor:quasi-unipotent} ensures that the action of $\rho$ is quasi-unipotent. Similarly to the proof of Corollary~\ref{cor:continuous-action}, we can reduce to the case $\F\in D^b_c(X\times_s\eta; \Z/\ell^n\Z)$ for some integer $n$. Then, for any $g\in P_{1,\ell}$, we already know that $\rho_g$ is unipotent. So we can write
    \[
    \rho_g=1+\phi
    \]
    for some nilpotent $\phi\in \rm{End}_{\Z/\ell^n\Z}(\pi_X^* \F)$. Therefore, we conclude that
    \[
    (\rho_g)^{\ell^m}=\rho_g^{\ell^m}=(1+\phi)^{\ell^m}=1
    \]
    for large enough $m$. Now we use that $\rm{End}_{\Z/\ell^n\Z}(\pi_X^*\F)$ is a finite $\Z/\ell^n\Z$-module to conclude that $\rho(P_{1,\ell})$ is a finite $\ell^\infty$-torsion group. However, $P_{1,\ell}$ is pro-(prime-to-$\ell$) group, so there are no non-trivial (possibly not continuous) homomorphisms to a finite $\ell^\infty$-torsion group by Lemma~\ref{lemma:no-homs}. Therefore, $\rho(P_{1,\ell})$ must be trivial. 
\end{proof}

\begin{cor}\label{cor:rho-continuous} Under the assumption of Corollary~\ref{cor:quasi-unipotent}, $\rho$ is continuous.
\end{cor}
\begin{rmk} It is probably true that $\rho$ is continuous without any assumption on $K$. However, our proof uses the structure theory for the Galois group of an arithmetic field that is not available without this assumption. 
\end{rmk}
\begin{proof}[Proof of Corollary~\ref{cor:rho-continuous}]
    Let $\F=\lim_n \F_n$ with $\F_n\in D^b_c(X\times_s \eta; \Z/\ell^n\Z)$. Similarly to the proof of Corollary~\ref{cor:continuous-action}, it suffices to prove the claim for each $\F_n$ separately. Thus, we only need to show that there is an open subgroup of $I$ that acts trivially on $\F_n$. Corollary~\ref{cor:quasi-unipotent} ensures that there is an open subgroup $I_1\subset I$ that acts unipotently. Then Corollary~\ref{cor:p1-doesnot-act} implies that $P_{1, \ell}=I_1 \cap P_{\ell}$ lies in the kernel of $\rho$. Thus the action of $\rho$ on $I_1$ factors through the quotient
    \[
    I_1/P_{1, \ell} \subset \Z_\ell(1).
    \]
    Since $I_1/P_{1, \ell} \hookrightarrow \Z_\ell(1)$ is a continuous morphism of profinite groups, we conclude that $I_1/P_{1, \ell}$ is a non-trivial closed subgroup of $\Z_\ell(1)\cong \Z_\ell$. Therefore, \cite[Proposition 2.7.1(a,b)]{profinite} implies that it must be (non-canonically) isomorphic to $\Z_\ell$. Thus, it suffices to show that a finite index subgroup of the quotient $I_1/P_{1, \ell}$ acts trivially on $\pi_X^*\F_n$. This follows formally from the fact that $\rm{End}_{\Z/\ell^n\Z}(\pi^*_X\F_n)$ is a finite group.
\end{proof}

\subsection{Mixed sheaves on Deligne's topos}

In what follows, we fix a non-archimedean arithmetic field  $K$ of residue characteristic $p$ and a prime number $\ell\neq p$. \smallskip

The main goal of this section is to define the notion of a mixed object in $D^b_c(X\times_s \eta; \Q_\ell)$ for an arithmetic field $K$ and a finite type $k$-scheme $X$. The main issue is that the standard notion of a mixed sheaf is  only defined for (complexes of) sheaves on a finite type $k$-scheme, but there is no canonical functor  $D^b_c(X\times_s \eta; \Q_\ell)\to D^b_c(X; \Q_\ell)$. So it is not entirely formal to extend the standard definition to $D^b_c(X\times_s \eta; \Q_\ell)$. \smallskip

A way to overcome this issue is to observe that the Galois group of the residue field $G_s$ is isomorphic to a free pro-finite group $\wdh{\Z}$ (recall that the residue field $k$ is finite by Remark~\ref{rmk:finite-residue-field}). So the continuous surjection
\[
r\colon G_\eta \to G_s
\]
has plenty of continuous sections $\sigma\colon G_s\to G_\eta$, and each such $\sigma$ defines a functor
\[
\sigma^*_X\colon D^b_c(X\times_s \eta; \Q_\ell) \to D^b_c(X; \Q_\ell)
\]
by Construction~\ref{construction:sigma} and Remark~\ref{rmk:adic-good}. Then a natural way to define mixedness is to require a complex $\F\in D^b_{c}(X\times_s \eta; \Q_\ell)$ to be mixed after applying the pullback functor $\sigma^*_X$ for some continuous section $\sigma\colon G_s \to G_\eta$. The main content of this section is to show that this definition is independent of a choice of $\sigma$.

\begin{rmk} It will be crucial for our arguments in the paper to know that mixedness in $D^b_c(X\times_s \eta; \Q_\ell)$ can be checked after a finite extension $K\subset K'$. For this argument, it is crucial to know that the notion mixedness is independent of a choice of $\sigma$.
\end{rmk}

In the next proposition, we are going to use the notion of pure and mixed objects in $D^b_c(X; \Q_\ell)$ for a finite type $k$-scheme $X$. We refer to \cite[Section II.12]{KW} for an extensive discussion of this notion. We also refer to \cite[Section I.2]{KW} for the notion of punctually pure and mixed objects in $\rm{Shv}_c(X; \Q_\ell)$. The proof of the proposition below will use that punctual purity/mixedness can be defined for a more general notion of Weil sheaves (see \cite[Convention on p. 8]{KW}), we refer to \cite[Section I.1]{KW} for an extensive discussion of this notion.

\begin{prop}\label{prop:mixed-independent} Let $K$ be an arithmetic non-archimedean field, let $X$ be a finite type $k$-scheme, let $\sigma, \sigma'\colon G_s \to G_\eta$ be two continuous sections, and let $\F\in D^b_{c}(X\times_s \eta; \Q_\ell)$. Then $\sigma_X^*\F\in D^b_c(X; \Q_\ell)$ is mixed of weights $\leq w$ (in the sense of \cite[Definition III.12.3]{FK}) if and only if ${\sigma'}_X^*\F\in D^b_c(X; \Q_\ell)$ is mixed of weights $\leq w$.
\end{prop}
\begin{proof}
Mixedness of weights $\leq w$ is the condition on cohomology sheaves, so we can assume that $\F$ lies in the heart $D^b_c(X\times_s \eta; \Q_\ell)^{\heartsuit}$ of the standard $t$-structure on $D^b_c(X\times_s \eta; \Q_\ell)$ (see Corollary~\ref{cor:constructible-t-structure-rationally}). Then $\sigma_X^*\F$ is mixed of weights $\leq w$ is equivalent to $\sigma_X^*\F$ being (punctually) mixed sheaf on $X$ of weights $\leq w$. Then Lemma~\ref{lemma:lisse-stratification} implies that there is a stratification $X=\sqcup_{i=1}^n X_i$ such that $\pi_X^*\F|_{X_{i,\ov{s}}}$ is lisse, mixed of weights $\leq w$, and $X_{i, \rm{red}}$ is smooth for each $i$. Since mixed sheaves of weights $\leq w$ are preserved by extensions and $(j\times_s \eta)_!$ (see Definition~\ref{defn:functors-rationally}) preserves mixed complexes of weight $\leq w$ for a locally closed immersion $j\colon Z\to X$, we can replace $X$ with $X_{i, \rm{red}}$ to assume that $\pi_X^*\F$ is lisse and $X$ is smooth. \smallskip

In this case, \cite[Variante (3.4.9)]{Weil2} implies that there is a functorial (essentially finite) increasing filtration 
\[
\rm{Fil}^n_{\rm{W}}\F \subset \sigma^*_X\F
\]
by lisse $\Q_\ell$-sheaves such that $\rm{gr}_{\rm{W}}^n\sigma_X^*\F$ is pure of weight $n$ (or zero). \smallskip

Let us denote $\sigma(F)=\Phi$ and $\sigma'(F)=\Phi'$ for a geometric Frobenius $F\in G_s$, and $b\colon X_{\ov{s}} \to X_{s}$ the natural morphism of schemes. We wish to show that ${\sigma'_X}^*\F$ is also (punctually) mixed sheaf on $X$ of the same weights as $\sigma_X^*\F$.\smallskip

First, we note that $b^*(\sigma^*_X \F)\simeq \pi^*_X\F\simeq b^*({\sigma'_X}^*\F)$. So we can think of $b^*\rm{Fil}^n_{\rm{W}}\F$ as subsheaves of $\pi^*_X\F\simeq b^*({\sigma'_X}^*\F)$. Secondly, we note that the notion of a (punctually) mixed $\Q_\ell$-sheaf on $X$ depends only on the underlying Weil sheaf. Thus, ${\sigma'_X}^*\F$ is (punctually) mixed of weights $\{w_i\}_{i\in I}$ if and only if the Weil sheaf $(\pi_X^*\F, \rho(\Phi')\colon F^*\pi_X^*\F\to \pi_X^*\F)$ is punctually mixed of weights $\{w_i\}_{i\in I}$. And, by assumption, the Weil sheaf $(\pi_X^*\F, \rho(\Phi)\colon F^*\pi_X^*\F\to \pi_X^*\F)$ is a (punctually) mixed Weil sheaf of weights $\{w_i\}_{i\in I}$ with $w_i\leq w$. \smallskip

{\it Claim~$1$. $\rho(\Phi')$ induces an isomorphism $b^*\rm{Fil}^n_{\rm{W}}\F \to b^*\rm{Fil}^n_{\rm{W}}\F$ for each integer $n$.} It suffices to check on closed points of $X$, so we can assume that $X$ is a point. Then arguing similarly to the Step~$2$ (and using that $f_{*}$ preserves local systems of weight $d$ for a finite \'etale $f$) in the proof of Corollary~\ref{cor:quasi-unipotent}, we can assume that $X=\Spec k$. Then $D^b_c(X\times_s\eta; \Q_\ell)\simeq D^b_c(\eta; \Q_\ell)$, and so we can assume that $K$ is a local field since $D^b_c(\eta; \Q_\ell)$ depends only on the Galois group $G_K$. Thus the result follows from \cite[Proposition-definition (1.7.5)]{Weil2}. \smallskip

{\it Claim~$2$. Weil sheaves $(b^*\rm{Fil}^n_{\rm{W}}\F, \rho(\Phi')\colon F^*b^*\rm{gr}_{\rm{W}}^n \F \to b^*\rm{gr}_{\rm{W}}^n \F)$ are pure of weight $n$.} Again, the same reduction as in the proof of Claim~$1$ reduces the question to the case $X=\Spec k$ is the base point. Then the claim follows from \cite[Lemme (1.7.4)]{Weil2}.\smallskip

Now claims $1$ and $2$ together imply that the Weil sheaf
\[
(\pi_X^*\F=b^*\sigma_X^*\F, \rho(\Phi')\colon F^*\pi_X^*\F\to \pi_X^*\F)
\]
admits an essentially finite filtration by Weil $\Q_\ell$-sheaves 
\[
(b^*\rm{Fil}^n_{\rm{W}}\F, \rho(\Phi')\colon F^*b^*\rm{Fil}^n_{\rm{W}}\F \to b^*\rm{Fil}^n_{\rm{W}}\F)\subset (b^*{\sigma'}_X^*\F, \rho(\Phi')\colon F^*\pi_X^*\F\to \pi_X^*\F)
\]
such that each quotient is a pure Weil sheaf of weight $n$. Since $(b^*\rm{gr}^n_{\rm{W}}\F, \rho(\Phi')\colon F^*b^*\rm{gr}_{\rm{W}}^n \F \to b^*\rm{gr}_{\rm{W}}^n \F)$ is a zero Weil sheaf if and only if $(b^*\rm{gr}_{\rm{W}}^n \F, \rho(\Phi)\colon F^*b^*\rm{gr}_{\rm{W}}^n \F \to b^*\rm{gr}_{\rm{W}}^n \F)$ is a zero Weil sheaf, we conclude that $(b^*{\sigma'}_X^*\F, \rho(\Phi')\colon F^*\pi_X^*\F\to \pi_X^*\F)$ is a mixed Weil sheaf and its weights coincide with the weights of $(b^*{\sigma}_X^*\F, \rho(\Phi)\colon F^*\pi_X^*\F\to \pi_X^*\F)$
\end{proof}


\begin{cor}\label{cor:pure-independent} Let $K$ be an arithmetic non-archimedean field, let $X$ be a finite type $k$-scheme, let $\sigma, \sigma'\colon G_s \to G_\eta$ be two continuous sections, and let $\F\in D^b_{c}(X\times_s \eta; \Q_\ell)$. Then $\sigma_X^*\F\in D^b_c(X; \Q_\ell)$ is pure of weight $w$ (resp. mixed of weights $\geq w$) if and only if ${\sigma'}_X^*\F\in D^b_c(X; \Q_\ell)$ is pure of weight $w$ (resp. mixed of weights $\geq w$).
\end{cor}
\begin{proof}
    Note that, for every continuous section $\sigma\colon G_s \to G_\eta$, we have an isomorphism 
    \[
    \sigma^*_X(\bf{D}_{X\times_s\eta}(\F)) \simeq \bf{D}_{X}(\sigma^*_X \F)
    \]
    by Remark~\ref{rmk:verdier-duality-compatible} and Remark~\ref{rmk:rational-good}. Therefore, the claim for mixed complexes of weight $\geq w$ follows from Proposition~\ref{prop:mixed-independent} applied to $\bf{D}_{X\times_s \eta}(\F)$. The claim for pure complexes of weight $w$ follows from Proposition~\ref{prop:mixed-independent} applied to both $\F$ and $\bf{D}_{X\times_s \eta}(\F)$.
\end{proof}

\begin{defn}\label{defn:pure-mixed-Deligne} An object $\F\in D^b_c(X\times_s \eta; \Q_\ell)$ is {\it mixed of weights $\leq w$} (resp. {\it mixed of weights $\geq w$}, resp. {\it pure of weight $w$}) if $\sigma_X^*\F\in D^b_c(X; \Q_\ell)$ is mixed of weights $\leq w$ (resp. mixed of weights $\geq w$, resp. pure of weight $w$) for a(ny) choice of a continuous splitting $\sigma\colon G_s \to G_\eta$.
\end{defn}


Now we discuss that mixedness (resp. purity) of a complex in $D^b_c(X\times_s \eta; \Q_\ell)$ can be checked after a (possibly non-finite) extension of arithmetic fields. 

\begin{notation}\label{notation:field-extension}  For an extension of non-archimedean arithmetic fields $K\subset K'$, we denote by $k'$ residue field of $K'$, $\eta'$ the classifying topos of the pro-finite group $G_{K'}$, and $s'$ the classifying topos of the pro-finite group $G_{k'}$. The diagram
\[
\begin{tikzcd}
\eta' \arrow{d}\arrow{r} & s'\arrow{d} &\arrow{l} X_{s', \et} \arrow{d}\\
\eta \arrow{r}& s & \arrow{l} X_{s,\et}
\end{tikzcd}
\]
commutes up to an equivalence. So we have a natural morphism of Deligne's topoi 
\[
b\colon X_{s'}\times_{s'}\eta'\to X \times_s \eta.
\]
\end{notation}

\begin{lemma}\label{lemma:mixed-after-base-change} Let $K\subset K'$ be a (possibly non-algebraic) extension of arithmetic fields, let $X$ be a finite type $k$-scheme, and let $\F\in D^b_{c}(X\times_s \eta; \Q_\ell)$. Then $\F$ is mixed of weights $\leq w$ (resp. mixed of weights $\geq w$, resp. pure of weight $w$) if and only if $b^*\F$ is mixed of weights $\leq w$ (resp. mixed of weights $\geq w$, resp. pure of weight $w$).
\end{lemma}
\begin{proof}
    Since the residue fields $k$ and $k'$ are finite, the extension $k\subset k'$ is also finite. Therefore, we can find an unramified sub-extension $K\subset K''  \subset K'$ such that $k''=k'$. Lemma~\ref{lemma:finite-extensions} implies that $K''$ is arithmetic, so we treat the case of an unramified and ``totally ramified" extensions separately. 
    
    Proposition~\ref{prop:mixed-independent} ensures that we can check mixedness with respect to any continuous section $\sigma$. We will crucially use this property in the proof. \smallskip
    
    {\it Case $1$. The extension $K\subset K'$ is finite unramified} Since $K\subset K'$ is unramified, we have a Cartesian square
    \[
    \begin{tikzcd}
    G_{K'}\arrow{d}\arrow{r}{r'} & G_{k'}\arrow{d}\\
    G_K \arrow{r}{r} & G_k.
    \end{tikzcd}
    \]
    So the universal property of a pullback diagram implies that a continuous section $\sigma\colon G_k \to G_K$ defines a continuous section $\sigma'\colon G_{k'}\to G_K$. We consider the natural morphism of \'etale topos $q\colon X_{s', \et} \to X_\et$ to see that we have a canonical isomorphism
    \[
    q^*\sigma^*_X\F \simeq (\sigma'_{X_{s'}})^{*}b^*\F
    \]
    for any $\F\in D^b_c(X\times_s \eta; \Q_\ell)$. Proposition~\ref{prop:mixed-independent} (resp. Corollary~\ref{cor:pure-independent}) ensures that $b^*\F$ is mixed of weights $\leq w$ (resp. mixed of weights $\geq w$, resp. pure of weight $w$) if and only if so is $(\sigma'_{X_{s'}})^{*}b^*\F\simeq q^*\sigma^*_X \F$. Therefore, it suffices to show that  a complex of sheaves is mixed (resp. pure) of prescribed weights if and only if the same holds after a finite extension of the ground field. This is standard and can be deduced from \cite[Permanence Property (3) on p.14]{KW} by a standard argument. \smallskip
    
    {\it Case~$2$. $K\subset K'$ is ``totally ramified'', i.e. induces an isomorphism on residue field $k\simeq k'$.} In this case, we have a commutative diagram
    \[
    \begin{tikzcd}
    G_{K'}\arrow{d}{u}\arrow{rd}{r'} &\\
    G_K \arrow{r}{r} & G_k
    \end{tikzcd}
    \]
    with surjective $r$ and $r'$. Therefore, we can choose a continuous section $\sigma'\colon G_k \to G_{K'}$. So the composition 
    \[
    \sigma\coloneqq u\circ \sigma' \colon G_k \to G_K
    \]
    is a section of $r$. Then we see that
    \[
    {\sigma'}^*_{X} b^* \F \simeq \sigma_X^*\F.
    \]
    Thus the result follows from Proposition~\ref{prop:mixed-independent} and Corollary~\ref{cor:pure-independent}. 
\end{proof}

\subsection{Monodromy operator}

For the rest of the section, we fix a non-archimedean arithmetic field $K$ (unless it is specified otherwise) of residue characteristic $p>0$ and a prime number $\ell\neq p$. \smallskip

The main goal of this section is to show that any complex $\F\in D^b_c(X\times_s \eta; \Q_\ell)$ comes equipped with a {\it monodromy operator} 
\[
N\colon \pi^*_X \F \to \pi^*_X\F(-1).
\]
We will construct this operator by adapting  Grothendieck's original construction of the monodromy operator on cohomology groups of a variety over a $p$-adic field, using quasi-unipotence of the inertia action established in Corollary~\ref{cor:quasi-unipotent} and Corollary~\ref{cor:p1-doesnot-act}. \smallskip

Before we start discussing the construction of the monodromy operator, we recall the construction of the exponent and logarithm morphisms. In what follows, we fix a finite type $k$-scheme $X$ and a complex $\F\in D^b_{c}(X\times_s \eta; \Q_\ell)$. \smallskip

For any unipotent operator $1+\varphi\in \rm{Aut}_{\Q_\ell}\left(\pi_X^*\F\right)$, we define {\it logarithm}
\[
\log (1+\varphi) = \sum_{k=1}^\infty (-1)^{k+1} \frac{\varphi^k}{k} \in \rm{End}_{\Q_\ell}\left(\pi_X^*\F\right)
\]
that is easily seen to be a nilpotent endomorphism of $\pi_X^*\F$. Likewise, for a nilpotent operator $\psi\in \rm{End}_{\Q_\ell}\left(\pi_X^*\F\right)$, we  define {\it exponent}
\[
\exp \psi = 1+ \sum_{k=1}^\infty \frac{\psi^k}{k!} \in \rm{Aut}_{\Q_\ell}(\pi_X^*\F)
\]
that is easily seen to be a unipotent automorphism of $\pi_X^*\F$. 

\begin{lemma}\label{lemma:continuous} Let $K$ be a non-archimedean field  and let $X$ be a finite type $k$-scheme. Then 
\[
\rm{exp}\colon \rm{End}_{\Q_\ell}\left(\pi_X^*\F\right)^{\rm{nil}} \to \rm{Aut}_{\Q_\ell}\left(\pi^*\F\right)^{\rm{uni}},
\]
\[
\rm{log} \colon \rm{Aut}_{\Q_\ell}\left(\pi_X^*\F\right)^{\rm{uni}}\to \rm{End}_{\Q_\ell}\left(\pi_X^*\F\right)^{\rm{nil}}
\]
are continuous with respect to the topologies defined in Definition~\ref{defn:topology}, and inverses to each other. 
\end{lemma}
\begin{proof}
    We note that $\rm{End}_{\Q_\ell}\left(\pi_X^*\F\right)$ is a finite $\Q_\ell$-algebra by Lemma~\ref{lemma:different-topologies}. Therefore, there is an integer $N$ such that, for any nilpotent $\varphi\in \rm{End}_{\Q_\ell}\left(\pi_X^*\F\right)$, we have $\varphi^N=0$. Therefore, one easily sees that both $\rm{log}$ and $\rm{exp}$ are polynomials in $\varphi$ and $\psi$ respectively, so they are continuous. \smallskip
    
    Using that all infinite sums in the definition of $\rm{log}$ and $\rm{exp}$ boil down to finite sums, one easily checks that $\exp\big(\log (1+\varphi)\big)=1+\varphi$ and $\log(\exp \psi)=\psi$ for any $\varphi, \psi \in \rm{End}_{\Q_\ell}\left(\pi_X^*\F\right)^{\rm{nil}}$. 
\end{proof}

\begin{lemma}\label{lemma:monodromy-operator} Let $K$ be a non-archimedean arithmetic field, $X$ a finite type $k$-scheme, $\F\in D^b_c(X\times_s \eta; \Q_\ell)$, and $I_1\subset I$ is an(y) open subgroup such that $\rho|_{I_1}$ is unipotent (it exists by Corollary~\ref{cor:quasi-unipotent}). Then there is a unique (independent of $I_1$) nilpotent morphism 
\[
N\colon \pi_X^*\F\to \pi_X^*\F(-1)
\]
in $D^b_c(X_{\ov{s}}; \Q_\ell)$ such that 
\[
\rho_g=\exp(Nt_\ell(g))
\]
for $g\in I_1$. 
\end{lemma}
\begin{proof}
    Firstly, we choose some compatible sequence $\zeta_{\ell^n}\in \ov{K}$ of $\ell$-power roots of unity. It both trivializes the Galois group $\Z_\ell(1)\cong \Z_\ell$ and $\pi^*_X\F(-1)\cong \pi^*_X\F$. So $t_\ell(I_1)\subset \Z_\ell$ is isomorphic to $\ell^n\Z_\ell$ for some integer $n$. We pick $u\in I_1$ such that $t_\ell(u)=\ell^m$ for $m\geq n$. \smallskip
    
    Now we note that uniqueness of $N$ is clear because the formula
    \[
    \rho_u=\exp(Nt_\ell(u))    
    \]
    implies that $N=\frac{\log(\rho_u)}{\ell^m}$. It is also independent of $I_1$ because for two choices $I_1$ and $I'_1$, we can find $u\in I_1\cap I'_1$ such that $t_\ell(u)=\ell^m$ for some large $m$. \smallskip
    
    Now we show existence. We pick $u\in I_1$ such that $t_\ell(u)=\ell^n\in \Z_\ell$. Firstly, we note that the formula above
    \[
    N=\frac{\log(\rho_u)}{\ell^n}
    \]
    is independent of a choice of (compatible) trivializations $\Z_\ell(1)\cong \Z_\ell$ and $\pi^*_X\F(-1) \cong \pi^*_X\F$, and so defines a homomorphism $N\colon \pi^*_X\F\to \pi^*_X\F(-1)$. Now we wish to show that
    \[
    \rho_g=\exp(Nt_\ell(g))
    \]
    for any $g\in I_1$. This formula clearly holds for $g=u^m$ for an integer $m$. Now Corollary~\ref{cor:rho-continuous} and Lemma~\ref{lemma:continuous} imply that  
    \[
    g\mapsto \rho_g\in \rm{Aut}_{\Q_\ell}\left(\pi_X^*\F\right), \text{ and}
    \]
    \[
    g\mapsto \exp\left(Nt_\ell\left(g\right)\right) \in \rm{Aut}_{\Q_\ell}\left(\pi_X^*\F\right)
    \]
    are two continuous homomorphisms $I_1 \to \rm{Aut}_{\Q_\ell}\left(\pi_X^*\F\right)$ that are trivial on $P_{1, \ell}$ and coincide on a dense subgroup 
    \[
    \ell^n\Z \subset I_1/P_{1,\ell}\simeq \ell^n\Z_\ell\subset \Z_\ell\simeq \Z_\ell(1).
    \]
    Therefore, they coincide everywhere. \smallskip
\end{proof}

\begin{rmk} We could have defined $N$ to be a nilpotent morphism $\pi_X^*\F \to \pi_X^*\F$, but then this operator would depend on a choice of a trivialization $\rm{Gal}(K_\ell/K_{\rm{nr}})\simeq \Z_\ell(1)\cong \Z_\ell$.
\end{rmk}

\begin{defn}\label{defn:monodromy-operator} The {\it monodromy operator} of $\F\in D^b_{c}(X\times_s \eta; \Q_\ell)$ is a nilpotent morphism 
\[
    N\colon \pi_X^*\F \to \pi^*_X\F(-1)
\]
constructed in Lemma~\ref{lemma:monodromy-operator}.
\end{defn}

\begin{rmk} In Section~\ref{section:log-pure}, we show that $N$ descends to a morphism $\F\to \F(-1)$ in $D^b_c(X\times_s \eta; \Q_\ell)$ for a perverse $\F$ (see Definition~\ref{defn:perverse}). It should be possible to show that $N\colon  \pi_X^*\F\to \pi^*_X\F(-1)$ descends to an operator $N\colon \F\to \F(-1)$ for any $\F\in D^b_c(X\times_s \eta; \Q_\ell)$. However, it seems a bit difficult to do it rigorously and we will never need this, so we do not discuss this generalization here. 
\end{rmk}

Our next goal is to show that the monodromy operator intervenes with the action of the Galois group via the cyclotomic character $\chi_\ell$. But before we do this, we need to recall how to raise to $\ell$-adic power in $\rm{End}_{\Q_\ell}(\pi_X^*\F)$.

\begin{rmk}\label{rmk:ell-adic-powers} For $g\in I_1$ and $x\in \Z_\ell$, it makes sense to talk about $\rho_g^x\in \rm{End}_{\Q_\ell}\left(\pi_X^*\F\right)$. Indeed, $\rho(I_1)$ with its subspace topology is a pro-$\ell$ group by Corollary~\ref{cor:p1-doesnot-act}. Therefore, the homomorphism
\[
\a_{g}\colon \Z\to \rho(I_1) \subset \rm{End}_{\Q_\ell}\left(\pi_X^*\F\right)
\]
\[
n\mapsto \rho(x)^{n}
\]
is continuous in the $\ell$-adic topology on $\Z$. So it uniquely extends to a continuous homomorphism
\[
\a_{g, \ell}:\Z_\ell \to \rho(I_1)\subset \rm{End}_{\Q_\ell}\left(\pi_X^*\F\right).
\]
We define $\rho(g)^{x}\coloneqq \a_{g, \ell}(x)$. 
\end{rmk}

\begin{cor}\label{cor:commutation} Let $X$ be a finite type $k$-scheme and let $\F\in \cal{D}^b_c(X\times_s \eta; \Q_\ell)$ with the monodromy operator $N\colon \pi_X^*\F\to \pi_X^*\F(-1)$. Then the following diagram
\[
\begin{tikzcd}
\ov{g}^*\pi_X^*\F\arrow{r}{\ov{g}^*\left(N\right)} \arrow{d}{\rho_g} & \ov{g}^*\pi_X^*\F(-1) \arrow{d}{\rho_{g} \otimes \chi^{-1}_\ell(g)} \\
\pi^*_X\F \arrow{r}{N} & \pi_X^*\F(-1)
\end{tikzcd}
\]
commutes for every $g\in G_\eta$.
\end{cor}
\begin{proof}
    The result basically follows from the commutativity relations inside the Galois group $G_\eta$. We spell out a detailed proof for the convenience of the reader. \smallskip
    
    We use Remark~\ref{rmk:normal-subgroup} to get an open {\it normal} subgroup $I_1\subset G$ such that $\rho|_{I_1}$ is unipotent. We pick $u\in I_1$ and $\ell^n$ as in the proof of Lemma~\ref{lemma:monodromy-operator}. \smallskip
    
    {\it Step~$1$. $\rho_{g}^{-1}=\ov{g}^*(\rho_{g^{-1}})$}. Firstly, we note the the formula
    \[
    \rm{Id}=\rho_{1}=\rho_{gg^{-1}}=\rho_g\circ \ov{g}^{*}(\rho_{g^{-1}})
    \]
    implies that $\ov{g}^{*}(\rho_{g^{-1}})$ is a right inverse to $\rho_g$. Likewise, the formula
    \[
    \rm{Id}=\rho_{1}=\rho_{g^{-1}g}=\rho_{g^{-1}}\circ (\ov{g^{-1}})^*(\rho_g)
    \]
    implies that $\rm{Id}=\ov{g}^*(\rho_{g^{-1}})\circ \rho_g$, and so $\ov{g}^*(\rho_{g}^{-1})$ is also a left inverse to $\rho_g$. \smallskip
    
    {\it Step~$2$. $\rho_g\circ \ov{g}^*(\rho_u)\circ\rho_{g}^{-1}=\rho_{gug^{-1}}$}. It follows from the following sequence of equalities:
    \begin{align*}
        \rho_{gug^{-1}}=\rho_g\circ \ov{g}^{*}(\rho_{ug^{-1}}) =\rho_g\circ \ov{g}^*(\rho_u)\circ\ov{g}^*(\rho_{g^{-1}})=\rho_g\circ \ov{g}^*(\rho_u)\circ \rho_{g}^{-1},
    \end{align*}
    where the last equality uses Step~$1$. \smallskip
    
    {\it Step~$3$. $\chi_\ell(g)\log \rho_u = \log \rho_{gug^{-1}}$} Firstly, we note that $gug^{-1}\in I_1$ by normality of $I_1$, so $\log \rho_{gug^{-1}}$ makes sense.  We put $P_{1, \ell} \coloneqq  I_1\cap P_\ell = I_1\cap \ker(t_\ell)$, so Corollary~\ref{cor:p1-doesnot-act} ensures that $\rho|_{I_1}$ factors through $t_\ell$, so we denote by 
    \[
    \ov{\rho}\colon I_1/P_{1, \ell} = t_\ell(I_1) \to \rm{Aut}_{\Q_\ell}(\pi_X^*\F)
    \]
    the unique continuous homomorphism such that $\rho_h=\ov{\rho}\circ t_\ell(h)$ for any $h\in I_1$. Then
    \[
    \rho_{gug^{-1}}=\ov{\rho}(t_\ell(gug^{-1}))=\ov{\rho}(\chi_\ell(g)t_\ell(u))=\ov{\rho}(t_\ell(u))^{\chi_\ell(g)}=\rho(u)^{\chi_\ell(g)},
    \]
    where the third equality uses continuity of $\ov{\rho}$ (that, in turn, comes from continuity of $\rho$ established in Corollary~\ref{cor:rho-continuous}). This formally implies that
    \[
    \log \rho_{gug^{-1}}=\log \rho_u^{\chi_\ell(g)}=\chi_\ell(g)\log \rho_u, 
    \]
    where the last equality comes from the continuity of logarithm (see Lemma~\ref{lemma:continuous}).\smallskip
    
    {\it Step~$4$. $\chi_\ell(g)N=\rho_g\circ \ov{g}^*N\circ \rho_g^{-1}$} The claim follows from a sequence of equalities:
    \begin{align*}
    \rho_g\circ \ov{g}^*N\circ \rho_g^{-1} & = \rho_g\circ \ov{g}^*\left(\frac{\log \rho_u}{\ell^n}\right) \circ \rho_g^{-1} \\
    & = \frac{1}{\ell^n}\rho_{g}\circ \log\left(\ov{g}^*(\rho_u)\right)\circ \rho_g^{-1} \\
    & = \frac{1}{\ell^n}\log\left(\rho_{g}\circ \ov{g}^*(\rho_u)\circ \rho_g^{-1}\right) \\
    & =\frac{1}{\ell^n}\log(\rho_{gug^{-1}})\\
    & =\frac{1}{\ell^n}\chi_\ell(g)\log\rho_u\\
    & =\chi_\ell(g)N.
    \end{align*}
    The first equality holds by the construction of $N$. The second and third equalities are trivial. The fourth equality follows from Step~$2$. The fifth equality follows from Step~$3$. And the last equality follows from the construction of $N$. Therefore, we get that
    \[
    N \circ \rho_g= \chi^{-1}_\ell(g) \rho_g \circ  \ov{g}^*( N)  = (\rho_g \otimes \chi_\ell^{-1}(g))\circ \ov{g}^*N.  \qedhere
    \]
\end{proof}

\subsection{Weight filtration}

For the rest of the section, we fix a non-archimedean arithmetic field $K$ (unless it is specified otherwise) of residue characteristic $p>0$ and a prime number $\ell\neq p$. \smallskip

The main goal of this section is to construct the weight filtration on any mixed perverse sheaf $\F$ on $X\times_s \eta$. Firstly, we recall the definition of perverse sheaves on Deligne's topos.

\begin{defn}\label{defn:perverse} An object $\F\in D^b_c(X\times_s\eta; \Q_\ell)$ is {\it perverse} if $\F$ lies in the heart of the perverse $t$-structure constructed in Lemma~\ref{lemma:perverse-t-structure-rationally}. We denote by $\rm{Perv}(X\times_s \eta; \Q_\ell)$ the (abelian) category of perverse objects in $D^b_c(X\times_s\eta;\Q_\ell)$.
\end{defn}

\begin{rmk} Alternatively, an object $\F\in D^b_c(X\times_s \eta; \Q_\ell)$ is perverse if and only if $\pi^*_X\F\in D^b_c(X_{\ov{s}}; \Q_\ell)$ is perverse.  
\end{rmk}

In order to construct the weight filtration, it will be convenient to descend the mondoromy operator $N\colon \pi^*_X\F\to \pi^*_X\F(-1)$ to a morphism $N\colon \F \to \F(-1)$ for a perverse sheaf $\F$. \smallskip

We start the section by explaining this descend argument. For any  $\F\in \rm{Perv}(X\times_s\eta; \Q_\ell)$, Lemma~\ref{lemma:monodromy-operator} provides us with a canonical nilpotent operator
\[
N\colon \pi^*_X\F \to \pi^*_X\F(-1).
\]
Now \cite[Lemma III.4.3]{KW} ensures that 
\[
\rm{RHom}_{\Q_\ell}(\pi^*_X\F, \pi^*_X\F(-1)) \in D^{\geq 0}(\Q_\ell),
\]
so Lemma~\ref{lemma:hom-formula-rational} implies that
\[
\rm{Hom}_{\Q_\ell}(\F, \F(-1)) = \rm{Hom}_{\Q_\ell}(\pi^*_X\F, \pi^*_X\F(-1))^{G_\eta}.
\]
Therefore, Corollary~\ref{cor:commutation} implies that the monodromy operator $N$ is $G_\eta$-invariant, and so it descends to a nilpotent morphism
\[
N\colon \F \to \F(-1)
\]
in the abelian category $\rm{Perv}(X\times_s \eta; \Q_\ell)$. We will often abuse the notation and denote two versions of $N$ by the same letter. \smallskip

\begin{defn}\label{defn:monodromy-operator-perverse} A {\it monodromy operator} of $\F\in \rm{Perv}(X\times_s \eta; \Q_\ell)$ is a nilpotent morphism 
\[
    N\colon \F \to \F(-1)
\]
constructed above.
\end{defn}

In order to construct the weight filtration, we follow the strategy of \cite[Th\'eorem\`e 5.3.5]{BBD}. The crucial missing ingredients is purity of all mixed simple perverse sheaves on $X\times_s \eta$, and vanishing of certain Ext groups between pure objects. We prove both results in this section. 
 
\begin{lemma}\label{lemma:simple-local-systems-pure} Let $X$ be an irreducible smooth finite type $k$-scheme and let $\F$ be a simple $\Q_\ell$-local system on $X\times_s \eta$ (see Definition~\ref{defn:lisse-objects}). Then $\F$ is pure. 
\end{lemma} 
\begin{proof}
    {\it Step~$1$. Reduce to the case of a trivial inertia action on $\pi^*_X\F$.} Consider a nilpotent monodromy operator $N\colon \F \to \F(-1)$. Since $\F$ is simple, we conclude that $\F=\ker N$, so $N=0$. Therefore, Lemma~\ref{lemma:monodromy-operator} ensures that the inertia group $I$ acts on $\pi^*_X\F$ via a finite subquotient $I/I_1$. Therefore, there is a finite totally ramified Galois extension $K\subset K'$ such that $I_{K'}$ acts trivially on $\pi^*_X\F$. Let us denote by 
    \[
    b\colon X\times_s \eta' \to X\times_s \eta
    \]
    the natural morphism of Deligne's topoi. Corollary~\ref{cor:pure-independent} ensures that $\F$ is pure if and only if $b^*\F$ is pure. Moreover, $I_{K'}$ acts trivially on $\pi^*_X b^*\F$ by construction. Thus we would like to say it suffices to show the claim for $b^*\F$. However, $b^*\F$ may not be simple anymore, so we cannot simply replace $\F$ with $b^*\F$ for the purpose of proving that $\F$ is pure. To make this reduction work, we use a version of the Galois descent. \smallskip
    
    First, we note that $\eta'$ is a slice topos $\eta_{/h_{\Spec K'}}$ for an effective epimorphism $h_{\Spec K'} \to *$. Therefore, \cite[Exp. IV, Proposition 5.11]{SGA4_2} implies that $X\times_s \eta'$ is a slice topos $(X\times_{s}\eta)_{/q^*(h_{\Spec K'})}$, where $q\colon X\times_s \eta \to \eta$ is the natural projection. In particular, $q^*(h_{\Spec K'}) \to *$ is an effective epimorphism as a pullback of an effective epimorphism. \smallskip
    
    Now we use that $q^*$ commutes with finite limits and colimits to deduce that 
    \begin{align*}
    q^*(h_{\Spec K'})\times_{*} q^*(h_{\Spec K'}) &\simeq q^*(h_{\Spec K'})\times_{q^*(*)} q^*(h_{\Spec K'}) \\
    & \simeq q^*(h_{\Spec K'}\times_* h_{\Spec K'}) \\
    & \simeq q^*(\sqcup_{g\in G_{K'/K}} h_{\Spec K'}) \\
    &\simeq \sqcup_{g\in G_{K'/K}} q^*(h_{\Spec K'}).
    \end{align*}
    So $q^*(h_{\Spec K'})\to *$ is a Galois covering with the Galois group $G_{K'/K}$. In particular, for each $g\in G_{K'/K}$, there is an equivalence
    \[
    c_g^*\colon D^b_c(X\times_s \eta'; \Q_\ell) \to D^b_c(X\times_s \eta'; \Q_\ell)
    \]
    satisfying the cocycle condition. By Galois descent\footnote{Galois descent for $\Z/\ell^n\Z$-local system is obvious. The case of $\Z_\ell$-local systems follows by taking a limit. To get descent for $\Q_\ell$-local systems, one uses Lemma~\ref{lemma:existence-of-a-lattice}.}, we can identify $\Q_\ell$-local systems on $X\times_s \eta$ with $\Q_\ell$-local systems on $X\times_s \eta'$ equipped with a family of isomorphisms
    \[
    \psi_g\colon c_g^*\F\to \F
    \]
    satisfying the cocycle condition. \smallskip
    
    Now suppose that $\G\subset b^*\F$ is a simple non-zero sub-local system. Denote by $\{\G_i\}_{i\in I}$ the set of isomorphism classes of $\Q_\ell$-local systems $c_g^*\G$ for all $g\in G_{K'/K}$. Since $b^*\F$ is defined over $X\times_s \eta$, we see that each $\G_i$ is also a sub-local system of $b^*\F$. Since $\G$ is simple (and so all $\G_i$ are simple), we conclude that there is an inclusion
    \[
    \varphi\colon \bigoplus_{i\in I} \G_i \to b^*\F.
    \]
    By construction, this inclusion is $G_{K'/K}$-stable, thus it defines a non-zero inclustion
    \[
    \cal{H} \to \F
    \]
    of $\Q_\ell$-local systems on $X\times_s \eta$. Since $\F$ is simple, we conclude that $\cal{H}\to \F$ is an isomorphism. Therefore, $\varphi$ is isomorphism as well. Now we use that, for each $i\in I$, there is $g\in G_{K'/K}$ such that $c_g^*\G \simeq \G_i$, thus $\G$ is pure of weight $w$ if and only if $\G_i$ is pure of weight $w$ for each $i\in I$. Therefore, $b^*\F$ is pure of weight $w$ if and only if $\G$ is pure of weight $w$. So  Corollary~\ref{lemma:mixed-after-base-change} ensures that $\F$ is pure of weight $w$ if and only if $\G$ is pure of weight $w$. Thus, we can replace $K$ with $K'$ and $\F$ with $\G$ for the purpose of proving that $\F$ is pure. \smallskip
    
    {\it Step~$2$. Finish the proof.} Since the inertia group $I$ acts trivially on $\F$, Corollary~\ref{cor:trivial-action-descend-rationally} guarantees that $\F$ descends to a $\Q_\ell$-local system on $X$. More precisely, there is a $\Q_\ell$-local system $\G$ on $X$ such that $p_X^*\G\simeq \F$, where $p_X\colon X\times_s \eta\to X_{\et}$ is the projection morphism. Since $p_X$ is conservative, we conclude that $\G$ must be a simple $\Q_\ell$-local system on $X$. In particular, for every continuous section $\sigma\colon G_s\to G_\eta$ of the projection $r\colon G_\eta \to G_s$, we see that the $\Q_\ell$-local system $\sigma_X^*\F\simeq \sigma^*_Xp^*_X\G \simeq \G$ is simple and mixed. Thus it is pure by \cite[Th\'eorem\`e (3.4.1)(ii) and Variante (3.4.9)]{Weil2}.   
\end{proof}
 
\begin{lemma}\label{lemma:simple-perverse-sheaves-pure} Let $X$ be a finite type $k$-scheme and let $\F\in \rm{Perv}(X\times_s \eta; \Q_\ell)$ be a mixed simple perverse sheaf. Then $\F$ is pure.
\end{lemma}
\begin{proof}
   By Lemma~\ref{lemma:classification-simple-perverse-objects}, there is a locally closed subscheme $U\subset X$ and a simple $\Q_\ell$-local system $\G$ on $U$ such that $U_{\rm{red}}$ is smooth and $\F\simeq (j\times_s \eta)_{!*}(\G[\dim U])$. Since $(j\times_s \eta)_{!*}$ preserves pure perverse sheaves by \cite[Corollaire 5.3.2]{BBD} and Lemma~\ref{lemma:extension-pullback}(\ref{lemma:extension-pullback-2}), it suffices to show that $\G[\dim U]$ is pure. This follows from Lemma~\ref{lemma:simple-local-systems-pure}.  
\end{proof}

Now we discuss the vanishing result for Ext groups. We start with a preliminary lemma on vanishing of Galois cohomology groups.

\begin{lemma}\label{lemma:no-H-1} Let $V$ be a continuous finite dimensional $\Q_\ell$-representation of $G_\eta$. Suppose that the inertia group $I$ acts trivially on $V$, and the eigenvalues of the action of the geometric Frobenius $F\in G_s$ are $q$-Weil number of strictly positive weights. Then $\rm{H}^0_{\rm{cont}}(G_\eta, V)=0$ and $\rm{H}^1_{\rm{cont}}(G_\eta, V)=0$. 
\end{lemma}
\begin{proof}
    The $\rm{H}^0$-claim is clear because 
    \[
    \rm{H}^0_{\rm{cont}}(G_\eta, V)= V^{G_\eta}=V^{G_s}=0
    \]
    because all eigenvalues of $F$ have strictly positive weights. \smallskip
    
    Now we discuss the $\rm{H}^1$-claim. The Hochschild--Serre spectral sequence for the closed normal subgroup $I\subset G_\eta$ gives us the spectral sequence
    \[
    \rm{E}^{p, q}_2=\rm{H}^p_{\rm{cont}}(G_s, \rm{H}^q_{\rm{cont}}(I, V)) \Longrightarrow \rm{H}^{p+q}_{\rm{cont}}(G_\eta, V).
    \]
    So we have the following short exact sequence
    \[
    0 \to \rm{H}^1_{\rm{cont}}(G_s, V^I) \to \rm{H}^1_{\rm{cont}}(G_\eta, V) \to \rm{H}^1_{\rm{cont}}(I, V)^{G_s}. 
    \]
    We use the isomorphism $G_s\simeq \wdh{\Z}$ to identify the first cohomology group $\rm{H}^1_{\rm{cont}}(G_s, V^I)$ with the Frobenius coinvariants $(V^I)_{F}$. Since the geometric Frobenius acts with positive weights on $V$, the same holds for the action of Frobenius on $V^I$. Therefore, we conclude that $(V^I)_{F}=0$. \smallskip
    
    We are only left to show that $\rm{H}^1_{\rm{cont}}(I, M)^{G_s}=0$. Note that  $I$ fits into a short exact sequence
    \[
    0 \to P_{\ell} \to I \xr{t_\ell} \Z_\ell(1) \to 0,
    \]
    where $P_{\ell}$ is pro-(prime-to-$\ell$)-group. We use the Hoschield-Serre spectral sequence and the fact that pro-(prime-to-$\ell$) groups have trivial higher continuous cohomology with coefficients in finite $\Q_\ell$-vector spaces to conclude that
    \[
    \rm{H}^1_{\rm{cont}}(I, V) = \rm{H}^1_{\rm{cont}}(\Z_\ell(1), V).
    \]
    Now the action of $I$ is trivial on $V$ by assumption. Therefore, the same holds for the $\Z_\ell(1)$-action, and so  
    \[
    \rm{H}^1_{\rm{cont}}(I, V) = \rm{H}^1_{\rm{cont}}(\Z_\ell(1), V) = \rm{Hom}_{\rm{cont}}(\Z_\ell(1), V)=V(-1).
    \]
    Therefore, since $\Q_\ell(-1)$ is pure of weight $2$, the weights of Frobenius action on $V(-1)$ are still strictly positive. Thus $V(-1)^{G_s}=0$ finishing the proof.
\end{proof}

\begin{lemma}\label{lemma:no-homs-pure} Let $X$ be a finite type $k$-scheme, let $\F\in \rm{Perv}(X\times_s \eta; \Q_\ell)$ be pure of weight $w$, and let $\G\in \rm{Perv}(X\times_s \eta; \Q_\ell)$ be pure of weight $w'$. Suppose $w< w'$, then $\rm{Hom}_{\Q_\ell}(\F, \G)=0$ and $\rm{Ext}^1_{\Q_\ell}(\F, \G)=0$.
\end{lemma}
\begin{proof}
    {\it Step~$1$. The $\rm{Hom}$-statement.} Firstly, we note that the object
    \[
    \rm{RHom}_{/\eta,\Q_\ell}(\F, \G)\coloneqq \rm{R}(f\times_s \eta)_*\rm{R}\cal{H}om_{\Q_\ell}\in D^b_c(\eta; \Q_\ell)
    \]
    is mixed of weights $\geq w'-w$ by Lemma~\ref{lemma:pushforward-pullback}(\ref{lemma:pushforward-pullback-2}), Lemma~\ref{lemma:homs-after-base-change}, Remark~\ref{rmk:rational-good}, and \cite[Stabilit\'es 5.1.14]{BBD}. In particular, action of any lift of Frobenius $F\in G_\eta$ on 
    \[
    \rm{Ext}^i_{/\eta, \Q_\ell}(\F, \G)\coloneqq \rm{H}^i(\rm{RHom}_{/\eta,\Q_\ell}(\F, \G))
    \]
    has eigenvalues $q$-Weil numbers of weight $\geq w'-w>0$ for $i\geq 0$. Now Corollary~\ref{cor:compute-hom-Q-ell} and \cite[Lemma III.4.3]{KW} imply that
    \[
    \rm{Hom}_{\Q_\ell}(\F, \G) = \rm{Hom}_{/\eta, \Q_\ell}(\F, \G)^{G_\eta}.
    \]
    Since all eigenvalues of $F$ acting of $\rm{Hom}_{/\eta, \Q_\ell}(\F, \G)$ are $q$-Weil numbers of {\it strictly positive} weights, we conclude that there are no-nontrivial invariants. Thus $\rm{Hom}_{\Q_\ell}(\F, \G)=0$.\smallskip
    
    {\it Step~$2$. The $\rm{Ext}^1$-claim can be checked after a finite Galois extension $K\subset K'$.} Let $K\subset K'$ be a finite Galois extension with the Galois group $G_{K'/K}$, and $b\colon X_{s'}\times_{s'} \eta' \to X_s\times_s \eta$ the corresponding projection morphism. Then Lemma~\ref{lemma:hom-formula-rational} (or a Galois descent argument as in the proof of Lemma~\ref{lemma:simple-local-systems-pure}) implies that 
    \[
    \rm{RHom}_{\Q_\ell}(\F, \G) \simeq \rm{R}\Gamma(G_{K'/K}, \rm{RHom}_{\Q_\ell}(b^*\F, b^*\G)).
    \]
    Since a finite group does not have higher cohomology groups with coefficients in a $\Q_\ell$-vector space, we conclude that
    \[
    \rm{Ext}^i_{\Q_\ell}(\F, \G) \simeq \rm{Ext}^i_{\Q_\ell}(b^*\F, b^*\G)^{G_{K'/K}}. 
    \]
    Therefore, it suffices to prove the result after a finite Galois extension $K\subset K'$. \smallskip

    {\it Step~$3$. Reduce to the case $I$ acts trivially on both $\pi^*_X\F$ and $\pi^*_X\G$.} By Step~$1$, there are non-trivial homomorphisms 
    \[
    \pi_X^*\F \to \pi^*_X\F(-1)
    \]
    and 
    \[
    \pi_X^*\G \to \pi^*_X\G(-1).
    \]
    In particular, the monodromy operators $N_{\F}$ and $N_{\G}$ are zero. Thus, Lemma~\ref{lemma:monodromy-operator} implies that there is an open subgroup $I_1\subset I$ that acts trivially on $\pi^*_X\F$ and $\pi^*_X\G$. In particular there is a finite totally ramified Galois extension $K\subset K'$ such that $I_{K'}$ acts trivially on $\pi^*_Xb^*\F$ and $\pi^*_Xb^*\G$. Therefore, Step~$2$ guarantees that we can replace $K$ with $K'$, $\F$ with $b^*\F$, and $\G$ with $b^*\G$ to assume that the action of inertia is trivial on both $\pi^*_X\F$ and $\pi^*_X\G$. \smallskip

    {\it Step~$4$. Finish the proof.} Lemma~\ref{lemma:hom-formula-rational} implies that we have an exact sequence
    \[
    0\to \rm{H}^1_{\rm{cont}}(G_\eta, \rm{Hom}_{/\eta, \Q_\ell}(\F, \G))\to \rm{Ext}^1_{\Q_\ell}(\F, \G) \to \rm{Ext}^1_{/\eta, \Q_\ell}(\F, \G)^{G_\eta}.
    \]
    Since the weights are strictly positive for the $F$-action on $\rm{Ext}^1_{/\eta, \Q_\ell}(\F, \G)$, we conclude that 
    \[
    \rm{Ext}^1_{/\eta, \Q_\ell}(\F, \G)^{G_\eta}=0.
    \]
    Therefore, it suffices to show that
    \[
    \rm{H}^1_{\rm{cont}}(G_\eta, \rm{Hom}_{/\eta, \Q_\ell}(\F, \G))=0.
    \]
    This follows from Lemma~\ref{lemma:no-H-1}.
\end{proof}

\begin{thm}\label{thm:weight-filtration} Let $X$ be a finite type $k$-scheme and let $\F\in \rm{Perv}(X\times_s \eta; \Q_\ell)$ be a mixed perverse sheaf. Then there is a unique functorial increasing weight filtration
\[
\rm{Fil}_{\rm{W}}^n\F \subset \F
\]
such that 
\begin{enumerate}
    \item each $\rm{Fil}_{\rm{W}}^n\F$ is a perverse sheaf;
    \item $\rm{Gr}^n_{\rm{W}} \F$ is zero or a pure sheaf of weight $n$;
    \item $\rm{Fil}_{\rm{W}}^{-n}\F = 0$ and $\rm{Fil}_{\rm{W}}^n\F = \F$ for a large $n\gg 0$.
\end{enumerate}
Furthermore, the weight filtration satisfies the following properties:
\begin{enumerate}
    \item any morphism of mixed perverse sheaves $f\colon \F\to \G$ is strictly compatible with the weight filtrations, i.e.  $f(\rm{Fil}_{\rm{W}}^\bullet \F) =\rm{Fil}_{\rm{W}}^\bullet \G \cap f(\F)$;
    \item for any continuous section $\sigma\colon G_s \to G_\eta$ of the projection $r\colon G_\eta \to G_s$, there is an equality of filtrations
    \[
    \sigma_X^*\rm{Fil}_{\rm{W}}^\bullet \F = \rm{Fil}_{\rm{W}}^\bullet \sigma_X^*\F,
    \]
    where $\rm{Fil}_{\rm{W}}^\bullet \sigma_X^*\F$ is the weight filtration from \cite[Th\'eorem\`e 5.3.5]{BBD}.
\end{enumerate}
\end{thm}
\begin{proof}
    The proof of \cite[Th\'eorem\`e 5.3.5]{BBD} (or \cite[Lemma III in \textsection III.9]{KW}) adapts to this situation essentially without any change using the results already obtained in this section. For the convenience of the reader, we repeat the argument here. 

    We start with existence of the weight filtration. We note that $\rm{Perv}(X\times_s \eta)$ is noetherian and artinian by Lemma~\ref{lemma:perverse-artinian}. Thus every object is of finite length. We argue by induction on the length $l(\F)$. \smallskip
    
    If $l(\F)=1$, then $\F$ is a simple perverse sheaf, and so it is pure by Lemma~\ref{lemma:simple-perverse-sheaves-pure}. Thus it clearly admits a weight filtration. Now suppose that $l=l(\F)>1$ and we know existence of a weight filtration for any mixed perverse sheaf $\G$ of length $< l$. Then we pick any simple perverse sheaf $\F_0\subset \F$, it is mixed by \cite[Proposition 5.3.1]{BBD}, and so it is pure of some weight $w'$ by Lemma~\ref{lemma:simple-perverse-sheaves-pure}. Consider a short exact sequence 
    \begin{equation}\label{eqn:ses}
    0 \to \F_0 \to \F \xr{\alpha} \G \to 0.
    \end{equation}
    Clearly, $\G$ is a mixed perverse sheaf of length $<l$. Therefore, it admits a weight filtration by the induction hypothesis. We consider two cases:
    
    {\it Step~$1$. $\rm{Fil}_{\rm{W}}^{w'-1}\G=\G$.} Then (\ref{eqn:ses}) splits by Lemma~\ref{lemma:no-homs-pure} (applied to $\F_0$ and $\rm{gr}_{\rm{W}}^w\G$ for $w<w'$), so $\F=\F_0\oplus \G$. In particular, the filtration
    \[
    \rm{Fil}_{\rm{W}}^n\F = \rm{Fil}_{\rm{W}}^n\G \text{ for } n\leq w'-1,
    \]
    \[
    \rm{Fil}_{\rm{W}}^n \F = \F \text{ for } n\geq w'
    \]
    does the job. \smallskip
    
    {\it Step~$2$. General $\G$.} We consider a perverse sheaves $\G'\coloneqq \rm{Fil}_{\rm{W}}^{w'-1}\G$, and $\F' \coloneqq \alpha^{-1}(\G') \subset \F$. Step~$1$ guarantees that $\F'$ admits a weight filtration with weights $\leq w'-1$, and there is a short exact sequence
    \[
    0\to\F' \to \F \to \G'' \to 0
    \]
    with $\G''\coloneqq \G/\G'$. By construction, $\G''$ admits a weight filtration with weights $\geq w'$. Therefore, the filtration
    \[
    \rm{Fil}_{\rm{W}}^{n} \F = \rm{Fil}_{\rm{W}}^{n} \F' \text{ if } n\leq w'-1,
    \]
    \[
    \rm{Fil}_{\rm{W}}^n \F = \alpha^{-1}(\rm{Fil}_{\rm{W}}^n \G'') \text{ if } n \geq w'
    \]
    is a weight filtration on $\F$. \smallskip
    
    Uniqueness of a weight filtration follows from the Hom-statement in Lemma~\ref{lemma:no-homs-pure}. \smallskip
    
    Now note that \cite[Proposition 5.3.1]{BBD} implies that a subquotient of a perverse pure sheaf of weight $w$ is pure of weight $w$. Thus, for any morphism $f\colon \F \to \G$ of mixed perverse sheaves, both $f(\rm{Fil}_{\rm{W}}^\bullet \F)$ and $\rm{Fil}_{\rm{W}}^\bullet \G \cap f(\F)$ define weight filtrations on $f(\F)$. Therefore, they must coincide by the uniqueness of a weight filtration. \smallskip
    
    Finally, if $\sigma\colon G_s\to G_\eta$ is any continuous section, then $\sigma^*_X \rm{Fil}_{\rm{W}}^\bullet \F$ is an essentially finite filtration by mixed perverse sheaves such that $w$-th graded piece is pure of weight $w$. Thus it should coincide with the weight filtration from \cite[Th\'eorem\`e 5.3.5]{BBD} due to the uniqueness of a weight filtration. 
\end{proof}


\subsection{Monodromy-pure sheaves on Deligne's topos}\label{section:log-pure}

The main goal of this section is to define the notion of a monodromy-pure object in $D^b_c(X\times_s \eta; \Q_\ell)$. The motivation behind the definition is that most interesting mixed objects in $D^b_c(X\times_s \eta; \Q_\ell)$ are rarely pure (e.g. nearby cycles). However, one would wish to define some notion in-between pure and mixed sheaves to capture these interesting examples. This is done via the notion of a monodromy-pure sheaf that is essentially an axiomatization of monodromy weight conjecture. \smallskip

For the rest of the section, we fix an arithmetic non-archimedean field $K$ of residue characteristic $p$, a prime number $\ell\neq p$, and a finite type $k$-scheme $X$. 

We refer to Definition~\ref{defn:monodromy-operator-perverse} for the definition of a nilpotent operator for an perverse sheaf $\F\in \rm{Perv}(X\times_s \eta; \Q_\ell)$. By \cite[(1.6.1)]{Weil2}, there is a unique increasing {\it monodromy filtration} $\rm{Fil}_{\rm{M}}^\bullet\F$ such that
\begin{enumerate}
    \item $\rm{Fil}_{\rm{M}}^{-k}\F=0$ and $\rm{Fil}_{\rm{M}}^k\F = \F$ for a sufficiently large $k$;
    \item $N(\rm{Fil}^k_{\rm{M}} \F)$ lies in $\rm{Fil}_{\rm{M}}^{k-2}\F(-1)$;
    \item $N$ induces an isomorphism on the associated graded pieces
\[
N^k\colon \rm{gr}_M^k\F \xr{\sim} \rm{gr}_M^{-k}\F(-k)
\]
for each $k\geq 0$. 
\end{enumerate}

\begin{defn}\label{defn:monodromy-pure} A mixed perverse sheaf $\F\in \rm{Perv}(X\times_s\eta; \Q_\ell)$ is {\it monodromy-pure of weight $w$} if $\rm{gr}^M_{k}\F$ is pure of weight $k+w$ for each integer $k$. 

An object $\F\in D^b_c(X\times_s \eta; \Q_\ell)$ is {\it monodromy-pure of weight $w$} if ${}^p\cal{H}^i(\F)$ is monodromy-pure of weight $w+i$ for each integer $i$. 
\end{defn}  

\begin{rmk} Alternatively, one can reformulate Definition~\ref{defn:monodromy-pure} by saying that the weight filtration coincides with the shifted monodromy filtration. 
\end{rmk}

Now we show that we can check that a complex $\F\in D(X\times_s \eta; \Q_\ell)$ is monodromy-pure after a certain extension of arithmetic fields. 


\begin{defn} An extension of non-archimedean fields $K\subset L$ is {\it topologically algebraic} if there is an algebraic extension $K\subset K'$ with an isomorphism of topological $K$-algebras $\wdh{K'}\simeq L$ 
\end{defn}

\begin{lemma}\label{lemma:monodromy-pure-after-extension} Let $K\subset K'$ be a topologically algebraic extension of arithmetic non-archimedean fields, let $X$ be a finite type $k$-scheme, let $\F\in D^b_c(X\times_s \eta; \Q_\ell)$, and let $b\colon X_{s'}\times_{s'} \eta' \to X_s\times_s \eta$ be the morphism of topoi from Notation~\ref{notation:field-extension}. Then $\F$ is a monodromy-pure perverse sheaf of weight $w$ if and only if $b^*\F\in D^b(X_{s'}\times_{s'}\eta'; \Q_\ell)$ is a monodromy-pure perverse sheaf of weigth $w$.
\end{lemma}
\begin{proof}
    It is easy to see that $\F$ is perverse if and only if $b^*\F$ is perverse. Lemma~\ref{lemma:mixed-after-base-change} ensures that $\F$ is mixed if and only if $b^*\F$ is mixed. Now we denote by $I_{\ell, K}$ (resp. $I_{\ell, K'}$) be the canonical $\Z_\ell(1)$ quotient of the inertia group $I_K$ (resp. $I_{K'}$). Since $K\subset K'$ is topologically algebraic, \cite[Prop.~2.4.6(iii)]{Berkovich-etale}\footnote{We note that $M_K$ in \cite[Prop.~2.4.6(iii)]{Berkovich-etale} denotes the group $I_K/P_K$, where $I_K\subset G_K$ is the inertia subgroup and $P_K\subset G_K$ is the wild inertia subgroup} implies that the natural morphism $I_K/P_K \to I_{K'}/P_{K'}$ is injective. By passing to the pro-$\ell$ parts, we conclude that we get an injective morphism
    \[
    I_{\ell, K'} \subset I_{\ell, K}.
    \]
    Since $I_{\ell, K'}$ and $I_{\ell, K}$ are (non-canonically) isomorphic to $\Z_\ell$, \cite[Proposition 2.7.1(a,b)]{profinite} implies that $I_{\ell, K'} \to I_{\ell, K}$ is an injective morphism of finite index. Therefore, the uniqueness claim in Lemma~\ref{lemma:monodromy-operator} ensures that $b^*N=N$, and so the monodromy filtration on $\F$ pullbacks to the monodromy filtration on $b^*\F$ (for example, by the uniqueness property of the monodromy filtration). Therefore, the claim follows from Lemma~\ref{lemma:mixed-after-base-change}. 
\end{proof}

\begin{lemma}\label{lemma:finite-morphisms-monodromy-pure} Let $K$ be an arithmetic non-archimedean field, let $f\colon X\to Y$ be a finite morphism of finite type $k$-schemes, and let $\F\in D^b_c(X\times_s \eta; \Q_\ell)$. Then $\F$ is a monodromy-pure perverse sheaf of weight $w$ if and only if $\rm{R}(f\times_s \eta)_*\F\in D^b(Y\times_s \eta;\Q_\ell)$ is a monodromy-pure perverse sheaf of weigth $w$.
\end{lemma}
\begin{proof}
    Since $f$ is a finite morphism, $\rm{R}(f\times_s \eta)_*$ is perverse exact. Thus, we conclude that $\rm{R}(f\times_s \eta)_* \rm{gr}_{\rm{M}}^i \F \simeq \rm{gr}_{\rm{M}}^i \rm{R}(f\times_s \eta)_*\F$. So it suffices to show that a perverse sheaf $\G\in \rm{Perv}(X\times_s \eta; \Q_\ell)$ is pure of weight $w$ if and only if $\rm{R}(f\times_s \eta)_*\G$ is pure of weight $w$. After choosing a continuous splitting $\sigma\colon G_s\to G_\eta$, it boils down to an analogous question for perverse sheaves on $X$, which is classical and left to the reader. 
\end{proof}

\section{Quasi-unipotent monodromy Theorem and mixedness of nearby cycles}

The main goal of this section is to prove the Grothendieck Quasi-Unipotent Monodromy Theorem for rigid-analytic varieties and mixedness of the $\ell$-adic nearby cycles (if residue field is a finite field). 

For the rest of the section, we fix a discretely valued non-archimedean field $K$ with a uniformizer $\varpi\in \O_K$, a completed algebraic closure $C=\wdh{\ov{K}}$, a prime number $\ell\neq p$, and a ring $\Lambda=\Z/\ell^n\Z$, $\Z_\ell$, or $\Q_\ell$.\smallskip

For a finite extension $K\subset L$, we denote by $k\subset l$ the induced extension of residue fields. We also denote by $\eta'=(\Spec L)_\et$ the \'etale topos $\Spec L$ and by $s'=(\Spec l)_\et$ the \'etale topos $\Spec l$.\smallskip

We denote the Galois group of $K$ by $G_\eta$ (or $G_K$ if there is any ambiguity) and the inertia group by $I$ (or $I_K$ if there is any ambiguity). \smallskip

\subsection{Nearby cycles of constant sheaves}\label{section:nearby-constant}

In this section, we discuss some preliminary results that we will need in our proof of the Grothendieck Quasi-Unipotent Monodromy Theorem. \smallskip

\begin{defn} A morphism of admissible formal $\O_K$-schemes $\mf\colon \X \to \cY$ is {\it rig-surjective} if its generic fiber $\mf_{\eta}\colon \X_\eta \to \cY_\eta$ is a surjective morphism of adic spaces. 

A {\it rig-surjective site} is a site whose underlying category is $\rm{Adm}_{\O_K}$ is the category of admissible formal $\O_K$-schemes, and whose coverings are given by families $\{\mf_i\colon \X_i \to \X\}_{i\in I}$ such that $I$ is finite and $|\X|=\cup_{i\in I} |\mf_i|(|\X_i|)$. 

A {\it $v$-site} is a site whose underlying category is $\rm{Ad}^{\rm{qcqs}}_{\Q_p}$ is the category of qcqs strongly noetherian adic spaces over $\Spa(\Q_p, \Z_p)$, and whose coverings are given by families $\{f_i\colon X_i\to X\}_{i\in I}$ such that $I$ is finite, $f_i$ are of finite type, and $|X|=\cup_{i\in I} |f_i|(|X_i|)$.
\end{defn}

The first step is to show that any $v$-hypercovering is of universal cohomological descent. Then we show that any rigid-analytic space admits a $v$-hypercovering by rigid-analytic varieties with especially nice formal models. \smallskip

We refer the reader to \cite{conrad-hyper} and \cite[\href{https://stacks.math.columbia.edu/tag/01FX}{Tag 01FX}]{stacks-project} for an extensive discussion of hypercovers. 

\begin{defn} An augmented simplicial object $a\colon Y_\bullet \to X$ in $\rm{Ad}^{\rm{qcqs}}_{\Q_p}$ is {\it of cohomological descent} if the natural morphism
\[
\F \to \rm{R}a_*a^*\F
\]
is an isomorphism for any $\F\in D^+(X_\et; \Lambda)$.

An augmented simplicial object $a\colon Y_\bullet \to X$ in $\rm{Ad}^{\rm{qcqs}}$ is {\it of universal cohomological descent} if, for every morphism $X'\to X$ in $\rm{Ad}^{\rm{qcqs}}_{\Q_p}$, the base change $Y_\bullet\times_X X'\to X'$ is of cohomological descent. 
\end{defn}

\begin{lemma}\label{lemma:hyperdescent} Let $a\colon Y_\bullet \to X$ be a $v$-hypercovering in $\rm{Ad}^{\rm{qcqs}}_{\Q_p}$. Then $a$ is of cohomological descent, i.e.
the natural morphism
\[
\F \to \rm{R}a_*a^*\F
\]
is an isomorphism for any $\F\in D^+(X_\et; \Lambda)$.
\end{lemma}
\begin{proof}
    We give a proof for $\Lambda=\Z/\ell^n\Z$, the case of $\Lambda=\Z_\ell$ or $\Q_\ell$ follows formally from this one by passing to a limit. 

    {\it Step~$1$: $X=\Spa(C, C^+)$ for an algebraically closed field $C$, and $Y_\bullet \to X$ is a \v{C}ech covering for a $v$-covering $Y\to X$.} We note that the morphism $Y\to X$ admits a section by (the proof of) \cite[Lemma 7.2.3]{H3}. Thus $\check{C}(Y/X)\to X$ is of universal cohomological descent by \cite[Theorem 7.2]{conrad-hyper}. \smallskip
    
    {\it Step~$2$: $Y_\bullet \to X$ is a \v{C}ech covering for a $v$-covering $Y\to X$.} For any $x\in X$, we denote by $(C(x), C(x)^+)$ a Huber pair obtained as a completed algebraic closure of $(k(x), k(x)^+)$ where $k(x)$ is the residue field of $X$ at $x$. This pair comes with the natural morphism $g_x\colon \Spa(C(x), C(x)^+) \to X$ sending the unique closed point of $\Spa(C(x), C(x)^+)$ to $x\in X$. \smallskip
    
    Since each $a_i\colon Y_i \to X$ is of finite type, \cite[Theorem 4.1.1(c)']{H3} ensures that the formation of $\rm{R}^ja_{i, *}$ commutes with $g_x^*$ for any $j\geq 0$. Thus, we can argue as in the proof of \cite[Theorem 7.7]{conrad-hyper} to reduce to the case $X=\Spa(C(x), C(x)^+)$. This case was already done in Step~$1$. \smallskip

    {\it Step~$3$: General $v$-hypercovering $Y_\bullet \to X$.} Step~$2$ and the definition of a $v$-hypercovering imply that the natural morphisms
    \[
    Y_{n+1} \to \rm{cosk}_n\rm{sk}_n (Y_\bullet/X)
    \]
    are of universal cohomological descent. Thus, \cite[Theorem 7.15]{conrad-hyper} ensures that $\rm{cosk}_n\rm{sk}_n (Y_\bullet/X) \to X$ is of universal cohomological descent. Thus, \cite[Lemma 7.14]{conrad-hyper} implies that $Y_\bullet \to X$ is of universal cohomological descent as well. 
\end{proof}

Now we show that any admissible formal $\O_K$-scheme admits a rig-surjective hypercovering by strictly semi-stable formal $\O_K$-schemes in some weak sense.

\begin{defn} A finitely presented $\O_K$-scheme $X$ is called {\it strictly semi-stable} if Zariski-locally it admits an \'etale morphism 
\[
U \to \Spec \frac{\O_K[t_0, \dots, t_l]}{(t_{0}\cdots t_{m}-\pi)}
\]
for some integers $m\leq l$, and a uniformizer $\pi \in \m_K\setminus \m_K^2$. \smallskip

A formal $\O_K$-scheme $\X$ is {\it algebraically strictly semi-stable} if there a strictly semi-stable $\O_K$-scheme $X$ such that $\X$ is isomorphic to the formal $\varpi$-adic completion of $X$.
\end{defn}

\begin{lemma}\label{lemma:strictly-semi-stable-covering} Let $K$ be a $p$-adic non-archimedean discretely valued field and let $\X$ be an admissible formal $\O_K$-scheme such that $\X_\eta$ is of (pure) dimension $d$. Then there is a finite extension $K\subset K'$ and a rig-surjection $\X'\to \X_{\O_{K'}}$ such that $\X'$ is an algebraically strictly semi-stable formal $\O_{K'}$-scheme with $\X'_\eta$ of (pure) dimension $d$.
\end{lemma}
\begin{proof}
    Consider the generic fiber $\X_\eta$. First, \cite[Theorem 5.2.2]{Temkin-resolution} implies that there exists a resolution of singularities $g \colon X' \to \X_{\eta, \rm{red}}$. By composing it with the natural surjection $\iota \colon \X_{\eta, \rm{red}} \to \X_\eta$, we obtain the morphism 
    \[
    f\colon X' \to \X_{\eta}.
    \]
    If $\X_\eta$ is of (pure) dimension $d$, the same holds for $X'$. Now $f$ can be extended to a morphism of formal $\O_K$-schemes $\X' \to \X$ by \cite[Lemma 8.4/4]{B} that is rig-surjective by construction. Therefore, it suffices to prove the question in the situation when $\X$ has smooth generic fiber $\X_\eta$. In this case, the result follows from the proof of \cite[Theorem 3.3.1]{Temkin} or \cite[Theorem 1.3]{Z1}. 
\end{proof}

\begin{cor}\label{cor:strictly-semi-stable-hypercovering} Let $K$ be a $p$-adic non-archimedean discretely valued field, let $\X$ be an admissible formal $\O_K$-scheme with generic fiber $\X_\eta$ of (pure) dimension $d$, and let $n$ be an integer. Then there is a finite extension $K\subset L$, and a rig-surjective hypercovering $a\colon \X_\bullet\to \X_{\O_L}$ such that, for each $i\leq n$, there is a subfield $K\subset K_i\subset L$ and algebraically strictly semi-stable formal $\O_{K_i}$-scheme $\cY_i$ such that $\cY_{i, \O_L}\simeq \X_i$ and $\cY_{i, \eta}$ are of (pure) dimension $d$ for $i\leq n$.
\end{cor}
\begin{proof}
    The proof is similar to that of \cite[Theorem 4.16]{conrad-hyper}. The essential point is to show that every admissible formal $\O_K$-scheme $\X$ admits a rig-surjective covering by a strictly semi-stable formal scheme after a finite extension of $\O_K$. This was already done in Lemma~\ref{lemma:strictly-semi-stable-covering}.
\end{proof}

Now suppose that $\X$ is an admissible formal $\O_K$-scheme. Then Definition~\ref{defn:nearby-cycles} produces the functor
\[
\rm{R}\Psi_{\X}\colon D(\X_\eta; \Lambda) \to D(\X_s \times_s \eta; \Lambda),
\]
for any admissible simplicial formal $\O_K$-scheme $\X$. We extend it to simplicial admissible formal schemes via the formalism of simplicial topoi from \cite[\href{https://stacks.math.columbia.edu/tag/09WB}{Tag 09WB}]{stacks-project}. \smallskip

\begin{lemma}\label{lemma:trivial-action-good-model} Let $K\subset L$ be a finite extension of non-archimedean fields and let $\X$ be an algebraically strictly semi-stable formal $\O_K$-scheme.
Then
\begin{enumerate}
    \item $g-1$ acts trivially on $\pi^*_{\X_{s'}}\rm{R}^j\Psi_{\X_{\O_L}} \Lambda$ for any $g\in I_L$ and $j\geq 0$;
    \item if $k$ is a finite field, $\rm{R}\Psi_{\X_{\O_L}} \Q_\ell$ is mixed of weights $\leq \dim \X_\eta$, and $\geq -\dim \X_\eta$ (see Definition~\ref{defn:pure-mixed-Deligne}).
\end{enumerate}
\end{lemma}
In the proof of this lemma, we will freely use that $\rR\Psi_{\X} \Lambda \in D^b_{ctf}(\X_s\times_s \eta; \Lambda)$. This is guaranteed by Remark~\ref{rmk:finiteness-char-p} and its evident $\Z_\ell$ and $\Q_\ell$-versions. 
\begin{proof}
    By Lemma~\ref{lemma:compute-nearby-cycles}(\ref{lemma:compute-nearby-cycles-3}), we see that
    \[
    \rm{R}\Psi_{\X_{\O_L}} \Lambda \simeq b^*\rm{R}\Psi_{\X}\Lambda
    \]
    in $D(\X_{s'}\times_{s'}\eta')$, where $b\colon \X_{s'}\times_{s'} \eta' \to \X_{s}\times_s \eta$ is the natural morphism of Deligne's topoi. In particular,
    \[
    \pi^*_{\X_s}\rm{R}\Psi_{\X} \Lambda \simeq \pi^*_{\X_{s'}}\rm{R}\Psi_{\X_{\O_L}}\Lambda
    \]
    compatible with the $I_L$-action. Therefore, for the purpose of proving $(1)$, it suffices to prove the claim for $K=L$ and $\X$ an algebraically strictly semi-stable formal $\O_K$-scheme. Likewise, for the purpose of proving $(2)$, we can do same reduction by Lemma~\ref{lemma:mixed-after-base-change}. \smallskip
    
    By definition $\X$ is the $\varpi$-adic completion of a strictly semi-stable $\O_K$-scheme $X$. So we use the comparison of analytic and algebraic nearby cycles (see Theorem~\ref{thm:comparison}) to reduce the question to showing that, for a strictly semi-stable $\O_K$-scheme $X$, the element $(g-1)$ acts trivially on $\pi^*_X\rm{R}^j\Psi^{\rm{alg}}_X \Lambda$ for each $g\in I_K$ and $j\geq 0$, and $\rm{R}\Psi^{\rm{alg}}_X \Q_\ell$ is mixed of weights $\leq \dim X_\eta$ and $\geq -\dim X_\eta$ (we note that $\dim X_\eta=\dim \X_\eta$ as both are equal to $\dim X_s$). This essentially follows from the explicit computation of nearby cycles for strictly semi-stable schemes in \cite{Saito}. \smallskip 
    
    Namely, the first part is exactly \cite[Proposition 1.1]{Saito}. For the second part, we note that \cite[Proposition 1.2(2) and Corollary 1.3(1)]{Saito} imply that $\rm{R}^j\Psi^{\rm{alg}}_X \Q_\ell$ is mixed of weight $\leq 2j$. This already implies that $\rm{R}\Psi^{\rm{alg}}_X\Q_\ell$ is mixed. We are only left to show that it is mixed of weights $\leq \dim X_\eta$. Now note that $\rm{R}^j\Psi^{\rm{alg}}_X\Q_\ell=0$ for $j> \dim X_\eta$ by the Artin-Grothendieck Vanishing Theorem (see \cite[Corollary 7.5.2]{Lei-Fu} and its evident extension to $\Q_\ell$-coefficients). Therefore, we conclude that $\rm{R}\Psi^{\rm{alg}}_X\Q_\ell$ is mixed of weights\footnote{See \cite[Definition III.12.3]{FK} for the numerology around weights.} less or equal to
    \[
    \rm{max}_j\left(w\big(\rm{R}^j\Psi^{\rm{alg}}_X\Q_\ell\big)-j\right)=\dim X_\eta.
    \]
    Now we note that $\rm{R}\Psi^{\rm{alg}}$ commutes with Verdier duality by \cite[Corollary 3.8]{cat-traces} (and its evident extension to the $\Q_\ell$-case). Or, in other words, 
    \[
    \rm{R}\Psi^{\rm{alg}}_X\bf{D}_{X_\eta}(\Q_\ell) \simeq \bf{D}_{X_s\times_s \eta}(\rm{R}\Psi^{\rm{alg}}_X \Q_\ell).
    \]
    Since $X$ is regular, we can separately pass to each connected component to assume that $X_\eta$ is pure of dimension $d$ for some integer $d\geq 0$. Then $\bf{D}_{X_\eta}(\Q_\ell)\simeq \Q_\ell(d)[2d]$ and, therefore, 
    \[
    \bf{D}_{X_s\times_s \eta}(\rm{R}\Psi^{\rm{alg}}_X \Q_\ell)\simeq \rm{R}\Psi_{X}^{\rm{alg}}\Q_\ell(d)[2d] \leq d -2d+2d=d
    \]
    by the established above inequality on weights (see also \cite[Remark on p. 131]{KW}). Therefore, $\rm{R}\Psi^{\rm{alg}}_X \Q_\ell\geq -d$. 
\end{proof}

For the later use, we will also need the following lemma about weights on the \'etale cohomology groups of algebraically strictly semi-stable formal $\O_K$-schemes:

\begin{lemma}\label{lemma:weights-strictly-semistible} Let $K$ be a local field and let $\X$ be an algebraically strictly semi-stable formal $\O_K$-scheme of dimension $d$. Then, for any $g\in G_\eta$ projecting to the geometric Frobenius in $G_s$ and any integer $i\geq 0$, the eigenvalues of $g$ acting on $\rm{H}^i(X_{\wdh{\ov{\eta}}}, \Q_\ell)$ are $q$-Weil numbers of weights $\geq 0$.
\end{lemma}
\begin{proof}
    Let us denote by the irreducible components of $\X_s$ by $D_1, \dots, D_m$. For a non-empty subset $I\subset \{1, \dots, m\}$, we put $\X_{s, I}\coloneqq \cap_{i\in I} D_i$, we also put 
    \[
    \X_{s}^{(n)}=\bigcup_{I\subset \{1, \dots, m\}, |I|=n+1} \X_{s, I}.
    \]
    
    We first start with the action of $G_\eta$ on $\rm{H}^i(X_{\wdh{\ov{\eta}}}, \Q_\ell)$. The proof of \cite[Corollary 2.8(1), (2)]{Saito} and the identification of the cohomology of the nearby cycle and the cohomology of generic fiber (see Remark~\ref{rmk:action-coincide} and its evident extension to  $\Q_\ell$-cohomology) construct a spectral sequence
    \[
    \rm{E}^{n, m}_1 = \bigoplus_{j\geq \rm{max}(0, -n)} \rm{H}^{m-2j}(\X_{\ov{s}}^{(n+2j)}, \Q_\ell(-j)) \Longrightarrow \rm{H}^{n+m}(\X_{\wdh{\ov{\eta}}}, \Q_\ell). 
    \]
    Therefore, it suffices to show that all eigenvalues of the action of any Frobenius-lift on 
    \[
    \rm{H}^{m-2j}\left(\X_{\ov{s}}^{(n+2j)}, \Q_\ell(-j)\right)=\rm{H}^{m}\left(\X_{\ov{s}}^{(n+2j)}, \Q_\ell(-j)[-2j]\right)
    \]
    are $q$-Weil number of weights $\geq 0$ for any $n, m, i$. Since $\X_{s}^{(n+2j)}$ is smooth, $\Q_\ell(-j)[-2j]$ is pure of weight $0$. Therefore, $\rm{H}^{m}\left(\X_{\ov{s}}^{(n+2j)}, \Q_\ell(-j)[-2j]\right)$ is mixed of weights $\geq m$ by Weil conjectures (see \cite[Stabilit\'es 5.1.14(i*)]{BBD}). 
\end{proof}

Now we discuss some consequences of Lemma~\ref{lemma:trivial-action-good-model}.

\begin{lemma}\label{lemma:trivial-action-any-model} Let $K$ be a discretely valued $p$-adic non-archimedean field and let $\X$ be an admissible formal $\O_K$-scheme. Then
\begin{enumerate}
    \item there is an open subgroup $I_1\subset I$ (independent of $\Lambda$ and $\ell\neq p$) such that, for each $j\geq 0$, there is an integer $N_j$ such that $(g-1)^{N_j}$ acts trivially on $\pi^*_{\X_s}\rm{R}^j\Psi_{\X} \Lambda$ for any $g\in I_1$;
    \item if $k$ is a finite field, $\rm{R}\Psi_{\X} \Q_\ell$ is mixed.
\end{enumerate}
\end{lemma}
\begin{proof}
    Firstly, we note that rigid-analytic Artin-Grothendieck vanishing (see \cite[Theorem 7.3]{arc-topology}, \cite[Theorem 1.3]{Hansen-vanishing}, or \cite[Theorem 1.10]{GZ}) implies 
    \[
    \rm{R}^j\Psi_{\X} \Lambda = 0
    \]
    for $j> d\coloneqq \dim \X_\eta$. Therefore, it suffices to prove the claim for $j\leq d$. Both claims can be checked after a finite extension of $K$ (see Lemma~\ref{lemma:mixed-after-base-change} for mixedness), so Lemma~\ref{lemma:strictly-semi-stable-covering} and Lemma~\ref{lemma:trivial-action-good-model} ensure that we can assume that $\X$ admits a rig-surjective hypercovering 
    \[
    a\colon \cY_{\bullet} \to \X
    \]
    such that $\rm{R}^j\Psi_{\cY_n} \Q_\ell$ is mixed and $(g-1)$ acts trivially on $\rm{R}^j\Psi_{\cY_n} \Lambda$ for each $g\in I_K$, $j\geq 0$, $n\leq d$. \smallskip
    
    Now Lemma~\ref{lemma:hyperdescent} implies that
    \[
    \Lambda \simeq \rm{R}a_{\eta, *}a_{\eta}^* \Lambda.
    \]
    So 
    \[
    \rm{R}\Psi_{\X} \Lambda \simeq \rm{R}\Psi_{\X} \rm{R}a_{\eta, *} \Lambda \simeq \rm{R}(a_{s}\times_s\eta)_{ *}\rm{R}\Psi_{\cY_\bullet} \Lambda.
    \]
    Therefore, we can use the Grothendieck spectral sequence
    \[
    \rm{E}_2^{i, j} = \rm{R}^i (a_{s}\times_s\eta)_{ *}\rm{R}^j\Psi_{\cY_\bullet}\Lambda \Longrightarrow \rm{R}^{i+j}\Psi_{\X}\Lambda.
    \]
    So it suffices to show that, for each $i+j\leq d$,
    \begin{enumerate}
        \item there is an integer $M_{i, j}$ such that $(g-1)^{M_{i, j}}$ acts trivially on $\pi^*_{\X_s} \rm{R}^i (a_{s}\times_s\eta)_{ *}\rm{R}^j\Psi_{\cY_\bullet}\Lambda$;
        \item each $\rm{R}^i (a_{s}\times_s\eta)_{ *}\rm{R}^j\Psi_{\cY_\bullet}\Q_\ell$ is mixed if $k$ is a finite field.
    \end{enumerate}
    Now we use \cite[\href{https://stacks.math.columbia.edu/tag/0D7A}{Tag 0D7A}]{stacks-project} to get a spectral sequence
    \[
    \rm{E}_1^{n, m} = \rm{R}^m (a_{n, s}\times_s\eta)_{ *}\rm{R}^j\Psi_{\cY_n}\Lambda \Longrightarrow \rm{R}^{n+m} (a_{s}\times_s\eta)_{ *}\rm{R}^j\Psi_{\cY_\bullet}\Lambda,
    \]
    where $a_n\colon \cY_n\to \X$ is the augmentation morphism. Since each $a_{n, s}$ is of finite type, \cite[Theorem I.9.4]{KW} and Lemma~\ref{lemma:pushforward-pullback}(\ref{lemma:pushforward-pullback-2}) (and Remark~\ref{rmk:adic-good}) imply that $\rm{R}^m (a_{n, s}\times_s\eta)_{ *}$ preserves mixed complexes and triviality of an action, so it suffices to show that, for each $n+j+m\leq d$ (both claims below do not depend on $n$ and $m$ though),
    \begin{enumerate}
        \item there is an integer $M_n$ such that $(g-1)^{M_{n}}$ acts trivially on $\pi^*_{\X_s} \rm{R}^j\Psi_{\cY_n}\Lambda$;
        \item $\rm{R}^j\Psi_{\cY_n}\Q_\ell$ is mixed if $k$ is a finite field.
    \end{enumerate}
    This follows from our assumption on the hypercovering $\cY_\bullet \to \X$. 
\end{proof}

\begin{cor}\label{cor:trivial-action-any-model} Let $K$ be a discretely valued $p$-adic non-archimedean field and let $\X$ be an admissible formal $\O_K$-model. Then there is an open subgroup $I_1\subset I$ (independent of $\Lambda$ and $\ell$) such that, for each $j\geq 0$, there is an integer $N$ such that $(g-1)^{N}$ acts trivially on $\pi^*_{\X_s}\rm{R}\Psi_{\X} \Lambda$ for any $g\in I_1$.
\end{cor}
\begin{proof}
    We choose $I_1$ as in Lemma~\ref{lemma:trivial-action-any-model}, and denote by $N'\coloneqq \max_{i=1, \dots, d}(N_i)$ where $d=\dim \X_\eta$. We also set up 
    \[
    N=d\cdot \max_{i=1,\dots, d}(N_i).
    \]
    Then, for any $g\in I_1$, $(g-1)^{N'}$ acts trivially on each $\pi^*_{\X_s} \rm{R}^j\Psi_{\X} \Lambda$ by the choice of $N'$ and the fact that $\rm{R}^j\Psi_{\X} \Lambda=0$ for $j>d$ (see \cite[Theorem 7.3]{arc-topology} and \cite[Theorem 1.3]{Hansen-vanishing}). Therefore, $((g-1)^{N'})^d=(g-1)^N$ acts trivially on $\pi^*_{\X_s} \rm{R}\Psi_{\X} \Lambda$.
\end{proof}

\subsection{Nearby cycles of the intersection complex}\label{section:nearby-IC}

The main goal of this section is to show a version of the Grothendieck's Local Monodromy Theorem for both (compactly supported) cohomology and intersection cohomology of a qcqs rigid-analytic variety. \smallskip

Throughout this section, we fix a $p$-adic non-archimedean field $K$ and a prime number $\ell \in \O_K^\times$. \smallskip
 
We recall that \cite[Construction 4.12]{Bhatt-Hansen} defines the notion of an IC-sheaf $\rm{IC}_{X, \Lambda}$ for any qcqs rigid-analytic $K$-space $X$ and a coefficient ring $\Lambda\in \{\Z/\ell^n\Z, \Q_\ell\}$. To define $\rm{IC}_{X,\Lambda}$, we fix an dense Zariski-open subset $U\subset X$ such that $U_{\rm{red}}$ is smooth and define $\rm{IC}_{X,\Lambda}\coloneqq j_{!*}\Q_\ell[d_{U}]$ where $d_U\colon |U_{\rm{red}}| \to \Z$ is the dimension function\footnote{The dimension function on a smooth rigid-analytic space is locally constant, so it makes sense to shift a complex by $d_U$.}.   

\begin{lemma}\label{lemma:ic-subquotient} Let $K$ be a $p$-adic non-archimedean field, let $X$ be a reduced irreducible qcqs rigid-analytic $K$-variety of pure dimension $d$,  let $f\colon X' \to X$ be a resolution of singularities that is an isomorphism on $U=X^{\rm{sm}}$, and let $\Lambda=\Z/\ell^n\Z$ or $\Q_\ell$. Then $\rm{IC}_{X, \Lambda}$ is one of the simple perverse subquotients of $^p\rm{R}^0f_*(\Lambda[d])$.
\end{lemma}
\begin{proof}
    Let $\F_i$ be a finite set of all simple perverse subquotients of $^p\rm{R}^0f_*\left(\Lambda[d]\right)$. By \cite[Theorem 4.2, Theorem 4.11]{Bhatt-Hansen}, each $\F_i$ is isomorphic to $j_{!*} \cal{L}[d]$ for some Zariski locally-closed $j\colon V\to X$ and a simple locally constant $\Lambda$-sheaf $\cal{L}$ on $V$.\smallskip
    
    Now we take $U\subset X$ to be a non-empty Zariski open subset such that $f$ is an isomorphism over $U$. Then $\big({}^p\rm{R}^0f_*(\Lambda[d])\big)|_U\simeq \big(\Lambda[d]\big)|_U$. So there is a unique $\F_i$ among simple perverse subquotients of $^p\rm{R}^0f_*(\Lambda[d])$ such that $\F_i|_{U}\simeq \Lambda_{U}[d] \simeq \rm{IC}_{X, \Lambda}|_U$. Now one can argue as in algebraic geometry (see \cite[Corollary III.5.4]{KW}) to show that $\F_i \simeq \rm{IC}_{X, \Lambda}$.
\end{proof}

\begin{lemma}\label{lemma:ic-subquotient-smooth} Let $K$ be a $p$-adic non-archimedean field and let $f\colon X'\to X$ be an alteration of smooth connected finite type $K$-schemes of pure dimension $d$. Then $\Q_{\ell, X}[d]$ is one of the simple perverse subquotients of $^p\rm{R}^0f_*(\Q_{\ell, X'}[d])$.
\end{lemma}
\begin{proof}
    The proof is essentially identical to that of Lemma~\ref{lemma:ic-subquotient}. The only new difference is to show that there is a dense non-empty open $U\subset X$ such that $\Q_{\ell, U}[d]$ is subquotient of $({}^p\rm{R}f_*\Q_{\ell, X'}[d])|_U$. Any alteration is generically finite by definition. Moreover, $f$ is generically smooth since it is a morphism of smooth finite type schemes over a field of characteristic $0$. Therefore, we can choose $U$ to be a non-empty open locus where $f$ is finite \'etale. Therefore, the question is reduced to showing that $\Q_{\ell, X}$ is a subquotient of $f_* \Q_{\ell, X'}$ for a finite \'etale $f\colon X'\to X$. This follows from the existence of the trace map $\rm{tr}_f\colon f_* \Q_{\ell, X'} \to \Q_{\ell, X}$ since the composition
    \[
    \Q_{\ell, X} \to f_* \Q_{\ell, X'}\xr{\rm{tr}_f} \Q_{\ell, X}
    \]
    is the multiplication by $\deg f$. 
\end{proof}

We use Lemma~\ref{lemma:ic-subquotient} to show that the action of inertia on the nearby cycles of the IC-complex is always quasi-unipotent: 

\begin{lemma}\label{lemma:ic-nearby-cycles-action} Let $K$ be a discretely valued $p$-adic non-archimedean field and let $\X$ be an admissible formal $\O_K$-scheme with adic generic fiber $\X_\eta$. Then there is an open subgroup $I_1\subset I$ (independent of $\Lambda$ and $\ell$) such that there is an integer $N$ such that $(g-1)^{N}$ acts trivially on $\pi^*_{\X_s}\rm{R}\Psi_{\X} \rm{IC}_{X, \Lambda}$ for any $g\in I_1$
\end{lemma}
\begin{proof}
    The topological invariance of the \'etale topos implies that one can replace $\X$ by $(\X, \O_{\X}/\rm{nil}(\X))$ to assume that $\X$ (and, therefore, $X$) are reduced. \smallskip
    
    Now we consider the normalization morphism $\mf\colon \cY \to \X$. Then $\mf_\eta\colon \cY_\eta \to \X_\eta$ is finite and an isomorphism over a Zariski-dense Zariski-open subset $V\subset \X_\eta$ by \cite[Theorem 2.1.2 and Theorem 2.1.3]{conrad}. Since $\mf_\eta$ is finite, it is both perverse and constructible exact, and so (arguing as in the proof of Lemma~\ref{lemma:ic-subquotient}) one sees that 
    \[
    \rm{IC}_{\X_\eta, \Lambda} \simeq \rm{R}\mf_{\eta, *}\rm{IC}_{\cY_\eta, \Lambda}\simeq \mf_{\eta, *}\rm{IC}_{\cY_\eta, \Lambda}.
    \]
    Therefore, Lemma~\ref{lemma:compute-nearby-cycles}(\ref{lemma:compute-nearby-cycles-2}) ensures that 
    \[
    \rm{R}\Psi_{\X}\rm{IC}_{\X_\eta, \Lambda} \simeq \rm{R}\mf_{s, *}\rm{R}\Psi_{\cY}\rm{IC}_{\cY_\eta, \Lambda}.
    \]
    So it suffices to prove the claim for $\cY$. In other words, we may and do assume that $\X$ is normal. In this case, $\O_\X(\X)$ is integrally closed in $\O_{\X_\eta}(\X_\eta)$, so every non-trivial idempotent in $\O_{\X_\eta}(\X_\eta)$ lives in $\O_\X(\X)$. Geometrically, this means that every connected component of $\X_\eta$ lifts to a connected component of $\X$, so it suffices to prove the claim for each connected component separately. \smallskip
    
    So we may assume that $X$ is normal and connected (and so is irreducible). Then \cite[Theorem 5.2.2]{Temkin-resolution} implies that there is a resolution of singularities
    \[
    f\colon X' \to \X_\eta
    \]
    that is an isomorphism on the (non-empty) smooth locus $\X_\eta^{\rm{sm}}$. By \cite[Lemma 8.4/4]{B}, we can extend it to a morphism of admissible formal $\O_K$-schemes
    \[
    \mathfrak{f}\colon \X' \to \X.
    \]
    Lemma~\ref{lemma:ic-subquotient} implies that $\rm{IC}_{\X_\eta, \Lambda}$ is a subquotient of 
    \[
    ^p\rm{R}^0\mathfrak{f}_{\eta, *} \Lambda[d].  
    \]
    Since $\pi^*_{\X_s} \rm{R}\Psi_{\X}$ is perverse $t$-exact by \cite[Theorem 4.2 and Theorem 4.11]{Bhatt-Hansen} and Lemma~\ref{lemma:compute-nearby-cycles}(\ref{lemma:compute-nearby-cycles-1}), we conclude that $\pi^*_{\X_s} \rm{R}\Psi_{\X} \rm{IC}_{\X, \Lambda}$ is a subquotient of a perverse sheaf
    \[
    ^p\cal{H}^0\left(\pi^*_{\X_s}\left(\rm{R}\Psi_{\X}\rm{R}\mathfrak{f}_{\eta, *} \Lambda[d]\right)\right)\simeq {}^p\cal{H}^0\left(\pi^*_{\X_s}\rm{R}(\mathfrak{f}_{s}\times_s \eta)_{*} \rm{R}\Psi_{\X'} \Lambda[d]\right).
    \]
   Corollary~\ref{cor:trivial-action-any-model} ensures that there is an open subgroup $I_1\subset I$ and integer $N$ such that, for any $g\in I_1$, $(g-1)^N$ acts as zero on $\pi^*_{\X'_s}\rm{R}\Psi_{\X'} \Lambda[d]$. Therefore, it formally implies that the same holds for 
   \[
   \pi^*_{\X_s}\rm{R}(\mathfrak{f}_{s}\times_s \eta)_{*} \rm{R}\Psi_{\X'} \Lambda[d] \simeq \rm{R}\mathfrak{f}_{\ov{s}, *} \pi^*_{\X'_s} \rm{R}\Psi_{\X'} \Lambda[d].
   \]
   And as a consequence, the same holds for the $I_1$-action on 
   \[
   {}^p\cal{H}^0\left(\pi^*_{\X_s}\rm{R}(\mathfrak{f}_{s}\times_s \eta)_{*}\rm{R}\Psi_{\X'} \Lambda[d]\right)\simeq {}^p\cal{H}^0\left(\pi^*_{\X_s}\left(\rm{R}\Psi_{\X}\rm{R}\mathfrak{f}_{\eta, *} \Lambda[d]\right)\right).
   \]
   Since $\pi^*_{\X_s} \rm{R}\Psi_{\X} \rm{IC}_{\X, \Lambda}$ is a perverse subquotient of ${}^p\cal{H}^0\left(\pi^*_{\X_s}\left(\rm{R}\Psi_{\X}\rm{R}\mathfrak{f}_{\eta, *} \Lambda[d]\right)\right)$, we conclude that the same claim holds for it.
\end{proof}

Now we discuss mixedness of the nearby cyles of the IC complex. The strategy is essentially the same as in the proof of Lemma~\ref{lemma:ic-nearby-cycles-action}:
we use Lemma~\ref{lemma:ic-subquotient} to reduce the case of a strictly semi-stable model that was already established in Lemma~\ref{lemma:trivial-action-good-model}:

\begin{lemma}\label{lemma:mixed-of-weight-d} Let $K$ be a $p$-adic local field and let $\X$ be an admissible formal $\O_K$-scheme with generic fiber $\X_\eta$ of dimension $d$. Then $\rm{R}\Psi_{\X}\rm{IC}_{\X_\eta, \Q_\ell}$ is mixed of weights $\leq 2d$ and $\geq 0$.
\end{lemma}
\begin{proof}
    Arguing as at the beginning of the proof of Lemma~\ref{lemma:ic-nearby-cycles-action}, we can reduce to the case of an adimssible formal $\O_K$-scheme with reduced, irreducible, normal generic fiber $\X_\eta$ of (pure) dimension $d$.\smallskip
    
    {\it Step~$1$. Smooth $\X_\eta$.} The question is local on $\X$, so we can assume that $\X=\Spf B$ is affine. Now we note that \cite[Theorem 3.1.3]{Temkin} (it essentially boils down to \cite[Th\'eorem\`e 7 on page 582 and Remarque 2(c) on p.588]{Elkik} and \cite[Proposition 3.3.2]{T0}) says that there is a flat, finitely presented $\O_K$-algebra $A$ such that $A_K$ is $K$-smooth, and the $\varpi$-adic completion $\wdh{A}$ is isomorphic to $B$. Therefore, using the comparison between analytic and algebraic nearby cycles (see Theorem~\ref{thm:comparison} and Remark~\ref{rmk:adic-good}), we conclude that it suffices to show that 
    \[
    \rm{R}\Psi_X^{\rm{alg}} \Q_\ell \in D^b_c(X_s\times_s \eta; \Q_\ell)
    \]
    is mixed of weights $\leq 2d$ and $\geq 0$ for a flat finitely presented $\O_K$-scheme $X$ with smooth generic fiber $X_\eta$ of dimension $d$. By Lemma~\ref{lemma:mixed-after-base-change}, it suffices to prove the claim after a finite extension of $K$. Therefore, \cite[Theorem 8.2]{DJ} ensures that, after a finite extension $K\subset L$, there is a generically \'etale alteration 
    \[
    f\colon X'\to X_{\O_L}
    \]
    such that $X'$ is strictly semi-stable over $\O_L$. By Lemma~\ref{lemma:mixed-after-base-change} and Lemma~\ref{lemma:compute-nearby-cycles}(\ref{lemma:compute-nearby-cycles-3}), we can replace $K$ with $L$ to assume that $X$ admits an alteration by a strictly semistable $\O_K$-model. Then Lemma~\ref{lemma:ic-subquotient-smooth} implies that $\rm{IC}_{X_\eta, \Q_\ell}=\Q_{\ell, X_\eta}[d]$ is a subquotient of ${}^p\cal{H}^0(\rm{R}f_{\eta, *}\Q_{\ell, X'_\eta}[d])$. Since the (algebraic) nearby cycles are perverse exact (see Lemma~\ref{lemma:perverse-t-structure-rationally} for the definition of the perverse $t$-structure on $D^b_c(X_s\times_s \eta; \Q_\ell)$ and \cite[Appendix]{BBDG} for the proof that $\rm{R}\Psi^{\rm{alg}}_X$ is perverse exact), we conclude that $\rm{R}\Psi^{\rm{alg}}_{X}\rm{IC}_{X_\eta, \Q_\ell}$ is a subquotient of \[
    {}^{p}\cal{H}^0(\rm{R}\Psi_{X}^{\rm{alg}} \rm{R}f_{\eta, *}\Q_{\ell, X'_\eta}[d]) \simeq {}^{p}\cal{H}^0(\rm{R}(f_{s}\times_s \eta)_{*}\rm{R}\Psi^{\rm{alg}}_{X'}\Q_{\ell, X'_\eta}[d]),
    \]
    where the last isomorphism follows from the properness of $f$. Now \cite[Stabilit\'es 5.1.7 and Proposition 5.3.1]{BBD} imply that a perverse subquotient of mixed sheaf of weights $\leq n$ is mixed of weight $\leq n$. Therefore, it suffices to show that 
    \[
    {}^{p}\cal{H}^0(\rm{R}(f_s\times_s \eta)_*\rm{R}\Psi^{\rm{alg}}_{X'}\Q_{\ell, X'_\eta}[d])
    \]
    is mixed of weights $\leq 2d$ and $\geq 0$. Now \cite[Th\'eor\`eme 5.4.1]{BBD} implies that it suffices to show that 
    \[
    \rm{R}(f_s\times_s \eta)_*\rm{R}\Psi^{\rm{alg}}_{X'}\Q_{\ell, X'_\eta}[d]
    \]
    is mixed of weights $\leq 2d$ and $\geq 0$. Now properness of $f_s$ and Weil conjectures imply (see \cite[Stabilit\'es 5.1.14]{BBD}) that it is sufficient to show that 
    \[
    \rm{R}\Psi^{\rm{alg}}_{X'}\Q_{\ell, X'_\eta}[d]
    \]
    is mixed of weight $\leq 2d$ and $\geq 0$. This follows from Lemma~\ref{lemma:trivial-action-good-model} (or, really, from the results from \cite{Saito} used in the proof of Lemma~\ref{lemma:trivial-action-good-model}). \smallskip
    
    {\it Step~$2$. Reduced irreducible $\X_\eta$.} We reduce this to the result of Step~$1$ using essentially the same strategy. We only point out the main diffierences. Firstly, we use \cite[Theorem 5.2.2]{Temkin-resolution} instead of \cite[Theorem 8.2]{DJ} to find a resolution of singularities
    \[
    f\colon X' \to \X_\eta
    \]
    that is an isomorphism on the (non-empty) smooth locus $\X_\eta^{\rm{sm}}$. By \cite[Lemma 8.4/4]{B}, we can extend it to a morphism of admissible formal $\O_K$-schemes
    \[
    \mathfrak{f}\colon \X' \to \X.
    \]
    Then we use Lemma~\ref{lemma:ic-subquotient} in place of Lemma~\ref{lemma:ic-subquotient-smooth} to ensure that $\rm{IC}_{\X_\eta, \Q_\ell}$ is a subquotient of 
    \[
    {}^p\rm{R}^0\mathfrak{f}_{\eta, *} \Q_\ell[d].  
    \]
    Then we use \cite[Theorem 4.11]{Bhatt-Hansen} in place of \cite[Appendix]{BBDG} to conclude that $\rm{R}\Psi_\X$ is perverse exact. Finally, we use Step~$1$ in place of Lemma~\ref{lemma:trivial-action-good-model}. The rest of the argument is the same. 
\end{proof}

Lemma~\ref{lemma:mixed-of-weight-d} essentially proves the crucial part of the $\ell$-adic conjecture \cite[Conjecture 4.15(1)]{Bhatt-Hansen}. Now we discuss the second part  of \cite[Conjecture 4.15(1)]{Bhatt-Hansen} that relates $\rm{R}\Psi_\X \rm{IC}_{\X_\eta, \Q_\ell}$ to $\rm{IC}_{\X_s, \Q_\ell}$. 

For the next definition, we fix a $p$-adic local field $K$ and a finite type $k$-scheme $Y$. 

\begin{defn}\label{defn:IC-Deligne} The {\it IC-sheaf} $\rm{IC}_{Y\times_s\eta; \Q_\ell}\in \rm{Perv}(Y\times_s \eta; \Q_\ell)$ is the intermediate extension (see Definition~\ref{defn:intermediate-extension-Deligne}) 
\[
\rm{IC}_{Y\times_s\eta, \Q_\ell} \coloneqq (j\times_s \eta)_{!*}\left(\Q_{\ell, U}[d_U]\right),
\]
for $j\colon U\hookrightarrow Y$ an open dense subscheme such that $U_{\rm{red}}$ is smooth and $d_U\colon U \to \bf{Z}$ the dimension function\footnote{The dimension function on a smooth finite type $k$-scheme is locally constant, so it makes sense to shift a complex by $d_U$.}. 

More generally, for an open dense subscheme $j\colon U\hookrightarrow Y$ such that $U_{\rm{red}}$ is smooth and a $\Q_\ell$-local system $\cal{L}$ on $U\times_s\eta$ (see Definition~\ref{defn:lisse-objects}), we define the associated {\it IC-sheaf}
\[
\rm{IC}_{Y\times_s\eta}(\cal{L})\coloneqq (j\times_s \eta)_{!*}\left(\cal{L}[d_U]\right)
\]
as the intermediate extension of $\cal{L}$. 
\end{defn}


Before we start the proof of \cite[Conjecture 4.15]{Bhatt-Hansen}, we need two preliminary lemmas.

\begin{lemma}\label{lemma:split} Let $K$ be a $p$-adic local field, let $Y$ be a finite type $k$-scheme, let $U\subset Y$ be an open dense subscheme such that $U_{\rm{red}}$ is smooth, and let $\F\in \rm{Perv}(Y\times_s \eta; \Q_\ell)$. Suppose that 
\begin{enumerate}
    \item $\pi^*_Y\F\cong \bigoplus_{j\in J}\rm{IC}_{Y_j}(\cal{M}_j)$ for some closed subschemes $Y_j\subset Y_{\ov{s}}$ and $\Q_\ell$-local systems $\cal{M}_j$ on $Y_j$;
    \item $\F|_{U\times_s \eta}= \cal{L}[d_U]$ for a $\Q_\ell$-local system $\cal{L}$ on $U\times_s \eta$.
\end{enumerate}
Then $\rm{IC}_{Y\times_s \eta}(\cal{L})$ is a direct summand of $\F$.
\end{lemma}
\begin{proof}
        In this proof, we will freely use Lemma~\ref{lemma:extension-pullback} without any further notice. In particular, we will freely use that $\pi^*_Y \rm{IC}_{Y\times_s \eta}(\cal{L}) \simeq \rm{IC}_{Y_{\ov{s}}}(\cal{L})$.\smallskip
        
        We start the proof by noting that the condition $\F|_{U\times_s \eta}=\cal{L}[d_U]$ implies that $\pi^*_Y\F|_{U_{\ov{s}}}\cong \pi_Y^*\cal{L}[d_U]$. Therefore, after some renaming, the first assumption on $\F$ can be rewritten as 
        \[
        \pi_Y^*\F \cong \rm{IC}_{Y_{\ov{s}}}(\cal{L})\oplus \bigoplus_{i\in I}\rm{IC}_{Z_i}(\cal{L}_i)
        \]
        for some closed subscheme $Z_i\subset X_{\ov{s}}\setminus U_{\ov{s}}$ and $\Q_\ell$-local systems $\cal{L}_i$ on $Z_i$. \smallskip
        
        We start the proof by considering the open immersion $j\colon U\hookrightarrow Y$ and the natural morphisms
        \[
        {}^p\cal{H}^0\left(\left(j\times_s \eta\right)_{!} \F|_{U\times_s \eta}\right)\xrightarrow{\alpha} \F \xrightarrow{\beta} {}^p\cal{H}^0\left(\rm{R}\left(j\times_s \eta\right)_{*} \F|_{U\times_s \eta}\right).
        \]
        We note that $\rm{Im}(\beta\circ \alpha)\simeq \rm{IC}_{Y\times_s \eta}(\cal{L})$, and put $\G\coloneqq \rm{Im}(\alpha)$. Then we observe that $\G$ comes with the natural inclusion $\iota \colon \G \hookrightarrow \F$ induced by $\alpha$ and the natural surjection
        \[
        \gamma \colon \G=\rm{Im}(\alpha) \twoheadrightarrow \rm{IC}_{Y\times_s \eta}(\cal{L}) = \rm{Im}(\beta\circ \alpha)
        \] 
        induced by $\beta$. 
        
        {\it Claim. $\gamma$ is an isomorphism.}  It suffices to show after applying $\pi^*_Y$ by Lemma~\ref{lemma:properties-Deligne}(\ref{lemma:properties-Deligne-3}). Let us denote by $i\colon Z \hookrightarrow Y$ the reduced closed complement of $U$. Since $i_{\ov{s}}^*$ is perverse right $t$-exact (see \cite[Lemma 4.1]{KW}) and $i_{\ov{s}}^*j_{\ov{s}, !}\big( \pi_Y^* \F|_{U\times_s \eta}\big)=0$, we conclude that 
        \[
        i_{\ov{s}}^*\big( \pi_Y^* {}^p\cal{H}^0\left(\left(j\times_s \eta\right)_{!} \F|_{U\times_s \eta}\right)\big) \in {}^pD^{\leq -1}(Y_{\ov{s}}; \Q_\ell).
        \]
        Using that $i_{\ov{s}}^*$ is perverse right $t$-exact once again, we conclude that $i_{\ov{s}}^* \pi_Y^*\G \in {}^pD^{\leq -1}(Y_{\ov{s}}; \Q_\ell)$. This implies that $\Hom(\pi_Y^* \G, \IC_{Z_i}(\cal{L}_i))=0$ for any $i\in I$. Therefore, we conclude that the composition
        \[
        \pi_Y^*\G \xhookrightarrow{\pi^*_Y(\iota)} \pi^*_Y \F \xrightarrowdbl{q} \bigoplus_{i\in I}\rm{IC}_{Z_i}(\cal{L}_i)
        \]
        is zero, where $q\colon \pi_Y^*\F \twoheadrightarrow \oplus_{i\in I} \IC_{Z_i}(\cal{L}_i)$ is the natural projection. Thus, we conclude that $\pi^*_Y(\iota) \colon \pi_Y^* \G \hookrightarrow \pi_Y^* \F$ induces an injection 
        \[
        \pi_Y^*\G \xhookrightarrow{\alpha_1} \rm{IC}_{Y_{\ov{s}}}(\pi_Y^*\cal{L})\subset \pi^*_Y\F.
        \]
        Now we use that $\rm{Hom}_{\Q_\ell}\big(\rm{IC}_{Z_i}(\cal{L}_i), \rm{IC}_{Y_{\ov{s}}}(\pi_Y^*\cal{L})\big)=0$ to conclude that 
        \[
        \pi^*_Y(\beta) \colon \pi_Y^* \F \to \pi^*_Y \big( {}^p\cal{H}^0\left(\rm{R}\left(j\times_s \eta\right)_{*} \F|_{U\times_s \eta}\right) \big)
        \]
        restricted on $\rm{IC}_{Y_{\ov{s}}}(\pi_Y^*\cal{L}) \subset \pi^*_Y\F$ induces an isomorphism
        \[
        \beta_1\colon \rm{IC}_{Y_{\ov{s}}}(\pi_Y^*\cal{L}) \xr{\sim} \rm{IC}_{Y_{\ov{s}}}(\pi_Y^*\cal{L}) \subset \pi^*_Y \big( {}^p\cal{H}^0\left(\rm{R}\left(j\times_s \eta\right)_{*} \F|_{U\times_s \eta}\right) \big).
        \]
        Therefore, we get a commutative diagram
        \[
        \begin{tikzcd}
            \pi^*_Y \G \arrow[r, hook, "\iota_1"]\arrow[rd, swap, "\pi^*_Y(\gamma)"] & \arrow{d}{\beta_1} \arrow[d, swap, "\wr"] \rm{IC}_{Y_{\ov{s}}}(\pi_Y^*\cal{L}) \\
            & \rm{IC}_{Y_{\ov{s}}}(\pi_Y^*\cal{L})
        \end{tikzcd}
        \]
        such that $\iota_1$ is injective. This implies that $\pi_Y^*(\gamma)$ and, therefore, $\gamma$ are injective. Since $\gamma$ was also surjective by construction, we conclude that $\gamma$ is an isomorphism. \smallskip
        
        The isomorphism $\G\xr{\sim} \rm{IC}_{Y\times_s \eta}(\cal{L})$ defines an injective morphism 
        \[
        a\colon \rm{IC}_{Y\times_s \eta}(\cal{L}) \hookrightarrow \F.
        \]
        Now we put $\cal{H}\coloneqq \rm{Im}(\beta)$. It comes with a natural injection
        \[
        \rm{IC}_{Y\times_s \eta}(\cal{L}) \xhookrightarrow{\gamma'} \cal{H}.
        \]
        An argument dual to the proof of Claim, implies that $\gamma'$ is an isomorphism. This gives a surjection
        \[
        b\colon \F \twoheadrightarrow \rm{IC}_{Y\times_s \eta}(\cal{L}).
        \]
        By construction, $b\circ a=\rm{id}$, so $\rm{IC}_{Y\times_s \eta}(\cal{L})$ is a direct summand of $\F$.
\end{proof}

\begin{lemma}\label{lemma:dense-opens-after-finite-cover} Let $k$ be a field, let $f\colon X\to Y$ be a finite morphism of finite type $k$-schemes, let $U\subset X$ be a dense open subscheme. Then there is a dense open subscheme $V\subset Y$ such that $f^{-1}(V)\subset Y$.
\end{lemma}
\begin{proof}
    First, we note that finiteness of $f$ implies that, for a generic point $\eta\in Y$, any point $\zeta\in f^{-1}(\{\eta\})$ is a generic point of $X$. Therefore, our assumptions on $f$ and $U$ imply that $Z\coloneqq f(X\setminus U)$ is a closed subset of $Y$ which does not contain any generic point of $Y$. Therefore, $V\coloneqq Y\setminus Z$ does the job. 
\end{proof}

Finally, we are ready to prove \cite[Conjecture 4.15]{Bhatt-Hansen}. 

\begin{thm}\label{thm:ell-adic-conjecture} Let $X$ be qcqs rigid-analytic variety over a $p$-adic local field $K$, and $\X$ an admissible formal $\O_K$-scheme with special fiber $\X_s$ of pure dimension $d$ and generic fiber $X=\X_\eta$. Then  
\begin{enumerate}
\item $\rm{R}\Psi_{\X} \rm{IC}_{X, \Q_\ell}$ is mixed of weights $\leq 2d$ and $\geq 0$;
\item $\rm{IC}_{\X_s\times_s \eta, \Q_\ell}$ is a direct summand of $\rm{gr}_{\rm{W}}^d \rm{R}\Psi_\X \rm{IC}_{X, \Q_\ell}$ the $d$-th graded piece of the weight filtration on $\rm{R}\Psi_\X \rm{IC}_{X, \Q_\ell}$ (see Theorem~\ref{thm:weight-filtration}).
\end{enumerate}
In particular, for any continuous section $\sigma\colon G_s\to G_\eta$ of the natural projection $G_\eta \to G_s$, $\rm{IC}_{\X_s, \Q_\ell}$ is a direct summand of the $d$-th graded piece of the weight filtration of $\sigma^*_{\X_s} \rm{R}\Psi_\X \rm{IC}_{X, \Q_\ell}$.
\end{thm}
\begin{proof}
    The first claim is already proven in Lemma~\ref{lemma:mixed-of-weight-d}. Thus Theorem~\ref{thm:weight-filtration} ensures that it makes sense to speak about the weight filtration on $\rm{R}\Psi_{\X} \rm{IC}_{X, \Q_\ell}$. \smallskip 

    The topological invariance of the \'etale topos implies that one can replace $\X$ by $(\X, \O_{\X}/\rm{nil}(\X))$ to assume that $\X$ (and, therefore, $X$) are reduced. \smallskip

    In order to show that $\rm{IC}_{\X_s\times_s \eta, \Q_\ell}$ is a direct summand of $\rm{gr}_{\rm{W}}^d \rm{R}\Psi_\X \rm{IC}_{X, \Q_\ell}$, it suffices to show that $\rm{gr}_{\rm{W}}^d \rm{R}\Psi_\X \rm{IC}_{X, \Q_\ell}$ satisfies the assumptions of Lemma~\ref{lemma:split} with $\cal{L}$ containing the constant sheaf $\Q_{\ell}$ as a direct summand. \smallskip

    {\it Step~$1$. $\pi_{\X_s}^*\rm{gr}_{\rm{W}}^d \rm{R}\Psi_\X \rm{IC}_{X, \Q_\ell}$ is a direct sum of $\rm{IC}$-sheaves.} To prove this, we choose any continuous section $\sigma\colon G_s \to G_\eta$ of the projection morphism $G_\eta\to G_s$. Then we consider the natural projection morphism
    \[
    b^*_{\X_s}\colon X_{\ov{s}, \et} \to X_\et.
    \]
    The perverse sheaf $\sigma^*_{\X_s}\rm{gr}_{\rm{W}}^d \rm{R}\Psi_\X \rm{IC}_{X, \Q_\ell}$ is pure, thus 
    \[
    b^*_{\X_s} \sigma_{\X_s}^*\rm{gr}_{\rm{W}}^d \rm{R}\Psi_\X \rm{IC}_{X, \Q_\ell} \simeq \pi^*_{\X_s} \rm{gr}_{\rm{W}}^d \rm{R}\Psi_\X \rm{IC}_{X, \Q_\ell}
    \]
    is a direct sum of IC-sheaves by \cite[Theorem III.10.6 and Corollary III.5.5]{KW}. \smallskip

    {\it Step~$2$. There is a dense open $\cal{U}\subset \X$ such that the $K$-adic space $\cal{U}_\eta$ is smooth of pure dimension $d$, the $k$-scheme $(\cal{U}_s)_{\rm{red}}$ is smooth, and the complex $\big(\rR\Psi_\X \IC_{X, \Q_\ell}\big)|_{\cal{U}_s\times_s \eta}$ is a $\Q_\ell$-local system concentrated in degree $-d$.} First, \cite[\href{https://stacks.math.columbia.edu/tag/00LC}{Tag 00LC}]{stacks-project} implies that $\X_s$ has finitely many embedded points. Therefore, we can find a dense open $\cU\subset \X$ such that $\cU_s$ has no embedded points. Then Proposition~\ref{prop:equi-dimension}(\ref{prop:equi-dimension-2}) implies that $\cU_\eta$ is of pure dimension $d$. Now Corollary~\ref{cor:good-opens}, Lemma~\ref{lemma:lisse-stratification}, and the fact that $\rR\Psi_\X \IC_{X, \Q_\ell}$ is a perverse sheaf imply that we can shrink further to assume that $\cU_\eta$ is smooth, that $\cU_s$ is smooth, and that $\big(\rR\Psi_\X \IC_{X, \Q_\ell}\big)|_{\cal{U}_s\times_s \eta}$ is a $\Q_\ell$-local system concentrated in degree $-d$. 

    Therefore, for the purpose of verifying the second assumption from Lemma~\ref{lemma:split} with $\cal{L}$ containing $\Q_\ell$ as a direct summand, we can assume that $X$ is smooth of pure dimension $d$ (so $\IC_{X, \Q_\ell}=\Q_\ell[d]$), $(\X_s)_{\rm{red}}$ is smooth of pure dimension $d$, and $\rR \Psi_{\X} \Q_\ell[d] = \cal{L}[d]$ for some $\Q_\ell$-local system $\cal{L}$. In this case, it suffices to show that there is a dense open formal subscheme $\cal{U}\subset \X$ such that 
    \begin{enumerate}
        \item $(\rR \Psi_{\X} \Q_{\ell})|_{\cal{U}} = \cal{L}|_{\cal{U}}$ is pure of weight $0$;
        \item the unit morphism $\Q_{\ell} \xr{\psi_\X} \rR\Psi_\X \Q_{\ell} = \cal{L}$ admits a splitting over $\cal{U}$.
    \end{enumerate}
    
    {\it Step~$3$. We show the existence of such $\cal{U}$.} The Reduced Fiber Theorem (see \cite[Theorem 2.1]{BLR4}) implies that there are a finite extension $K\subset L$ and a finite rig-isomorphism $g\colon \X' \to \X_{\O_L}$ such that $\X'$ has (geometrically) reduced special fiber. Let $b\colon \X_{s'} \times_{s'} \eta' \to \X_s \times_s \eta$ be the morphism from Notation~\ref{notation:field-extension}. Note that \cite[Lemma 2.1.5]{conrad} implies that $\X_{\eta'}$ is still smooth of pure dimension $d$, one also sees that $(\X_{s'})_{\rm{red}}$ is smooth of pure dimension $d$. Furthermore, the $\Q_\ell$-version of Lemma~\ref{lemma:compute-nearby-cycles}(\ref{lemma:compute-nearby-cycles-3}) implies that 
    \[
    b^* \rR\Psi_\X \Q_\ell[d] \simeq \rR\Psi_{\X_{\O_L}} \Q_\ell[d].
    \]
    Therefore, Lemma~\ref{lemma:dense-opens-after-finite-cover}, Lemma~\ref{lemma:mixed-after-base-change}, and Corollary~\ref{cor:splitting-after-field-extension} ensure that it suffices to prove the result for $\X_{\O_L}$, i.e., we can replace $K$ with $L$ and $\X$ with $\X_{\O_L}$. \smallskip

    Now we note that $\X'_s$ is of pure dimension $d$ by virtue of Proposition~\ref{prop:equi-dimension}(\ref{prop:equi-dimension-1}). Since the special fiber $\X'_s$ is (geometrically) reduced, there is a dense open $\cal{V}\subset \X'$ such that $\cal{V}_s$ is smooth over the residue field $k$. Now we observe that $\X' \to \X$ is surjective because it is rig-surjective, so Lemma~\ref{lemma:dense-opens-after-finite-cover} implies that we can replace $\X$ by a dense open formal subscheme to assume that $\X'_s$ is smooth over $k$. Since $\X'$ is flat over $\O_K$, we conclude that $\X'$ is a smooth formal $\O_K$-scheme. In this case, a combination of \cite[Theorem 3.1.3]{Temkin}, Theorem~\ref{thm:comparison}, and \cite[Proposition 9.2.3]{Lei-Fu} implies that the natural morphism
    \[
    \psi_{\X'} \colon \Q_{\ell} \xr{\psi_{\X'}} \rR\Psi_{\X'} \Q_{\ell}
    \]
    is an isomorphism. In particular, $\rR\Psi_{\X'} \Q_{\ell, \X_\eta}$ is pure of weight $0$. Lemma~\ref{lemma:compute-nearby-cycles} and the fact that $(g_s\times_s \eta)_*$ preserves pure complexes of weight $0$ imply that $\rR \Psi_{\X} \Q_{\ell} \simeq (g_{s}\times_s \eta)_* \rR \Psi_{\X'} \Q_{\ell}$ is pure of weight $0$. 

    We are only left to show that $\psi_\X$ admits a section over a dense open formal subscheme $\cU \subset \X$. Since $\psi_{\X'}$ is an isomorphism, this question boils down to showing that the unit morphism
    \[
    \Q_{\ell, \X_s \times_s \eta} \to (g_{s} \times_s \eta)_{*} \Q_{\ell, \X'_s\times_s \eta}
    \]
    admits a section over a dense open subset $\cal{U}\subset \X$. This follows immediately from Corollary~\ref{cor:section-finite-morphism}. 
\end{proof}

\subsection{Global results}

The main goal of this section is to derive some global results from the local results obtained in Section~\ref{section:nearby-constant} and Section~\ref{section:nearby-IC}. In particular, we discuss a local monodromy theorem and (a weak version of) the second part of the $\ell$-adic conjecture  \cite[Conjecture 4.15(ii)]{Bhatt-Hansen}. 

\begin{thm}\label{thm:local-monodromy-cohomology} Let $K$ be a discretely valued $p$-adic non-archimedean field, let $\ell\neq p$ be a prime number, let $\Lambda$ be a ring $\Z/\ell^n\Z$, $\Z_\ell$, or $\Q_\ell$, and let $X$ be a qcqs rigid-analytic variety over $K$. Then there is an open subgroup $I_1\subset I$ and an integer $N$ (both independent of $\ell\neq p$ and $\Lambda$) such that, for each $g\in I_1$, $(g-1)^N$ acts trivially on 
\[
\rm{H}^i(X_{\wdh{\ov{\eta}}}, \Lambda), \rm{H}^i_c(X_{\wdh{\ov{\eta}}}, \Lambda), \rm{IH}^i(X_{\wdh{\ov{\eta}}},\Lambda), \text{ and }\rm{IH}^i_c(X_{\wdh{\ov{\eta}}},\Lambda)
\]
for each integer $i$.
\end{thm}
\begin{proof}
    We start the proof by choosing an admissible formal $\O_K$-scheme $\X$ such that $\X_\eta=X$. Then Remark~\ref{rmk:action-coincide} and Remark~\ref{rmk:action-coincide-compact} (and their adic analogues) guarantee that 
    \[
    \rm{R}\Gamma(\X_{\wdh{\ov{\eta}}}, \Lambda) \simeq \rm{R}\Gamma(\X_{\ov{s}}, \pi_{\X_s}^*\rm{R}\Psi_{\X}\Lambda)
    \]
    and
    \[
    \rm{R}\Gamma_c(\X_{\wdh{\ov{\eta}}}, \Lambda) \simeq \rm{R}\Gamma_c(\X_{\ov{s}}, \pi_{\X_s}^*\rm{R}\Psi_{\X}\Lambda)
    \]
    compatible with the $G_\eta$-action. The same applies to the cohomology complex of $\rm{IC}_{X, \Lambda}$. Therefore, the result follows from Lemma~\ref{lemma:trivial-action-any-model} and Lemma~\ref{lemma:ic-nearby-cycles-action}. 
\end{proof}

Now we show a more refined version of Theorem~\ref{thm:local-monodromy-cohomology} for the action of $I$ on the cohomology groups $\rm{H}^i(X_{\wdh{\ov{\eta}}}; \Q_\ell)$. The next theorem will crucially use the formalism of simplicial schemes (and adic spaces) and their associated simplicial topoi. We refer to \cite[\href{https://stacks.math.columbia.edu/tag/09VI}{Tag 09VI}]{stacks-project} (and especially to \cite[\href{https://stacks.math.columbia.edu/tag/09WB}{Tag 09WB}]{stacks-project}, \cite[\href{https://stacks.math.columbia.edu/tag/0D94}{Tag 0D94}]{stacks-project}, and \cite[\href{https://stacks.math.columbia.edu/tag/0D93}{Tag 0D93}]{stacks-project}) for the foundational material on this subject. 

\begin{thm}\label{thm:unipotent-action-strong} Let $X$ be a qcqs rigid-analytic variety over a $p$-adic discretely valued field $K$, let $\ell\neq p$ be a prime number, and let $\Lambda$ be a ring $\Z/\ell^n\Z$, $\Z_\ell$, or $\Q_\ell$. Then there is an open subgroup $I_1\subset I$, independent of $\ell$ and $\Lambda$, such that, for all $g\in I_1$, and all integers $i$, $(g-1)^{i+1}=0$ on  $\rm{H}^i(X_{\wdh{\ov{\eta}}}, \Lambda)$.
\end{thm}
\begin{proof}
    Since $\rm{H}^i(X_{\wdh{\ov{\eta}}}; \Q_\ell)=\rm{H}^i(X_{\wdh{\ov{\eta}}}; \Z_\ell)\left[\frac{1}{\ell}\right]$ and $\rm{H}^i(X_{\wdh{\ov{\eta}}}; \Z_\ell)=\lim_n \rm{H}^i(X_{\wdh{\ov{\eta}}}, \Z/\ell^n\Z)$, it suffices to prove the claim for $\Lambda=\Z/\ell^n\Z$. \smallskip
    
    We put $d\coloneqq \dim X$ and note that  $\rm{H}^i(X_{\wdh{\ov{\eta}}}, \Z/\ell^n\Z)=0$ for $i>2d$ by virtue of \cite[Corollary 2.8.3]{H3}. Therefore, it suffices to prove the claim for $0\leq i\leq 2d$. \smallskip
    
    Then Corollary~\ref{cor:strictly-semi-stable-hypercovering} and Lemma~\ref{lemma:trivial-action-good-model} imply that there is a finite extension $K\subset L$ and a rig-surjective hypercovering $a\colon \cY_{\bullet}\to \X_{\O_L}$ such that, for each $g\in I_L$ and an integer $b$, the action of $g-1$ on $\pi^*_{\cY_{b, s}}\rm{R}^j\Psi_{\cY_b} \Lambda_{\cY_{b, \eta}}$ is trivial for $b\leq \dim \X_\eta$. Since the statement we are trying to prove is insensitive to a finite extension of $K$, we may and do assume that $K=L$, and so we have a rig-surjective hypercovering 
    \[
    a\colon \cY_\bullet \to \X
    \]
    with the properties as above. Note that the generic fiber $a_\eta\colon \cY_{\bullet,\eta}\to \X_\eta$ is a $v$-hypercovering, and so Lemma~\ref{lemma:hyperdescent} ensures that 
    \[
    \Lambda\to \rm{R}a_{\eta, *}a_{\eta}^*\Lambda
    \]
    is an isomorphism. Therefore, we conclude that
    \[
    \rm{R}\Psi_{\X} \Lambda \simeq \rm{R}\Psi_{\X}\rm{R}a_{\eta,*}\Lambda \simeq \rm{R}a_{\ov{s}, *}\rm{R}\Psi_{\cY_{\bullet}} \Lambda. 
    \]
    Applying the functor $\rm{R}\Gamma(\X_{\ov{s}}, -)$ to this isomorphism, we get a sequence of isomorphisms
    \begin{align*}
        \rm{R}\Gamma(\X_{\wdh{\ov{\eta}}}, \Lambda) &\simeq \rm{R}\Gamma(\X_{\ov{s}}, \pi^*_{\X_s}\rm{R}\Psi_{\X} \Lambda) \\
        & \simeq \rm{R}\Gamma(\X_{\ov{s}}, \pi^*_{\X_s} \rm{R}a_{\ov{s}, *}\rm{R}\Psi_{\cY_{\bullet}} \Lambda)\\
        &\simeq \rm{R}\Gamma(\cY_{\bullet, \ov{s}}, \pi^*_{\cY_{\bullet, s}}\rm{R}\Psi_{\cY_{\bullet}} \Lambda)
    \end{align*}
    compatible with the $G_\eta$-action. Now we use the Grothendieck spectral sequence
    \[
    \rm{E}^{i, j}_2 = \rm{H}^i\left(\cY_{\bullet, \ov{s}}, \pi^*_{\cY_{\bullet, s}}\rm{R}^j\Psi_{\cY_\bullet} \Lambda\right) \Longrightarrow \rm{H}^{i+j}\left(\X_{\wdh{\ov{\eta}}}, \Lambda\right)
    \]
    to see that it suffices to show that, for any $g\in I$, $g-1$ acts trivially 
    \[
    \rm{H}^i\left(\cY_{\bullet,\ov{s}}, \pi^*_{\cY_{\bullet, s}}\rm{R}^j\Psi_{\cY_\bullet} \Lambda\right)
    \]
    for any $i+j\leq 2d$. Now this action factors through 
    \[
    \rm{H}^i\left(\cY_{\bullet,\ov{s}}, (g-1)\pi^*_{\cY_{\bullet, s}}\rm{R}^j\Psi_{\cY_\bullet} \Lambda\right),
    \]
    so it suffices to show that $\rm{H}^i\left(\cY_{\bullet,\ov{s}}, (g-1)\pi^*_{\cY_{\bullet, s}}\rm{R}^j\Psi_{\cY_\bullet} \Lambda\right)$ is zero for any $i+j\leq 2d$. For this we use \cite[\href{https://stacks.math.columbia.edu/tag/09WJ}{Tag 09WJ}]{stacks-project} to get another spectral sequence
    \[
    \rm{E}^{n, m}_1=\rm{H}^m\left(\cY_{n, \ov{s}}, (g-1)\pi^*_{\cY_{n, s}}\rm{R}^j\Psi_{\cY_n} \Lambda\right) \Longrightarrow \rm{H}^{n+m}\left(\cY_{\bullet, \ov{s}}, (g-1)\pi^*_{\cY_{\bullet, s}}\rm{R}^j\Psi_{\cY_\bullet} \Lambda\right).
    \]
    So, after all, it suffices to show that 
    \[
    \rm{H}^m\left(\cY_{n, \ov{s}}, (g-1)\pi^*_{\cY_{n, s}}\rm{R}^j\Psi_{\cY_n} \Lambda\right) = 0 
    \]
    for $n+m+j\leq 2d$. Now we conclude that it is actually enough to show that 
    \[
    (g-1)\pi^*_{\cY_{n, s}}\rm{R}^j\Psi_{\cY_n} \Lambda= 0
    \]
    for $n\leq 2d$ and any $j\geq 0$. This now follows from our assumption on $\cY_\bullet$ finishing the proof. 
\end{proof}

Now we discuss the action of Frobenius on the (compactly supported) cohomology of qcqs rigid-analytic varieties.  

\begin{lemma}\label{lemma:positive-weighs-constant} Let $K$ be a local $p$-adic field, let $\ell\neq p$ be a prime number, and let $X$ be a qcqs rigid-analytic variety over $K$. Then, for any $g\in G_\eta$ projecting to the geometric Frobenius in $G_s$ and any integer $i\geq 0$, the eigenvalues of $g$ acting on $\rm{H}^i(X_{\wdh{\ov{\eta}}}; \Q_\ell)$ are $q$-Weil numbers of weights $\geq 0$.
\end{lemma}
\begin{proof}
    Similarly to the proof of Theorem~\ref{thm:unipotent-action-strong}, we see that Corollary~\ref{cor:strictly-semi-stable-hypercovering} and Lemma~\ref{lemma:weights-strictly-semistible} imply that there is a finite extension $K\subset L$ and a $v$-hypercovering 
    \[
    Y_\bullet \to X_L
    \]    
    such that, for any $g\in G_\eta$ projecting to the geometric Frobenius in $G_s$ and any integer $i\geq 0$, the eigenvalues of $g$ acting on $\rm{H}^i(Y_{n, \wdh{\ov{\eta}}}; \Q_\ell)$ are $q$-Weil numbers of weights $\geq 0$ for $n\leq 2\dim X$. Then we use \cite[\href{https://stacks.math.columbia.edu/tag/09WJ}{Tag 09WJ}]{stacks-project} to get a spectral sequence
    \[
    \rm{E}^{n, m}_1=\rm{H}^m\left(Y_{n, \wdh{\ov{\eta}}}, \Q_\ell\right) \Longrightarrow \rm{H}^{n+m}\left(X_{\wdh{\ov{\eta}}}, \Q_\ell\right).
    \]
    to conclude the same for the $G_\eta$-action on $\rm{H}^{i}\big(X_{\wdh{\ov{\eta}}}, \Q_\ell\big)$ for $i\leq 2\dim X$. Now \cite[Corollary 2.8.3]{H3} implies $\rm{H}^{i}\big(X_{\wdh{\ov{\eta}}}, \Q_\ell\big)=0$ for $i\geq 2\dim X+1$. This finishes the proof. 
\end{proof}

\begin{thm}\label{thm:global-weights} Let $K$ be a local $p$-adic field, let $\ell\neq p$ be a prime number, and let $X$ be a qcqs rigid-analytic variety over $K$. Then 
\begin{enumerate}
    \item\label{thm:global-weights-1} For any $g\in G_\eta$ projecting to the geometric Frobenius in $G_s$ and any integer $i\geq 0$, the eigenvalues of $g$ acting on $\rm{H}^i(X_{\wdh{\ov{\eta}}}; \Q_\ell)$ are $q$-Weil numbers of weights $\geq 0$;
    \item\label{thm:global-weights-2} For any $g\in G_\eta$ projecting to the geometric Frobenius in $G_s$ and any integer $i\geq 0$, the eigenvalues of $g$ acting on $\rm{H}^i_c(X_{\wdh{\ov{\eta}}}; \Q_\ell)$ are $q$-Weil numbers;
    \item\label{thm:global-weights-3} For any $g\in G_\eta$ projecting to the geometric Frobenius in $G_s$ and any integer $i$, the eigenvalues of $g$ acting on $\rm{IH}^i_c(X_{\wdh{\ov{\eta}}}; \Q_\ell)$ are $q$-Weil numbers of weights $\leq 2d+i$;
    \item\label{thm:global-weights-4} For any $g\in G_\eta$ projecting to the geometric Frobenius in $G_s$ and any integer $i$, the eigenvalues of $g$ acting on $\rm{IH}^i(X_{\wdh{\ov{\eta}}}; \Q_\ell)$ are $q$-Weil numbers of weights $\geq i$.
\end{enumerate}
Furthermore, if the $\ell$-adic Decomposition theorem for rigid-analytic varieties holds (see \cite[Conjecture 4.17]{Bhatt-Hansen}), then weights of a Frobenius lift action on $\rm{IH}^i(X_{\wdh{\ov{\eta}}}; \Q_\ell)$ are $\geq \rm{max}(0, i)$. 
\end{thm}
\begin{proof}
    (\ref{thm:global-weights-1}) follows from Lemma~\ref{lemma:positive-weighs-constant}. (\ref{thm:global-weights-2}) follows from Lemma~\ref{lemma:trivial-action-any-model}, Remark~\ref{rmk:action-coincide-compact} (and its evident extension to $\Q_\ell$-coefficients), and \cite[Stabilit\'es 5.1.14]{BBD}. (\ref{thm:global-weights-3}) and (\ref{thm:global-weights-4}) follow from Lemma~\ref{lemma:mixed-of-weight-d}, Remark~\ref{rmk:action-coincide-compact}, Remark~\ref{rmk:action-coincide}, and \cite[Stabilit\'es 5.1.14]{BBD}. \smallskip{}

    Now we assume that the $\ell$-adic Decomposition theorem holds for a resolution of singularities $f\colon X' \to X$, then Lemma~\ref{lemma:ic-subquotient} and an argument similar to Lemma~\ref{lemma:split} imply that $\rm{IC}_{X, \Q_\ell}$ is a direct summand of $\rm{R}f_*\Q_\ell[d_{X}]$. Thus $\rm{IH}^i(X, \Q_\ell)$ is a direct summand of $\rm{H}^{i-d}(X', \Q_\ell)$, so the result follows from (\ref{thm:global-weights-1}). 
\end{proof}

\section{Local weight-monodromy conjecture}

\subsection{Overview}

Let $K$ be a local field of residue characteristic $p$, and $\ell\neq p$ a prime number. In this section, we study the following local analogue of the global weight-monodromy conjecture.

\begin{conjecture}\label{conj:local-weight-monodromy-nearby} (Local Weight-Monodromy Conjecture) Let $\X$ be an admissible formal $\O_K$-scheme with smooth generic fiber $\X_{\eta}$. Then the nearby cycles $\rm{R}\Psi_{\X} \Q_\ell \in D^{b}_c(\X_s \times_s \eta; \Q_\ell)$ are monodromy-pure of weight zero (see Definition~\ref{defn:monodromy-pure}).
\end{conjecture}

\begin{rmk} Conjecture~\ref{conj:local-weight-monodromy-nearby} implies that, for any flat finite type $\O_K$-scheme $X$ with smooth generic fiber $X_\eta$, the nearby cycles $\rm{R}\Psi^{\rm{alg}}_X \Q_\ell$ are monodromy pure of weight $0$. Indeed, Theorem~\ref{thm:comparison} implies that $\rm{R}\Psi_{X}^{\rm{alg}} \Q_\ell \cong \rm{R}\Psi_{\widehat{X}} \Q_\ell$, so the algebraic version follows immediately from the analytic one.
\end{rmk}

When $\mathrm{char}\,K = p$, the algebraic case of this conjecture is a classical result of Gabber, and the rigid analytic case can easily be deduced from this using Elkik's algebraization theorems. For the convenience of the reader, we discuss this reduction in the next section. \smallskip

In the mixed characteristic case, we prove a slightly weakened version of Conjecture~\ref{conj:local-weight-monodromy-nearby}. The essential idea is to use tilting equivalence and the approximation results to reduce the question to the equicharacteristic $p>0$ case treated by Gabber.

\subsection{Equi-characteristic case}\label{section:char-p-weight-monodromy}

For the rest of this section, we fix an equicharacteristic $p>0$ local field $K$ with ring of integers $\O_K$ and residue field $k$. Non-canonically, $K$ is isomorphic to $\bf{F}_q((T))$ for some finite extension $\bf{F}_p\subset \bf{F}_q$. We also fix a prime number $\ell\neq p$.

\begin{lemma}\label{lemma:char-p-local-weight-monodromy} Let $\X$ be an admissible formal $\O_K$-scheme with smooth generic fiber $\X_{\eta}$. Then the nearby cycles $\rR\Psi_{\X} \Q_\ell \in D^{b}_c(\X_s \times_s \eta,\Q_\ell)$ are monodromy-pure of weight zero.
\end{lemma}
\begin{proof}
    The question is clearly local on $\X$, so we may assume that $\X=\Spf B$ is a rig-smooth admissible affine formal $\O_K$-scheme. Choose a non-canonical isomorphism $\O_K\simeq \bf{F}_q[[T]]$ and denote by $\O\coloneqq \bf{F}_q[T]^{\rm{h}}_{(T)}$ the henselization of $\bf{F}_q[T]$ at the maximal ideal $(T)$. Then \cite[Theorem 3.1.3]{Temkin}\footnote{It is formulated under the additional hypothesis that $k^\circ$ is complete. However, the same proof works under the weaker assumption that $k^\circ$ is henselian.} that says that an affine rig-smooth formal scheme $\X$ can be algebraized to an affine flat finitely presented $\O_K$-scheme $Y=\Spec A$ with smooth generic fibre $Y_K$. In other words, there is an isomorphism $\wdh{A}\simeq B$.\smallskip
    
    Now a combination of Theorem~\ref{thm:comparison} and \cite[Th. finitude, Proposition 3.7]{SGA41/2} shows that $\rm{R}\Psi_{\X} \Q_\ell \simeq \rm{R}\Psi^{\rm{alg}}_Y \Q_\ell$. Therefore, it suffices to prove the result for $Y=\Spec A$ over $\Spec \O$. In this case, the result follows from Gabber's Theorem (see \cite[Theorem 5.1.2]{BB}\footnote{The shift by $-1$ occurs in the formulation of \cite[Theorem 5.1.2]{BB} due to a different normalization of the nearby cycles.}) and standard spreading out techniques. 
\end{proof}

\subsection{A non-standard tilting construction}

In this section, we explain a non-standard tilting construction. This is the essential tool to reduce questions about nearby cycles in mixed characteristic to analogous questions in positive characteristic. \smallskip

For the rest of this section, we fix a $p$-adic local field $K$ and a prime number $\ell\neq p$. We denote by $K\subset K_\infty$ its $p^{1/p^\infty}$-Kummer extension (see Definition~\ref{defn:standard-zp-extension}), and by $K^\flat$ its non-standard tilt (see Remark~\ref{rmk:non-standard-tilt}) with a fixed morphism $\a\colon K^\flat \to K^\flat_\infty$ realizing $K^\flat_\infty$ as a completed perfection of $K^\flat$. \smallskip

Let $\bf{D}_{K}^{d}=\Spa K \langle T_1,\dots,T_d \rangle$ be the usual $d$-dimensional affinoid ball over $K$, and similarly for $K_\infty$, $K^\flat_\infty$, and $K^\flat$. Let $\widetilde{\bf{D}}_{K_\infty}^{d}=\Spa K \langle T_1^{1/p^\infty},\dots,T_d^{1/p^\infty} \rangle$ be the $d$-dimensional perfectoid ball, and similarly for $K_\infty^\flat$. \smallskip 

We note that \cite[Proposition 5.20]{Sch0} ensures that $\left(\widetilde{\bf{D}}^d_{K_\infty}\right)^\flat \simeq \widetilde{\bf{D}}^d_{K_\infty^\flat}$. So  \cite[Theorem 7.12]{Sch0} implies that there is a natural equivalence of sites 
\[
    \Et\left(\widetilde{\bf{D}}^d_{K_\infty}\right) \simeq \Et\left(\widetilde{\bf{D}}^d_{K_\infty^\flat}\right).
\]
On the other hand, \cite[Proposition 2.3.7]{H3} implies that the natural morphism of sites\footnote{We follow the terminology of StacksProject and use \cite[\href{https://stacks.math.columbia.edu/tag/00X1}{Tag 00X1}]{stacks-project} for our definition of a morphism of sites. In particular, the actual functors of the underlying categories go in the opposite direction.}
\[
\Et\left(\bf{D}^d_{K^\flat_\infty}\right) \to \Et\left(\bf{D}^d_{K^\flat}\right)
\]
is an equivalence\footnote{\cite[Proposition 2.3.7]{H3} is formulated on the level of topoi, but it is not hard to see that it reduces to an equivalence of sites in our situation}. \smallskip

\begin{construction}\label{construction:gamma} We compose the above isomorphisms with the natural morphism of sites $\Et\left(\widetilde{\bf{D}}^d_{K_\infty}\right) \to \Et\left(\bf{D}^d_{K_\infty}\right)$ to get a morphism of sites
\[
\Et\left(\bf{D}^d_{K^\flat}\right) \simeq \Et\left(\widetilde{\bf{D}}^d_{K^\flat_\infty}\right) \simeq \Et\left(\widetilde{\bf{D}}^d_{K_\infty}\right) \to \Et\left(\bf{D}^d_{K_\infty}\right)
\]
that we denote by
\[
\gamma\colon \Et\left(\bf{D}^d_{K^\flat}\right) \to \Et\left(\bf{D}^d_{K_\infty}\right). 
\]
\end{construction}

\begin{rmk}\label{rmk:site-theoretic} For an \'etale morphism $f\colon X\to \bf{D}^d_{K_\infty}$, the pullback $\gamma^*(f)\in \Et\left(\bf{D}^d_{K^\flat}\right)$ is denoted by $f^\flat\colon X^\flat \to \bf{D}^d_{K^\flat}$. 

\end{rmk}

For the rest of the section, we fix a rigid-analytic $K_\infty$-variety $X$ with an \'etale morphism $f\colon X\to \bf{D}^d_{K_\infty}$. 

\begin{defn}\label{defn:non-standard-tilt} The {\it non-standard tilt} of $(X, f)$ is the pair $(X^\flat, f^\flat)$ of the rigid-analytic $K^\flat$-variety $X^\flat$ and the \'etale morphism $f^\flat\colon X^\flat \to \bf{D}^d_{K^\flat}$ defined in Remark~\ref{rmk:site-theoretic}.

\end{defn}

\begin{construction}\label{construction:gamma-2}  We apply Construction~\ref{construction:gamma} to the slice sites to get a natural morphism of sites \[
    \gamma\colon \Et\left(X^{\flat}\right) \to \Et\left(X\right). 
\]
It induces a morphism of the associated topoi
\[
    \gamma\colon X^\flat_\et \to X_{\et}.
\]
\end{construction}

\begin{variant}\label{construction:gamma-variant} One could instead consider the morphism $\Et\left(\widetilde{\bf{D}}^d_{K_\infty}\right) \to \Et\left(\bf{D}^d_K\right)$ in place of the morphism $\Et\left(\widetilde{\bf{D}}^d_{K_\infty}\right) \to \Et\left(\bf{D}^d_{K_\infty}\right)$ in the first line of Construction~\ref{construction:gamma}. Then the same approach would define a morphism of topoi
\[
\gamma'\colon X^\flat_\et \to X_\et
\]
for any adic space $X$ with an \'etale morphism $X\to \bf{D}^d_K$. 
\end{variant}

Construction~\ref{construction:gamma-2} is our main tool to approach  Conjecture~\ref{conj:local-weight-monodromy-nearby}. Namely, the construction of a non-standard tilting and the proposition below will later allow us to reduce the mixed characteristic version of Conjecture~\ref{conj:local-weight-monodromy-nearby} to the characteristic $p$ version that was already established in Section~\ref{section:char-p-weight-monodromy}.

\begin{lemma}\label{lemma:split-2} Let $X \to \bf{D}_{K_\infty}^{d}$ be an \'etale morphism, and let $\gamma\colon X^{\flat}_{\et} \to X$ be the morphism of topoi from Construction~\ref{construction:gamma-2}. Then for $\Lambda \in \{ \mathbf{Z}/\ell^n, \mathbf{Z}_\ell, \Q_\ell\}$, the natural adjunction $\mathrm{id} \to \rm{R}\gamma_{\ast} \gamma^{\ast}$ associated with the adjoint pair $(\gamma^{\ast}, \rm{R}\gamma_{\ast}) \colon \cal{D}(X_{\et}; \Lambda) \rightleftarrows \cal{D}(X^{\flat}_{\et}; \Lambda)$ is canonically split.
\end{lemma}
\begin{proof}
It suffices to treat the case $\Lambda = \mathbf{Z}/\ell^n$; the case $\Lambda=\mathbf{Z}_\ell$ and $\Lambda=\Q_\ell$ then follow by a simple limit argument. \smallskip

Let $\cal{F} \in \cal{D}(X_{\et}; \Lambda)$ be any object, and let $V \to X$ be an \'etale map from some qcqs $V$, with associated \'etale map $V^\flat \to X^\flat$.  It then suffices to split the map 
\[ 
    \rm{R}\Gamma(V,\mathcal{F}) \to \rm{R}\Gamma(V, \rm{R}\gamma_{\ast} \gamma^{\ast}\mathcal{F}) \simeq \rm{R}\Gamma(V^\flat, \gamma^{\ast}\mathcal{F})
\]
functorially in $\mathcal{F}$ and $V$. Let $V_i$ resp. $\widetilde{V}$ be the pullback of $V\to X \to \mathbf{D}_{K_\infty}^{d}$ along the map 
\[ 
    \Spa K_\infty \langle T_1^{1/p^i},\dots,T_d^{1/p^i} \rangle \to \mathbf{D}_{K_\infty}^{d}, 
\]
resp. along the map $ \widetilde{\mathbf{D}}_{K_\infty}^{d} \to \mathbf{D}_{K_\infty}^{d}$. Then the $V_i$'s form an inverse system of qcqs rigid spaces such that $\widetilde{V} \simeq  \lim_i V_i$ as diamonds. Let $\gamma_i\colon V_i \to V$ and $\tilde{\gamma}\colon \tilde{V} \to V$ be the evident maps. Unwinding the constructions, we see that
\begin{align*}
    \rm{R}\Gamma(V^\flat, \gamma^{\ast}\mathcal{F}) & \simeq \rm{R}\Gamma(\tilde{V}, \tilde{\gamma}^{\ast}\mathcal{F}) \\
        & \simeq \hocolim \rm{R}\Gamma(V_i, \gamma_{i}^{\ast}\mathcal{F}) \\
        & \simeq \hocolim \rm{R}\Gamma(V, \gamma_{i,\ast}\gamma_{i}^{\ast}\mathcal{F}).
\end{align*}
By the projection formula, we get $\gamma_{i,\ast}\gamma_{i}^{\ast}\mathcal{F} \simeq \mathcal{F} \otimes^L \gamma_{i,\ast}\gamma_{i}^{\ast}\Lambda$ functorially in $\mathcal{F}$, so it suffices to split the map $\Lambda \to \gamma_{i,\ast}\gamma_{i}^{\ast}\Lambda$ compatibly with varying $i$. But $V_i \to V$ is finite flat of constant degree $p^{di}$, so the renormalized trace map $\frac{1}{p^{di}} \mathrm{tr}\colon   \gamma_{i,\ast}\gamma_{i}^{\ast}\Lambda \to \Lambda$ does the job (see \cite[Theorem 2.5.6]{LRZ} for the construction of the finite flat trace map in the rigid-analytic setup). 
\end{proof}

\subsection{Mixed characteristic case}

For the rest of the section, we fix a $p$-adic local field $K$ and a prime number $\ell\neq p$. The main goal of this section is to give a proof of Conjecture~\ref{conj:local-weight-monodromy-nearby} under some extra assumption on the admissible formal model $\X$. \smallskip

Before we do this, we need a preliminary lemma.

\begin{lemma}\label{lemma:stable-circle} Let $X=\Spa(A, A^+)$ be a smooth affinoid over $K$. Then $A^\circ$ is topologically finitely generated $\O_K$-algebra. Furthermore, if $A^\circ\otimes_{\O_K} k$ is reduced, then the natural morphism
\[
A^\circ \wdh{\otimes}_{\O_K} \O_{K_\infty}\to (A\wdh{\otimes}_K K_\infty)^\circ
\] 
is an isomorphism. In particular, $(A\wdh{\otimes}_K K_\infty)^\circ$ is topologically finitely generated.
\end{lemma}
\begin{proof}
    The first claim follows directly from \cite[Corollary 6.4/5]{BGR}. Now suppose that $A^\circ\otimes_{\O_K} k$ is reduced. Then note that $A^\circ \wdh{\otimes}_{\O_K} \O_{K_\infty}$ is a ring of definition in $A\wdh{\otimes}_K K_\infty$ with a reduced special fiber. Thus \cite[Proposition 3.4.1]{L-Jac} implies that 
    \[
    A^\circ \wdh{\otimes}_{\O_K} \O_{K_\infty}\to (A\wdh{\otimes}_K K_\infty)^\circ
    \]
    is an isomorphism. 
\end{proof}

\begin{rmk}\label{rmk:not-p-adic} The first part of Lemma~\ref{lemma:stable-circle} holds for any local field $K$ (not necessarily $p$-adic).
\end{rmk}

Now we show the first general result in the mixed characteristic case. In the proof below, we denote the \'etale topoi of $\Spec K$, $\Spec K_\infty$, $\Spec K_\infty^\flat$, and $\Spec K^\flat$ by $\eta$, $\eta_\infty$, $\eta_\infty^\flat$, and $\eta^\flat$ respectively. We note that the topoi $\eta_\infty$, $\eta_\infty^\flat$, and $\eta^\flat$ are canonically equivalent. 

\begin{thm}\label{thm:first-weight-monodromy} Let $X=\Spa(A, A^+)$ be a smooth $K_\infty$-affinoid space. Suppose that $A^\circ$ is topologically finitely generated $\O_{K_\infty}$-algebra and $X$ admits an \'etale map to an affinoid ball $\bf{D}^d_{K_\infty}$. Then $\rm{R}\Psi_{\X} \Q_\ell \in D^b_c(\X_s\times_s \eta_\infty; \Q_\ell)$ is monodromy-pure of weight zero for the canonical admissible formal $\O_K$-model $\X=\Spf A^\circ$.
\end{thm}
\begin{proof}
    Firstly, we note that \cite[Corollary D.5]{Zav-coherent} (based on \cite[Proposition 6.6.1]{Ach1}) ensures that $X$ admits a {\it finite} \'etale morphism $f_{\eta_\infty} \colon X \to \bf{D}^d_{K_\infty}$. This morphism clearly extends to a morphism 
    \[
    f\colon \X \to \widehat{\bf{A}}^d_{\O_{K_\infty}}
    \]
    that is automatically finite by \cite[Theorem 6.4/1(iii)]{BGR} and the fact that an integral morphism of topologically finitely generated $\O_{K_\infty}$-algebras must be finite. We denote the special fiber of $f$ by $f_s \colon \X_s \to \bf{A}^d_s$. \smallskip
    
    Now we consider the non-standard tilt $f^\flat_{\eta^\flat}\colon X^\flat \to \bf{D}^d_{K^\flat}$ and the morphism of topoi
    \[
    \gamma\colon X^\flat_\et \to X_\et 
    \]
    from Definition~\ref{defn:non-standard-tilt} and Construction~\ref{construction:gamma-2} respectively. By construction, $f^\flat_{\eta^\flat}$ is finite \'etale, so $X^\flat=\Spa(B, B^{+})$ is affine. So $f$ extends to a finite morphism\footnote{Remark~\ref{rmk:not-p-adic} ensures that $B^\circ$ is topologically finite type and so $\Spf B^\circ$ is an admissible formal $\O_{K^\flat}$-scheme.} 
    \[
    f^\flat\colon \X^\flat=\Spf B^\circ \to \widehat{\bf{A}}^d_{\O_K^\flat}.
    \]
    We denote its special fiber by $f^\flat_s\colon \X_s^\flat \to \bf{A}^d_s$. Now comes the key observation: \smallskip
    
    {\it Claim: The diagram 
    \[
    \xymatrix{
                \X_{\eta^{\flat}, \et}^{\flat}\ar[r]^{\gamma} \ar[d]^{\Psi_{\X^{\flat}}} & \X_{\eta_\infty, \et}   \ar[d]^{\Psi_{\X}}\\
                \X_{s}^{\flat}\times_{s}\eta^{\flat} \ar[d]^{f^{\flat}_s\times_s \eta^\flat} & \X_{s}\times_{s}\eta_\infty \ar[d]^{f_s\times_s \eta_\infty}\\
                \mathbf{A}_{s}^{d}\times_{s}\eta^{\flat}  \ar[r]^{\simeq} & \mathbf{A}_{s}^{d}\times_{s}\eta_\infty
            }
    \]
    commutes (up to an equivalence).} 
    
    We will prove this claim later, but now we assume the claim and deduce Theorem~\ref{thm:first-weight-monodromy} from it. Firstly we note that Lemma~\ref{lemma:split-2} and {\it Claim} imply that $(f_s\times_s\eta_{\infty})_{\ast} \rm{R}\Psi_{\X} \mathbf{Q}_\ell$ canonically splits as a summand of $(f_s\times_s\eta_{\infty})_{\ast} \rm{R}\Psi_{\X} \rm{R}\gamma_{\ast} \mathbf{Q}_\ell\simeq (f^{\flat}_s\times_s \eta^\flat)_{\ast} \rm{R}\Psi_{\X^{\flat}}\mathbf{Q}_\ell$. \smallskip
    
    Since $f^\flat_s$ is finite, Lemma~\ref{lemma:finite-morphisms-monodromy-pure} ensures that $(f^{\flat}_s\times_s \eta^\flat)_{\ast}$ preserves monodromy-pure perverse sheaves of weight $0$. Therefore,  Lemma~\ref{lemma:char-p-local-weight-monodromy} implies that $(f^{\flat}_s\times_s \eta^\flat)_{\ast} \rm{R}\Psi_{\X^{\flat}}\mathbf{Q}_\ell$ is monodromy-pure of weight $0$. \smallskip
    
    We use finiteness of $f_s$ to ensure that $(f_s\times_s \eta_\infty)_{\ast}$ reflects monodromy-pure perverse sheaves of weight $0$ (also due to Lemma~\ref{lemma:finite-morphisms-monodromy-pure}). Therefore, $\rm{R}\Psi_{\X} \mathbf{Q}_\ell$ is monodromy-pure of weight $0$ because $(f_s\times_s\eta_{\infty})_{\ast} \rm{R}\Psi_{\X} \mathbf{Q}_\ell$ is a direct summand of $(f^{\flat}_s\times_s \eta^\flat)_{\ast} \rm{R}\Psi_{\X^{\flat}}\mathbf{Q}_\ell$ that was shown to be monodromy-pure of weight $0$.

    \begin{proof}[Proof of Claim] By the universal property of $2$-fiber products, it suffices to show the diagram
    \[
    \xymatrix{
                \X_{\eta^{\flat}, \et}^{\flat}\ar[r]^{\gamma}\ar[d]^{\lambda_{\X^{\flat}}} & \X_{\eta_\infty, \et}\ar[d]^{\lambda_{\X}}\\
                \X_{s, \et}^{\flat} \ar[d]^{f^{\flat}_s} & \X_{s, \et} \ar[d]^{f_s}\\
                \mathbf{A}_{s, \et}^{d}\ar[r]^{\sim} & \mathbf{A}_{s, \et}^{d}
            }
    \]
    commutes (up to an equivalence). \smallskip
    
    {\it Step~$1$.} We note that the diagram
    \[
    \begin{tikzcd}
        \X^\flat_{\eta^\flat, \et}\arrow{d}{f^\flat_{\eta^\flat}} \arrow{r}{\gamma} & \X_{\eta_\infty, \et} \arrow{d}{f_\eta} \\
        \bf{D}^d_{K^\flat} \arrow{r}{\gamma} & \bf{D}^d_{K_\infty}
    \end{tikzcd}
    \]
    commutes (up to an equivalence) by construction (see Construction~\ref{construction:gamma-2}). \smallskip
    
    {\it Step~$2$.} We note that functoriality of the morphism $\lambda$ implies that the diagram
    \[
    \begin{tikzcd}
        \X_{\eta_\infty, \et} \arrow{r}{\lambda_{\X}} \arrow{d}{f_{\eta_\infty}} & \X_{s, \et}    \arrow{d}{f_s} \\
        \bf{D}^d_{K_\infty, \et} \arrow{r}{\lambda_{\wdh{\bf{A}}^d}} & \bf{A}^d_{s, \et}
    \end{tikzcd}
    \]
    commutes (up to an equivalence) and the same diagram for $\X^\flat$ and $f^\flat$ also commutes (up to an equivalence). \smallskip
    
    {\it Step~$3$.} Steps~$1$, $2$ and a standard diagram chase imply that it suffices to show that the diagram
    \[
    \begin{tikzcd}
        \bf{D}^d_{K^\flat, \et} \arrow{r}{\lambda} \arrow{d}{\lambda_{\wdh{\bf{A}}^d_{\O_K^\flat}}} & \bf{D}^d_{K_\infty, \et} \arrow{d}{\lambda_{\wdh{\bf{A}}^d_{\O_{K_\infty}}}} \\
        \bf{A}^d_{s, \et} \arrow{r}{\sim} & \bf{A}^d_{s, \et}
    \end{tikzcd}
    \]
    commutes (up to an equivalence). By construction (see Construction~\ref{construction:gamma}), it boils down to showing that the diagram
    \begin{equation}\label{eqn:final-commutativity}
    \begin{tikzcd}
        \widetilde{\bf{D}}^d_{K^\flat_\infty, \et} \arrow{r}{\sim} \arrow{d} & \widetilde{\bf{D}}^d_{K_\infty, \et} \arrow{d} \\
        \bf{A}^d_{s, \et} \arrow{r}{\sim} & \bf{A}^d_{s, \et}
    \end{tikzcd}
    \end{equation}
    commutes (up to an equivalence), where the top arrow is the tilting equivalence. To see commutativity of (\ref{eqn:final-commutativity}), we choose an element $p^\flat \in K_\infty^\flat$ as in \cite[Lemma 3.4(ii)]{Sch0} and note that any \'etale $k[T_1, \dots, T_d]$-algebra $C$ uniquely lifts to a formally \'etale $\O_{K_\infty}\langle T_1, \dots, T_d\rangle$-algebra $C_{K_\infty}$ (resp. a formally \'etale $\O_{K_\infty^\flat}\langle T_1, \dots, T_d\rangle$-algebra $C_{K^\flat_\infty}$). Furthermore, we put 
    \[
    \widetilde{C}_{K_\infty} \coloneqq C_{K_\infty} \wdh{\otimes}_{\O_{K_\infty}\langle T_1, \dots, T_d\rangle} \O_{K_\infty}\langle T_1^{1/p^\infty}, \dots, T_d^{1/p^\infty}\rangle,
    \]
    \[
    \widetilde{C}_{K^\flat_\infty} \coloneqq C_{K^\flat_\infty} \wdh{\otimes}_{\O_{K^\flat_\infty}\langle T_1, \dots, T_d\rangle} \O_{K^\flat_\infty}\langle T_1^{1/p^\infty}, \dots, T_d^{1/p^\infty}\rangle.
    \]
    Then (the proof of) \cite[Corollary 7.4.5]{Zav-coherent}\footnote{Strictly speaking, the proof of \cite[Corollary 7.4.5]{Zav-coherent} assumes that the ground field is of characteristic $0$, but the proof easily adapts to characteristic $p>0$ situation as well.} implies that $(\widetilde{C}_{K_\infty}[\frac{1}{p}], \widetilde{C}_{K_\infty})$ and $(\widetilde{C}_{K^\flat_\infty}[\frac{1}{p^\flat}], \widetilde{C}_{K^\flat_\infty})$ are perfectoid pairs. Therefore, commutativity of (\ref{eqn:final-commutativity}) boils down to constructing a functorial isomorphism 
    \[
    \Spa\left(\widetilde{C}_{K_\infty}\left[\frac{1}{p}\right], \widetilde{C}_{K_\infty} \right)^\flat \simeq \Spa\left(\widetilde{C}_{K^\flat_\infty}\left[\frac{1}{p^\flat}\right], \widetilde{C}_{K^\flat_\infty} \right).
    \]
    Following the construction of the tilting correspondence (see \cite[Theorem 5.2]{Sch0}), it suffices to construct a functorial $\O_{K_\infty}/p\simeq \O_{K_\infty^\flat}/p^\flat$-linear isomorphism
    \[
    \widetilde{C}_{K_\infty}/p \simeq \widetilde{C}_{K^\flat_\infty}/p^\flat.
    \]
    This isomorphism comes from the fact that the  \'etale $k[T_1, \dots, T_d]$-algebra $C$ admits a unique (up to a unique isomorphism) lift to an \'etale $\O_{K_\infty}/p[T_1^{1/p^\infty}, \dots, T_d^{1/p^\infty}] \simeq \O_{K^\flat_\infty}/p^\flat[T_1^{1/p^\infty}, \dots, T_d^{1/p^\infty}]$-algebra and the fact that both $\widetilde{C}_{K_\infty}/p$ and $\widetilde{C}_{K^\flat_\infty}/p^\flat$ provide such lifts. 
    \end{proof}
\end{proof}

\begin{thm}\label{thm:local-weight-monodromy} Let $K$ be a $p$-adic local field, and let $\X$ be an admissible formal $\O_K$-scheme with smooth generic fiber. Suppose that for each point $x\in \X$ there is an \'etale morphism $(\sU, u) \to (\X, x)$ of pointed formal schemes such that $\sU_\eta$ admits an \'etale morphism to $\bf{D}^d_K$. Then the nearby cycles $\rm{R}\Psi_{\X} \Q_\ell \in D^{b}_c(\X_s \times_s \eta,\Q_\ell)$ are monodromy-pure of weight zero.
\end{thm}
\begin{proof}
    Firstly, we note that the claim is \'etale local on $\X$, so we may assume that $\X=\Spf A^+$ is affine and its generic fiber $X=\X_\eta$ admits an \'etale map to a disc $\bf{D}^d_K$. \smallskip
    
    Lemma~\ref{lemma:monodromy-pure-after-extension} and Lemma~\ref{lemma:compute-nearby-cycles}(\ref{lemma:compute-nearby-cycles-2}) (and its evident extension to the case of $\Q_\ell$-coefficients) imply that it suffices to show that $\rm{R}\Psi_{\X_{\O_L}} \Q_\ell$ is monodromy pure of weight zero after some finite extension $K\subset L$. Now the Reduced Fiber Theorem (see \cite[Theorem 3.4.2]{L-Jac}) ensures that there is a finite extension $K\subset L$ (with a finite extension $k\subset l$ of residue fields) such that $B\coloneqq (A^+[\frac{1}{p}]\wdh{\otimes}_K L)^\circ$ has a reduced special fiber and the map $A^+\wdh{\otimes}_{\O_K}\O_L \to B$ is finite. We denote the \'etale topos of $\Spec L$ by $\eta'$ and of $\Spec l$ by $s'$. Then the $2$-commutative diagram  
    \[
    \begin{tikzcd}[column sep = huge, row sep = 3em]
        & (\Spf B)_{s'}\times_{s'} \eta' \arrow{d}{f_{s'}\times_{s'}\eta'} \\
    X_{L, \et} \arrow{ur}{\Psi_{\Spf B}} \arrow{r}{\Psi_{\X_{\O_L}}}& \X_{s'}\times_{s'} \eta'. 
    \end{tikzcd}
    \]
    implies that $\rm{R}\Psi_{\X_{\O_L}} \Q_\ell \simeq (f_{s'}\times_{s'}\eta')_* \rm{R}\Psi_{\Spf B} \Q_\ell$, so Lemma~\ref{lemma:finite-morphisms-monodromy-pure} ensures that it suffices to prove the claim for $\X=\Spf B$ and $K=L$. Therefore, we may and do assume that $X=\Spa(A, A^+)$ is an affinoid with an \'etale map to a disc and $\X=\Spf A^\circ$ with reduced special fiber. \smallskip
    
    We use Lemma~\ref{lemma:monodromy-pure-after-extension} and Lemma~\ref{lemma:compute-nearby-cycles}(\ref{lemma:compute-nearby-cycles-2}) (and its evident extension to the case of $\Q_\ell$-coefficients) again to say that it suffices to show that 
    \[
    \rm{R}\Psi_{\X_{\O_{K_\infty}}} \Q_\ell \in D^b_c(\X_s\times_s \eta_\infty; \Q_\ell)
    \]
    is monodromy-pure of weight $0$. Now Lemma~\ref{lemma:stable-circle} guarantees that $\X_{\O_{K_\infty}}\simeq \Spf (A\wdh{\otimes}_K K_\infty)^\circ$, so Theorem~\ref{thm:first-weight-monodromy} implies that $\rm{R}\Psi_{\X_{\O_{K_\infty}}} \Q_\ell$ is monodromy-pure of weight $0$ finishing the proof. 
\end{proof}

\begin{cor}\label{cor:cofinal-family} Let $K$ be a $p$-adic local field, and $X$ a smooth rigid-analytic $K$-variety. Then $X$ admits a cofinal family of admissible formal models $\{\X_i\}_{i\in I}$ such that $\rm{R}\Psi_{\X_i} \Q_\ell$ is monodromy-pure of weight $0$ for each $i\in I$.
\end{cor}
\begin{proof}
    It follows directly from Theorem~\ref{thm:local-weight-monodromy} and \cite[Proposition 3.7]{BLR3}.
\end{proof}

\section{Conjectures and questions}

In this section, we mention some conjectures and questions about $\ell$-adic cohomology groups of $p$-adic rigid-analytic varieties. 

\begin{conjecture1}\label{conj:Weights}(Weights) Let $K$ be a $p$-adic local field, $X$ a quasi-compact and quasi-separated rigid-analytic $K$-variety, and $\ell\neq p$ a prime number. Then 
\begin{enumerate}
    \item For any $g\in G_\eta$ projecting to the geometric Frobenius in $G_s$ and any integer $i\geq 0$, the eigenvalues of $g$ acting on $\rm{IH}^i_c(X_{\wdh{\ov{\eta}}}; \Q_\ell)$ are $q$-Weil numbers of weights $\leq 2(i+d)$;
    \item For any $g\in G_\eta$ projecting to the geometric Frobenius in $G_s$ and any integer $i\geq 0$, the eigenvalues of $g$ acting on $\rm{IH}^i(X_{\wdh{\ov{\eta}}}; \Q_\ell)$ are $q$-Weil numbers of weights $\geq 0$.
\end{enumerate}
In particular, if $X$ is smooth and proper, the eigenvalues of any geometric Frobenius lift on $\rm{H}^i(X_{\ov{\wdh{\eta}}}, \Q_\ell)$ are $\geq 0$ and $\leq 2i$. 
\end{conjecture1}

\begin{rmk1} If $X$ is smooth, proper, and {\it algebraic}. Then one can show that the eigenvalues of any geometric Frobenius lift on $\rm{H}^i(X_{\ov{\wdh{\eta}}}, \Q_\ell)$ are $\geq 0$ and $\leq 2i$. Indeed, one can first reduce to the strictly semi-stable case by de using de Jong's alterations. Then the result follows from \cite[Lemma 3.7(i)]{Saito}.
\end{rmk1}

\begin{conjecture1}\label{conj:exponent}(Exponent of Unipotency) Let $K$ be a $p$-adic discretely valued field, $X$ a quasi-compact and quasi-separated rigid-analytic $K$-variety, $\ell\neq p$ a prime, and $\Lambda\in \{\Z/\ell^n\Z, \Z_\ell, \Q_\ell\}$. Then there is an open subgroup $I_1\subset I$ (independent of $\ell$ and $\Lambda$) such that, for all $g\in I_1$ and $i\in \bf{N}$,  $(g-1)^{i+1}=0$ on $\rm{H}^i_{(c)}(X_{\wdh{\ov{\eta}}}, \Lambda)$, $\rm{IH}^{i-d}_{(c)}(X_{\wdh{\ov{\eta}}}, \Lambda)$.
\end{conjecture1}

\begin{rmk1} In the algebraic case, Conjecture~\ref{conj:Weights} is known for $\rm{H}^\bullet$ and $\rm{H}^\bullet_c$ due to Gabber and Illusie (see \cite[Theorem 2.3]{Illusie-Grothendieck}). It is also know for $\rm{IH}^\bullet(X_{\ov{\eta}}, \Q_\ell)$ and $\rm{IH}^{\bullet}_c(X_{\ov{\eta}}, \Q_\ell)$ in the algebraic situation by reducing to the smooth case via the Decomposition theorem (see \cite[Remark 2.5]{Illusie-Grothendieck}\footnote{Note that \cite{Illusie-Grothendieck} uses a different normalization for the intersection cohomology. So the shift by $d$ in Conjecture~\ref{conj:exponent} does not appear in \cite{Illusie-Grothendieck}.}). The $\Z/\ell^n\Z$ and $\Z_\ell$ versions for the intersection cohomology seem to be unknown even in the algebraic case.
\end{rmk1}

Now we discuss a possible approach to reducing the Weight-Monodromy Conjecture from Theorem~\ref{thm:local-weight-monodromy}. The natural question to ask is how the notion of monodromy pure complexes interacts with $6$-functors. It is tempting to ask whether $\rm{R}(f\times_s \eta)_*$ preserves monodromy-pure complexes of weight $w$ for a proper morphism $f\colon X \to Y$ of $k$-schemes. However, this cannot be true in this generality as the following example shows:

\begin{example1} Let $X$ be a Hopf surface over $K$, $\X$ an admissible formal $\O_K$-model of $X$ as in Theorem~\ref{thm:local-weight-monodromy}, and $f_s\colon \X_s\to \Spec k$ the structure morphism. If $\rm{R}(f_{s}\times_s \eta)_*$ preserves monodromy-pure complexes of weight $0$, then $\rm{H}^i(X_{\wdh{\ov{\eta}}}, \Q_\ell)$ satisfies the weight-monodromy conjecture (see Conjecture~\ref{conj:weight-monodromy}). However, this is already false for $\rm{H}^1(X_{\wdh{\ov{\eta}}}, \Q_\ell)$. 
\end{example1}

A special feature of Hopf surfaces is that they never admit an admissible formal model with {\it projective} special fiber (see \cite[Theorem 1.2 and Example 5.2]{Hansen-Li}). Therefore, it still makes sense to ask if $\rm{R}(f\times_s \eta)_*$ preserves monodromy-pure complexes of weight $w$ for {\it projective} $f$.

\begin{question1}\label{question:super-Weil} Let $K$ be a $p$-adic local field, $\ell\neq p$ a prime number, $f\colon X \to Y$ a projective morphism of finite type $k$-schemes, and $\F\in D^b_c(X\times_s \eta; \Q_\ell)$ monodromy pure of weight $w$. Is $\rm{R}(f\times_s \eta)_*\F\in D^b_c(Y\times_s \eta; \Q_\ell)$ monodromy-pure of weight $w$?
\end{question1}

\begin{rmk1}\label{rmk:question-weight-monodromy} A positive answer to Question~\ref{question:super-Weil} would imply that Theorem~\ref{thm:local-weight-monodromy} holds for {\it every} admissible formal $\O_K$-model $\X$ of a smooth qcqs rigid-analytic $K$-variety $X$. More importantly, it would imply that the Weight-Monodromy Conjecture holds for any smooth, proper rigid-analytic varieties with a projective reduction.
\end{rmk1}

\begin{conjecture1} Let $K$ be a $p$-adic local field, $X$ a smooth qcqs rigid-analytic $K$-variety, and $\ell\neq p$ a prime number. Suppose that $X$ admits an admissible formal $\O_K$-model $\X$ with a projective special fiber $\X_s$. Then the eigenvalues of any geometric Frobenius lift on $\rm{gr}^j_{\rm{M}} \rm{H}^i(X_{\wdh{\ov{\eta}}}, \Q_\ell)$ are $q$-Weil numbers of weight $i+j$ for every integers $i, j$.
\end{conjecture1}

\newpage

\appendix

\section*{Appendix}

\section{Deligne's category and nearby cycles}\label{appendix:deligne}

Let $K$ be a non-archimedean field with a ring of integers $\O_K$ and residue field $k=k(s)$. Let $X$ be a finite type $k$-scheme. The main goal of this Appendix is to recall the construction and basic properties of the category of ``sheaves on $X_{\ov{s}}$ with a continuous action of $G_K$''. The results of this Appendix are well-known to experts, but are not always easy to extract from the literature. However, we do not usually give full proofs in this section, and only give references to other papers. For the most part, we follow \cite{SGA7_2}, \cite{deGabber}, and \cite{Lu-Zheng}. \smallskip

For the rest of this section, we fix a non-archimedean field $K$ with ring of integers $\O_K$ and residue field $k=k(s)$. In what follows, we denote by $G_s$ the absolute Galois group of $k$ and by $G_\eta$ the absolute Galois group of $K$.  \smallskip

We denote by $s$ (resp. $\eta$) the classifying topos of the pro-finite group $G_s$ (resp. $G_\eta$), or equivalently the \'etale topos of $\Spec k$ (resp. $\Spec K$ or $\Spa(K, \O_K)$); it consists of discrete sets with equipped with a continuous action of $G_s$ (resp. $G_\eta$). The natural morphism $r\colon G_\eta \to G_s$ induces a canonical morphism of topoi $r\colon \eta \to s$. For each $g\in G_\eta$, we often denote its image $r(g)\in G_s$ simply by $\ov{g}$.\smallskip

For a finite type $k$-scheme, we will freely abuse the notation and denote by $X_{\ov{s}}$ both $X_{\ov{k}}$ and $X_{k^{\rm{sep}}}$. It should not cause any confusion because the associated \'etale topoi are canonically equivalent. 

\subsection{Definition of Deligne's topos}

The main goal of this section is to formalize the notion of a sheaf on $X_{\ov{s}}$ with a ``continuous'' action of $G_\eta$. More precisely, let $X$ be a qcqs $k$-scheme; by functoriality $X_{\ov{s}}$ admits the natural {\it right} action of $G_s$, and so the natural action of $G_\eta$ through the quotient $r\colon G_\eta \to G_s$. In particular, for each $g\in G_\eta$, there is an automorphism \[
\ov{g}\colon X_{\ov{s}} \to X_{\ov{s}}.
\]
This induces the morphism of \'etale topoi $\ov{g}\colon X_{\ov{s},\et}\to X_{\ov{s}, \et}$, and so the pullback functors 
\[
\ov{g}^*\colon \rm{Shv}\left(X_{\ov{s}, \et}\right)\to \rm{Shv}\left(X_{\ov{s}, \et}\right).
\]
This data defines a {\it right} action of $G_\eta$ on $X_{\ov{s}, \et}$, and so pullbacks define a left action of $G_\eta$. In particular, these pullbacks come with the identifications $\ov{g}^*\circ\ov{h}^*\simeq (\ov{gh})^*$.

\begin{defn} An {\it action of $G_\eta$} on an \'etale sheaf $\F$ on $X_{\ov{s}}$ is family of isomorphisms 
\[
\rho_g\colon \ov{g}^*\F \to \F \text{ } (g\in G_\eta)
\]
such that $\rho_{e}=\rm{Id}$ and the diagram
\[
\begin{tikzcd}
\ov{g}^*\left(\ov{h}^*\F\right) \arrow{d}{\rm{iso}}\arrow{r}{\ov{g}^*\left(\rho_h\right)} & \ov{g}^*\F\arrow{d}{\rho_g} \\
\left(\ov{gh}\right)^*(\F) \arrow{r}{\rho_{gh}}& \F
\end{tikzcd}
\]
commutes for any $g,h\in G_\eta$. 

We denote by $\rm{S}_{G_\eta}(X_{\ov{s}})$ the {\it category of $G_\eta$-sheaves on $X_{\ov{s}}$}. Concretely, the objects of this category are pairs $(\F, \rho)$ of an \'etale $X_{\ov{s}}$-sheaf $\F$ equipped with an action $\rho$ of $G_\eta$, and morphisms $(\F, \rho) \to (\G, \rho')$ are morphisms between $\F \to \G$ that intertwine the $G_\eta$-actions.  
\end{defn}

Now we wish to define continuous $G_\eta$-actions in this context. Let $k\subset k'$ be a finite extension, we denote by $G_{\eta, k'}$ to be the pre-image $r^{-1}(G_{k'}) \subset G_{\eta}$. Now let $(\F, \rho)$ be a $G_\eta$-sheaf on $X_{\ov{s}}$. Then we note, for a finite Galois extension $k\subset k'$ and an \'etale morphism $U' \to X_{k'}$, the action $\rho$ defines an (honest) action of the group $G_{\eta, k'}$ on $\F(U'\times_{X_{k'}} X_{\ov{s}})$. 

\begin{defn} A $G_\eta$-action on an \'etale sheaf $\F$ on $X_{\ov{s}}$ is {\it continuous} if, for every finite Galois extension $k\subset k'$ and a qcqs \'etale morphism $U' \to X_{k'}$, the associated action of $G_{\eta, k'}$ on $\F(U'\times_{X_{k'}} X_{\ov{s}})$ is continuous with respect to the discrete topology on $\F(U'\times_{X_{k'}} X_{\ov{s}})$. 

We denote by $\rm{T}_{G_\eta}(X_{\ov{s}})$ the full subcategory of $\rm{S}_{G_\eta}(X_{\ov{s}})$ that consists of $G_\eta$-sheaves with a continuous action.
\end{defn}

This definition is rather concrete. However, it is also helpful to consider another (more abstract) equivalent definition. For this, we recall that the $2$-category of topoi $\cal{T}$ admits all $2$-fiber products by \cite[Proposition (3.4)]{Giraud} (also, see \cite[Exp. XI, Th\'eorem\`e 3.2]{deGabber} for an explicit site-theoretic construction). The construction of this $2$-fiber product is not obvious, and in particular we warn the reader that this fiber product {\it does not} commute with the forgetful functor $\cal{T}\to \cal{C}at_2$ from the $2$-category of topoi to the $2$-category of categories. \smallskip

We now apply this construction in our case of interest. Namely, let $X$ be a qcqs $k$-scheme. The structure morphism $X \to \Spec k$ defines a morphism of \'etale topoi $X_\et \to s$, while the continuous morphism $r\colon G_\eta \to G_s$ defines a morphism of classifying topoi $\eta\to s$. 

\begin{defn}\label{defn:deligne-topos} {\it Deligne's Topos} $X\times_s \eta$ is the $2$-fiber product $X_\et\times_s \eta$. 
\end{defn}

Now we choose a point\footnote{A point of a topos $T$ is morphism of topoi $\rm{pt}\to T$.} $p_{\eta}\colon \ov{\eta} \to \eta$ (that is unique up to a (non-unique) isomorphism by \cite[\href{https://stacks.math.columbia.edu/tag/04HU}{Tag 04HU}]{stacks-project}), and a point $p_s\colon \ov{s}\to s$, and an equivalence $\varphi\colon \ov{\eta}\simeq \ov{s}$ such that the diagram
\begin{equation}\label{eqn:choices}
\begin{tikzcd}
\ov{\eta}\arrow{d}{\varphi} \arrow{r}{p_\eta}& \eta \arrow{d}{r}\\
\ov{s} \arrow{r}{p_s} & s
\end{tikzcd}
\end{equation}
commutes\footnote{Geometrically, this choice corresponds to a choice of an algebraic closure $\ov{K}$ of $K$ together with an identification of the residue field of $\ov{K}$ with an algebraic closure of $k$.}. In what follows, for a topos $T$, we denote by $\rm{Points}(T)=\rm{Map}_{\cal{T}}(\rm{pt}, T)$ the category of points of $T$. We can now formulate the main properties of $X\times_s \eta$:

\begin{lemma}\label{lemma:properties-Deligne} Let $X$ be a qcqs $k$-scheme. Then 
\begin{enumerate}
    \item\label{lemma:properties-Deligne-1} there is an equivalence $(X\times_s \eta)\times_\eta \ov{\eta}\simeq X_{\ov{s}, \et}$;
    \item\label{lemma:properties-Deligne-2} there is an equivalence $X\times_s \eta\simeq \rm{T}_{G_\eta}(X_{\ov{s}})$ such that the under the natural projection morphism (that comes from $(\ref{lemma:properties-Deligne-1})$) $\pi_{X}\colon X_{\ov{s},\et} \to X\times_s \eta$ the pullback functor $\pi^*_{X}$ is identified with the forgetful functor $\rm{T}_{G_\eta}(X_{\ov{s}}) \to X_{\ov{s}, \et}$;
    \item\label{lemma:properties-Deligne-3} $\pi^*_X$ induces an essentially surjective functor $\rm{Points}(X\times_s \eta) \to \rm{Points}(X_{\ov{s}})$. In particular, for every ring $\Lambda$, the natural morphism $\cal{D}(X\times_s \eta; \Lambda) \to \cal{D}(X_{\ov{s}}; \Lambda)$ is conservative. 
\end{enumerate}
\end{lemma}
\begin{proof}
    $(\ref{lemma:properties-Deligne-1})$ We note that $(X\times_s \eta)\times_\eta \ov{\eta} \simeq X_{\et}\times_s \ov{\eta}$. Using the diagram~(\ref{eqn:choices}), we conclude that it suffices to show that $X_\et\times_s \ov{s} \simeq X_{\ov{s}, \et}$. By the universal property of $2$-fiber products, there is a natural morphism 
    \[
    X_{\ov{s}, \et} \to X_\et\times_s \ov{s}
    \]
    that we need to show to be an equivalence.\smallskip
    
    For a finite Galois extension $k\subset k'\subset \ov{k}$, we denote by $s'$ the \'etale topos of $\Spec k'$ and by $X_{s'}$ the fiber product (of schemes) $X\times_k \Spec k'$. Then \cite[Lemma 8.3]{Morin} ensures that 
    \[
    \ov{s}\simeq \lim_{k\subset k'}\ov{s'} 
    \]
    and
    \[
    X_{\ov{s}, \et} \simeq \lim_{k\subset k'} X_{s', \et},
    \]
    where the (cofiltered) limit is taken in the $2$-category of topoi, and is taken over all finite Galois extensions of $k$ inside $\ov{k}$. Since cofiltered $2$-limits commute with $2$-fiber products, it suffices to show that 
    \[
    X_{s', \et}\to X_{\et} \times_s s'
    \]
    is an equivalence of topoi. This follows from \cite[Exp. IV, Proposition 5.11]{SGA4_2}.\smallskip
    
    $(\ref{lemma:properties-Deligne-2})$ This is \cite[Exp. XII, Construction 1.2.4]{SGA7_2} (and the discussion after this construction). Unfortunately, the discussion in SGA is pretty terse, so we also refer to \cite[Theorem 3.1]{Weil-site} for a proof of a similar result that can be adapted to this situation. \smallskip
    
    $(\ref{lemma:properties-Deligne-3})$ By the universal property of $2$-fiber products, we see that $\rm{Points}(X\times_s \eta)\simeq \rm{Points}(X)\times_{\rm{Points(s)}}\rm{Points}(\eta)$ where the fiber product is understood to be the $2$-fiber product in the $2$-category of categories. Then it suffices to show that $\rm{Points}(\eta) \to \rm{Points}(s)$ is essentially surjective. This follows from the fact that both categories contain only $1$ isomorphism class of objects (see \cite[\href{https://stacks.math.columbia.edu/tag/04HU}{Tag 04HU}]{stacks-project}). 
\end{proof}

Recall that every element $g\in G_\eta$ induces a morphism of \'etale topoi $\ov{g}\colon X_{\ov{s}, \et} \to X_{\ov{s}, \et}$ induced by the (right) action of $G_\eta$ on $X_{\ov{s}}$ (through the quotient $G_\eta \to G_s$).

\begin{construction}\label{construction:topos-theoretic} Lemma~\ref{lemma:properties-Deligne}~(\ref{lemma:properties-Deligne-1}, \ref{lemma:properties-Deligne-2}) implies that there is a natural morphism of topoi $\pi_X \colon X_{\ov{s}, \et}\to X\times_s \eta$ such that, for each $g\in G_\eta$, there is an isomorphism $\psi_g\colon \pi_X \simeq \pi_{X} \circ \ov{g}$ such that $\psi_{e}=\rm{Id}$ and the diagram
\[
\begin{tikzcd}
\pi_X\arrow{d}{\psi_{gh}} \arrow{r}{\psi_{g}} &  \pi_X\circ \ov{g}\arrow{d}{\psi_{h}\circ \ov{g}}\\
\pi_X \circ \ov{gh} \arrow{r}{\rm{iso}} &\pi_X \circ \ov{h}\circ \ov{g}
\end{tikzcd}
\]
commutes for every $g, h\in G_\eta$.
\end{construction}

We note that $\pi_X$ is natural in $X$, in the sense that for any morphism $X\to Y$ of qcqs $k$-schemes, the diagram
\[
\begin{tikzcd}
X_{\ov{s},\et}\arrow{d}{\pi_X} \arrow{r} &  Y_{\ov{s},\et} \arrow{d}{\pi_Y}\\
X \times_s \eta \ar{r} & Y \times_s \eta
\end{tikzcd}
\]
commutes up to canonical 2-isomorphism. 

\begin{construction}\label{construction:action} Now we note that, for every ring $\Lambda$, there is a strictly unitary functor (in the sense of \cite[\href{https://kerodon.net/tag/008K}{Tag 008K}]{kerodon} and \cite[\href{https://kerodon.net/tag/008R}{Tag 008R}]{kerodon}) 
\[
D(-;\Lambda)^*\colon \cal{T}^{\rm{op}} \to \cal{C}at
\]
from the $2$-category of topoi to the $2$-category of categories that sends a topos $T$ to the derived category $D(T; \Lambda)$ and a morphism of topoi $f\colon T\to T'$ to the pullback functor $f^*$ . In particular, for a prime number $\ell$, an integer $n\geq 1$, a qcqs $k$-scheme $X$, and a sheaf $\F\in D(X\times_s\eta; \Z/\ell^n\Z)$, we can pass to pullbacks in Construction~\ref{construction:topos-theoretic} to get a family of isomorphisms $\rho_g\colon \ov{g}^*\pi_X^*\F \to \pi_X^*\F$ such that $\rho_e=\rm{Id}$ and the diagram
\[
\begin{tikzcd}
\ov{g}^*\ov{h}^*\pi^*_X\F \arrow{d}{\rm{iso}} \arrow{r}{\ov{g}^*(\rho_h)}& \ov{g}^*\pi^*_X\F \arrow{d}{\rho_g}\\
(\ov{gh}^*) \pi^*_X \F \arrow{r}{\rho_{gh}}& \pi^*_X\F
\end{tikzcd}
\]
commutes for every $g, h\in G_\eta$. By restricting to the inertia subgroup $I\subset G_\eta$, we get a homomorphism
\[
\rho\colon I \to \rm{Aut}_{\Z/\ell^n\Z}(\pi^*_X\F)
\]
for any $\F\in D(X\times_s \eta; \Z/\ell^n\Z)$. 
\end{construction}

\begin{construction}\label{construction:sigma} Suppose $\sigma\colon G_s \to G_\eta$ is a continuous section of the projection morphism $r\colon G_\eta \to G_s$, so $\sigma$ defines a morphism of topoi $\sigma \colon s\to \eta$. The universal property of $2$-fiber products imply that this defines an essentially unique morphism of topoi 
\[
\sigma_X\colon X_{\et} \to X\times_s \eta. 
\]
In particular, for each prime number $\ell$ and an integer $n\geq 1$, we have the well-defined pullback functor
\[
\sigma_X^*\colon \cal{D}(X\times_s \eta; \Z/\ell^n\Z) \to \cal{D}(X; \Z/\ell^n\Z).
\]
\end{construction}

\begin{construction}\label{construction:p} For any $X$ be a qcqs $k$-scheme, Deligne's topos $X\times_s \eta$ comes with the natural canonical projection $p_X\colon X\times_s \eta \to X$. In particular, for any prime number $\ell$ and an integer $n$, there is a canonical pullback functor
\[
p^*_X\colon \cal{D}(X; \Z/\ell^n\Z)\to \cal{D}(X\times_s \eta; \Z/\ell^n\Z).
\]
\end{construction}

\begin{lemma}\label{lemma:trivial-action-descend} Let $X$ be a qcqs $k$-scheme, $\ell$ a prime number, and $n$ an positive integer. Suppose that $\F\in \rm{Shv}(X\times_s \eta; \Z/\ell^n\Z)$ such that the inertia action $I$ on $\pi^*_X\F$ is trivial. Then the natural morphism
\[
\F \to p_X^*p_{X, *}\F
\]
is an isomorphism. Furthermore, if $\F$ is in addition locally constant with finite rank free stalks, then $p_{X, *}\F$ is also locally constant with finite rank free stalks.
\end{lemma}
\begin{proof}
    We note that \cite[Proposition 9.2.1]{Lei-Fu} identifies $\rm{Shv}(X; \Z/\ell^n\Z)$ with sheaves of $\Z/\ell^n\Z$-modules on $X_{\ov{s}}$ with a continuous $G_s$-action. Then $p_{X, *}$ corresponds to the functor of $I$-invariants, and $p_X^*$ to the functor that sends a sheaf on $X_{\ov{s}}$ with a continuous $G_s$ to the same sheaf with a continuous action of $G_\eta$ through the quotient $G_\eta\to G_s$. Under these identifications, it becomes clear that the natural morphism
    \[
    \F \to p_{X}^*p_{X, *}\F
    \]
    is an isomorphism if $I$ acts trivially on $\pi^*_X\F$. \smallskip
    
    Now we assume that $\F$ is locally constant with finite free stalks, and consider the natural projection morphism $c_X\colon X_{\ov{s}, \et} \to X_{\et}$. Using that $c_X=\pi_X\circ p_X$ and $\F\simeq p_X^*p_{X, *}\F$, we conclude that
    \[
    c^*_X p_{X, *}\F \simeq \pi^*_X p_X^* p_{X, *} \F\simeq \pi^*_X \F
    \]
    is locally constant with finite free stalks. Thus the same holds for $p_{X, *}\F$.
\end{proof}

Now we wish to discuss the various functors on the Deligne's topos $X\times_s \eta$ for a qcqs $k$-scheme $X$. 

\begin{defn} For a finite type $k$-scheme $X$, an object $\F\in \cal{D}(X_s\times_s \eta, \Z/\ell^n\Z)$ is {\it constructible of finite tor dimension} if $\pi_X^* \F\in \cal{D}^b_{ctf}(X_{\ov{s}}; \Z/\ell^n\Z)$. We denote this category by $\cal{D}^b_{ctf}(X_s\times_s\eta; \Z/\ell^n\Z)$.
\end{defn}

\begin{lemma}\label{lemma:pushforward-pullback} Let $f\colon X\to Y$ be a morphism of qcqs $k$-schemes, $\ell$ a prime number, and $n\geq 1$ a positive integer. Suppose that $\rm{R}f_{\ov{s}, *}\colon \cal{D}(X_{\ov{s}};\Z/\ell^n\Z) \to \cal{D}(X_{\ov{s}};\Z/\ell^n\Z)$ is of finite cohomological dimension\footnote{This condition is automatic if $X$ and $Y$ are finite type over $k$ by \cite[Corollary 7.5.6]{Lei-Fu}.}. Then 
\begin{enumerate}
    \item\label{lemma:pushforward-pullback-1} the diagram 
    \[
\begin{tikzcd}
    \cal{D}(X\times_s \eta; \Z/\ell^n\Z) \arrow{d}{\rm{R}(f\times_s\eta)_*} \arrow{r}{\pi^*_{X}} & \cal{D}(X_{\ov{s}}; \Z/\ell^n\Z) \arrow{d}{\rm{R}f_{\ov{s}, *}} \\
    \cal{D}(Y\times_s \eta; \Z/\ell^n\Z) \arrow{r}{\pi^*_{Y}} & \cal{D}(Y_{\ov{s}}; \Z/\ell^n\Z),
\end{tikzcd}
\] commutes (up to a canonical isomorphism);
\item\label{lemma:pushforward-pullback-2} the diagram
\[
\begin{tikzcd}
    \cal{D}(X\times_s \eta; \Z/\ell^n\Z) \arrow{d}{\rm{R}(f\times_s \eta)_*} \arrow{r}{{\sigma}_X^*} & \cal{D}(X; \Z/\ell^n\Z) \arrow{d}{\rm{R}f_*} \\
    \cal{D}(Y\times_s \eta; \Z/\ell^n\Z) \arrow{r}{\sigma_Y^*} & \cal{D}(Y; \Z/\ell^n\Z). 
\end{tikzcd}
\]
commutes (up to a canonical isomorphism) for every continuous section $\sigma\colon G_s \to G_\eta$;
\item\label{lemma:pushforward-pullback-new} the diagram
\[
\begin{tikzcd}
    \cal{D}(X; \Z/\ell^n\Z) \arrow{d}{\rm{R}f_*}  \arrow{r}{p_X^*} & \cal{D}(X\times_s \eta; \Z/\ell^n\Z) \arrow{d}{\rm{R}(f\times_s \eta)_*}\\
    \cal{D}(Y; \Z/\ell^n\Z) \arrow{r}{p_Y^*} & \cal{D}(Y\times_s \eta; \Z/\ell^n\Z). 
\end{tikzcd}
\]
commutes (up to a canonical isomorphism);
\item\label{lemma:pushforward-pullback-3} The natural morphism
\[
c_{\F,n, m}\colon \rm{R}(f\times_s \eta)_*(\F) \otimes^L_{\Z/\ell^n\Z} \Z/\ell^m\Z \to \rm{R}(f\times_s \eta)_*(\F \otimes^L_{\Z/\ell^n\Z} \Z/\ell^m\Z) 
\]
is an isomorphism for any $\F\in \cal{D}(X\times_s\eta; \Z/\ell^n\Z)$ and $n\geq m$;

\item\label{lemma:pushforward-pullback-4} If $f$ is a morphism of finite type $k$-schemes and $\ell$ is invertible in $k$, $\rm{R}(f\times_s \eta)_*$ carries $\cal{D}^b_{ctf}(X\times_s\eta; \Z/\ell^n\Z)$ to $\cal{D}^b_{ctf}(Y\times_s\eta; \Z/\ell^n\Z)$.
\end{enumerate}
\end{lemma}
\begin{proof}
$(\ref{lemma:pushforward-pullback-1})$, $(\ref{lemma:pushforward-pullback-2})$, and $(\ref{lemma:pushforward-pullback-new})$ can be seen explicitly using the explicit site-theoretic construction of $2$-fiber products from \cite[Exp. XI, 3.1]{deGabber}. Alternatively, they follow directly from \cite[Lemma 1.3, Proposition 1.17, and Remark 1.18]{Lu-Zheng}.\smallskip

$(\ref{lemma:pushforward-pullback-3})$ follows directly from \cite[Corollary 1.20]{Lu-Zheng}.\smallskip

$(\ref{lemma:pushforward-pullback-4})$ We note that $(\ref{lemma:pushforward-pullback-1})$ implies that it suffices to show analogous claim for $\rm{R}f_{\ov{s}, *}$ which is standard (see \cite[Theorem 9.5.2]{Lei-Fu}).
\end{proof}

In what follows, we will need to be able to compute the Hom spaces in Deligne's topos. We now discuss some general results in this direction.

\begin{lemma}\label{lemma:homs-after-base-change} Let $X$ be a finite type $k$-scheme, let $\ell$ be a prime number, let $n$ be a positive integer, let $\F\in \cal{D}^b_{ctf}(X\times_s\eta; \Z/\ell^n\Z)$, and let $\G\in \cal{D}^+(X\times_s\eta; \Z/\ell^n\Z)$. Then the natural morphisms
\begin{equation}\label{eqn:pi}
\pi^*_X\rm{R}\cal{H}om_{\Z/\ell^n\Z}(\F, \G) \to \rm{R}\cal{H}om_{\Z/\ell^n\Z}(\pi_X^*\F, \pi_X^*\G),
\end{equation}
\begin{equation}\label{eqn:sigma}
\sigma^*_X\rm{R}\cal{H}om_{\Z/\ell^n\Z}(\F, \G) \to \rm{R}\cal{H}om_{\Z/\ell^n\Z}(\sigma_X^*\F, \sigma_X^*\G)
\end{equation}
are isomorphisms for any continuous section $\sigma\colon G_s\to G_\eta$ of the projection morphism $r\colon G_\eta\to G_s$. Similarly, if $\F\in \F\in \cal{D}^b_{ctf}(X; \Z/\ell^n\Z)$  and $\G \in \cal{D}^+(X; \Z/\ell^n\Z)$, then the natural morphism
\begin{equation}\label{eqn:p}
p^*_X\rm{R}\cal{H}om_{\Z/\ell^n\Z}(\F, \G) \to \rm{R}\cal{H}om_{\Z/\ell^n\Z}(p_X^*\F, p_X^*\G)
\end{equation}
is an isomorphism.
\end{lemma}
\begin{proof}
     We first apply \cite[Lemma 1.29]{Lu-Zheng} to the weakly \'etale morphism $f\colon \Spec \ov{K}\to \Spec K$ to conclude that (\ref{eqn:pi}) is an isomorphism if $\F$ has constructible cohomology sheaves in the sense of \cite{Lu-Zheng} (see the discussion after \cite[Corollary 1.26]{Lu-Zheng} for a precise definition). By \cite[Lemma 1.28]{Lu-Zheng} (and noting that $G_\eta \to G_s$ is already surjective), it suffices to show that a sheaf $\F \in \rm{Shv}(X\times_s\eta; \Z/\ell^n\Z)$ is noetherian if $\pi_X^*\F$ is constructible. By Lemma~\ref{lemma:properties-Deligne}~(\ref{lemma:properties-Deligne-3}), it suffices to show that $\pi_X^*\F$ is noetherian. This, in turn, follows from \cite[\href{https://stacks.math.columbia.edu/tag/09YV}{Tag 09YV}]{stacks-project}. \smallskip
    
    Before we discuss other isomorphisms, we note that the same proof applied to the morphism $\Spec \ov{k}\to \Spec k$ shows that the natural morphism
    \begin{equation}\label{eqn:c}
    c^*_X\rm{R}\cal{H}om_{\Z/\ell^n\Z}(\F, \G) \to \rm{R}\cal{H}om_{\Z/\ell^n\Z}(c_X^*\F, c_X^*\G)
    \end{equation}
    is an isomorphism, where $c_X\colon X_{\ov{s}, \et}\to X_\et$ is the natural projection, $\F\in \cal{D}^b_{ctf}(X; \Z/\ell^n\Z)$, and $\G \in \cal{D}^+(X; \Z/\ell^n\Z)$. \smallskip
    
    Now we show that the fact that (\ref{eqn:sigma}) is an isomorphism follows formally from the established above facts. To see this, we note that $c_X^*$ is conservative, so it suffices to show that the morphism
    \begin{equation}\label{eqn:sigma-hom}
    \sigma^*_X\rm{R}\cal{H}om_{\Z/\ell^n\Z}(\F, \G) \to \rm{R}\cal{H}om_{\Z/\ell^n\Z}(\sigma_X^*\F, \sigma_X^*\G),
    \end{equation}
    is an isomorphism after applying $c^*_X$. Then the result follows from the fact that (\ref{eqn:pi}) and (\ref{eqn:c}) are isomorphisms. Similarly, one can show that (\ref{eqn:p}) is an isomorphism. 
\end{proof}

\begin{cor}\label{cor:projection-hom} Let $X$ be a finite type $k$-scheme, $\ell$ a prime number invertible in $k$, and $\F, \G\in \cal{D}^b_{ctf}(X\times_s\eta; \Z/\ell^n\Z)$ for some integer $n\geq 1$. Then $\rm{R}\cal{H}om_{\Z/\ell^n\Z}(\F, \G)$ lies in $\cal{D}^b_{ctf}(X\times_s \eta; \Z/\ell^n\Z)$ and the natural morphism
\[
\rm{R}\cal{H}om_{\Z/\ell^n\Z}(\F, \G) \otimes^L_{\Z/\ell^n\Z} \Z/\ell^{n-1}\Z \to \rm{R}\cal{H}om_{\Z/\ell^{n-1}\Z}\left(\F\otimes^L_{\Z/\ell^n\Z} \Z/\ell^{n-1}\Z, \G\otimes^L_{\Z/\ell^n\Z} \Z/\ell^{n-1}\Z\right)
\]
is an isomorphism.
\end{cor}
\begin{proof}
    We firstly show that $\rm{R}\cal{H}om_{\Z/\ell^n\Z}(\F, \G)$ lies in $\cal{D}^b_{ctf}(X\times_s \eta; \Z/\ell^n\Z)$. Lemma~\ref{lemma:homs-after-base-change} ensures that it suffices to show that \[
    \rm{R}\cal{H}om_{\Z/\ell^n\Z}(\pi^*_X\F, \pi^*_X\G) \in \cal{D}^b_{ctf}(X_{\ov{s}}; \Z/\ell^n\Z).
    \]
    This follows from \cite[Theorem 9.5.3(ii)]{Lei-Fu}. \smallskip
    
    Now we show that the natural morphism 
    \[
    \rm{R}\cal{H}om_{\Z/\ell^n\Z}\left(\F, \G\right) \otimes^L_{\Z/\ell^n\Z} \Z/\ell^{n-1}\Z \to \rm{R}\cal{H}om_{\Z/\ell^{n-1}\Z}\left(\F\otimes^L_{\Z/\ell^n\Z} \Z/\ell^{n-1}\Z, \G\otimes^L_{\Z/\ell^n\Z} \Z/\ell^{n-1}\Z\right)
    \]
    is an isomorphism. Lemma~\ref{lemma:properties-Deligne}(\ref{lemma:properties-Deligne-3}) implies that $\pi^*_X$ is conservative, so it suffices to prove the claim after applying $\pi^*_X$. Therefore, it suffices to prove analogous claim for constructible, finite tor dimension complexes on $X_{\ov{s}}$. This is standard (see \cite[Proposition 10.1.17]{Lei-Fu}).
\end{proof}

Now let $X, \F, \G$ be as in Lemma~\ref{lemma:homs-after-base-change} and set $f\colon X \to \Spec k$ to be the structure morphism. We define
\[
\rm{R}\Hom_{/\eta, \Z/\ell^n\Z}(\F, \G) \coloneqq \rm{R}(f\times_s \eta)_*\rm{R}\cal{H}om_{\Z/\ell^n\Z}(\F, \G) \in \cal{D}(\eta; \Z/\ell^n\Z).
\]
By Lemma~\ref{lemma:pushforward-pullback}(\ref{lemma:pushforward-pullback-1}, \ref{lemma:pushforward-pullback-2}) and Lemma~\ref{lemma:homs-after-base-change}, we see that 
\[
\pi^*_s\rm{R}\Hom_{/\eta, \Z/\ell^n\Z}(\F, \G) \simeq \rm{R}\Gamma(X_{\ov{s}}, \rm{R}\cal{H}om_{\Z/\ell^n\Z}(\pi^*_X\F,\pi^*_X\G)) \simeq \rm{RHom}_{\Z/\ell^n\Z}(\pi^*_X\F, \pi^*_X\G).
\]
Informally, $\rm{R}\Hom_{/\eta, \Z/\ell^n\Z}(\F, \G) \in \cal{D}(\eta; \Z/\ell^n\Z)$ is a canonical descent of $\rm{RHom}_{\Z/\ell^n\Z}(\pi^*_X\F, \pi^*_X\G)\in \cal{D}(\ov{s}; \Z/\ell^n\Z)$ to an object of $\cal{D}(\eta;\Z/\ell^n\Z)$.

\begin{lemma}\label{lemma:Hom-formula} Let $X$ be a finite type $k$-scheme, $\ell$ a prime number, and $\F, \G\in \cal{D}^b_{ctf}(X\times_s\eta; \Z/\ell^n\Z)$ for some integer $n\geq 1$. Then there is a canonical isomorphism
\[
\rm{RHom}_{\Z/\ell^n\Z}(\F, \G) \simeq \rm{R}\Gamma_{\rm{cont}}(G_\eta, \rm{R}\Hom_{/\eta, \Z/\ell^n\Z}(\F, \G))
\]
\end{lemma}
\begin{proof}
We have a sequence of isomorphism
\begin{align*}
    \rm{RHom}_{\Z/\ell^n\Z}(\F, \G) & \simeq \rm{R}\Gamma(X\times_s \eta, \rm{R}\cal{H}om_{\Z/\ell^n\Z}(\F, \G)) \\
    & \simeq \rm{R}\Gamma(\eta; \rm{R}(f\times_s \eta)_*\rm{R}\cal{H}om_{\Z/\ell^n\Z}(\F, \G))\\
    & \simeq \rm{R}\Gamma_{\rm{cont}}(G_\eta, \rm{R}\Hom_{/\eta, \Z/\ell^n\Z}(\F, \G)),
\end{align*}
where the last isomorphism uses an identification of $\cal{D}(\eta;\Z/\ell^n\Z)$ with the category of discrete $\Z/\ell^n\Z[G_\eta]$-modules. 
\end{proof}

\begin{cor}\label{cor:hom-formula} Let $X$ be a finite type $k$-scheme, $\ell$ a prime number, and $\F, \G \in \cal{D}^b_{ctf}(X\times_s\eta; \Z/\ell^n\Z)$ for some integer $n\geq 1$. Suppose that $\rm{RHom}_{\Z/\ell^n\Z}(\pi^*_X\F, \pi^*_X\G)\in \cal{D}^{\geq 0}(\Z/\ell^n\Z)$. Then 
\[
\rm{Hom}_{\Z/\ell^n\Z}(\F, \G)=\rm{Hom}_{\Z/\ell^n\Z}(\pi^*_X\F, \pi^*_X\G)^{G_\eta}.
\]
\end{cor}

Now we discuss the finiteness assumptions for the Hom groups in Deligne's topos. 

\begin{lemma}\label{lemma:finitenss-group-cohomology} Let $K$ be a non-archimedean arithmetic field (see Definition~\ref{defn:arithmetic}), $\ell$ a prime number invertible in $\O_K$, and $M\in D^b_{c}(\eta; \Z/\ell^n\Z)$ for some integer $n\geq 1$. Then $\rm{R}\Gamma(\eta, M)\simeq \rm{R}\Gamma_{\rm{cont}}(G_\eta, M)\in D^b_{coh}(\Z/\ell^n\Z)$.
\end{lemma}
\begin{proof}
    Since $\rm{R}\Gamma_{\rm{cont}}(G_\eta, M)$ depends only on the Galois group of $K$, we can assume that $K$ is a local field. 
    
    First, we use a standard spectral sequence to reduce to the case of a finite discrete $G_\eta$-module $M$. Then the claim follows from \cite[Proposition 5.2/14 and Remark 2) on p.92]{Serre-Galois}.
\end{proof}

\begin{cor}\label{cor:finiteness-Deligne-topos} Let $K$ be a non-archimedean arithmetic field, and $X$ a finite type $k$-scheme, $\ell$ a prime number invertible in $\O_K$, and $\F,\G\in D^b_{ctf}(X\times_s \eta; \Z/\ell^n\Z)$ for some integer $n\geq 1$. Then $\rm{RHom}_{\Z/\ell^n\Z}(\F, \G)\in D^b_{coh}(\eta; \Z/\ell^n\Z)$. In particular, $\rm{Ext}^i_{\Z/\ell^n\Z}(\F, \G)$ are finite groups all integers $i$.
\end{cor}
\begin{proof}
    Lemma~\ref{lemma:Hom-formula} implies that 
    \[
    \rm{RHom}_{\Z/\ell^n\Z}(\F, \G)\simeq \rm{R}\Gamma_{\rm{cont}}\left(G_\eta, \rm{R}\Hom_{/\eta, \Z/\ell^n\Z}\left(\F, \G\right)\right).
    \]
    Lemma~\ref{lemma:pushforward-pullback}(\ref{lemma:pushforward-pullback-4}) and Corollary~\ref{cor:projection-hom} imply that 
    \[
    \rm{R}\Hom_{/\eta, \Z/\ell^n\Z}(\F, \G)\in D^b_{coh}(\eta; \Z/\ell^n\Z).
    \]
    Thus the result follows from Lemma~\ref{lemma:finitenss-group-cohomology}. 
\end{proof}

\subsection{Shriek functors}

Th main goal of this section is to discuss the construction of shriek functors for the Deligne topoi. \smallskip

For the rest of this section, we fix a non-archimedean field $K$ with residue field $k$, a prime number $\ell$ invertible in $k$, and an integer $n\geq 1$. 

\begin{construction} For a separated morphism $f\colon X\to Y$ between finite type $k$-schemes, \cite[Construction 1.8]{Lu-Zheng} defines a {\it lower shriek functor}
\[
\rm{R}(f\times_s \eta)_! \colon D(X\times_s \eta; \Z/\ell^n\Z) \to D(Y\times_s \eta;\Z/\ell^n\Z)
\]
such that $\rm{R}(f\times_s \eta)_!= \rm{R}(f\times_s \eta)_*$ for a proper $f$ and $\rm{R}(f\times_s \eta)_!$ is left adjoint to $(f\times_s \eta)^*$ for an open immersion $f$. 

In \cite[Construction 1.9]{Lu-Zheng}, they also define the {\it upper shriek functor}
\[
(f\times_s \eta)^!\colon D(Y\times_s \eta; \Z/\ell^n\Z) \to D(X\times_s \eta; \Z/\ell^n\Z)
\]
as a right adjoint to $\rm{R}(f\times_s \eta)_!$. 
\end{construction}

\begin{rmk} Using \cite[Appendix A.5]{lucas-thesis}, it is possible to promote $\rm{R}(f\times_s \eta)_!$ and $(f\times_s \eta)^!$ to functors of $\infty$-categories. We do not do this in this paper as we will never need this.
\end{rmk}

\begin{lemma}\label{lemma:shriek-pushforward-pullback} Let $f\colon X\to Y$ be a separated morphism of finite type $k$-schemes, $\ell$ a prime number invertible in $k$, and $n\geq 1$ a positive integer. Then 
\begin{enumerate}
    \item\label{lemma:shriek-pushforward-pullback-1} the diagram 
    \[
\begin{tikzcd}
    D(X\times_s \eta; \Z/\ell^n\Z) \arrow{d}{\rm{R}(f\times_s\eta)_!} \arrow{r}{\pi^*_{X}} & D(X_{\ov{s}}; \Z/\ell^n\Z) \arrow{d}{\rm{R}f_{\ov{s}, !}} \\
    D(Y\times_s \eta; \Z/\ell^n\Z) \arrow{r}{\pi^*_{Y}} & D(Y_{\ov{s}}; \Z/\ell^n\Z),
\end{tikzcd}
\] commutes (up to a canonical isomorphism);
\item\label{lemma:shriek-pushforward-pullback-2} the diagram
\[
\begin{tikzcd}
    D(X\times_s \eta; \Z/\ell^n\Z) \arrow{d}{\rm{R}(f\times_s \eta)_!} \arrow{r}{{\sigma}_X^*} & D(X; \Z/\ell^n\Z) \arrow{d}{\rm{R}f_!} \\
    D(Y\times_s \eta; \Z/\ell^n\Z) \arrow{r}{\sigma_Y^*} & D(Y; \Z/\ell^n\Z). 
\end{tikzcd}
\]
commutes (up to a canonical isomorphism) for every continuous section $\sigma\colon G_s \to G_\eta$;
\item\label{lemma:shriek-pushforward-pullback-new} the diagram
\[
\begin{tikzcd}
   D(X; \Z/\ell^n\Z) \arrow{d}{\rm{R}f_!}  \arrow{r}{p_X^*} & D(X\times_s \eta; \Z/\ell^n\Z) \arrow{d}{\rm{R}(f\times_s \eta)_!}\\
    D(Y; \Z/\ell^n\Z) \arrow{r}{p_Y^*} & D(Y\times_s \eta; \Z/\ell^n\Z). 
\end{tikzcd}
\]
commutes (up to a canonical isomorphism);

\item\label{lemma:shriek-pushforward-pullback-3} The natural morphism
\[
c_{\F,n, m}\colon \rm{R}(f\times_s \eta)_!\,\F \otimes^L_{\Z/\ell^n\Z} \Z/\ell^m\Z \to \rm{R}(f\times_s \eta)_!\left(\F \otimes^L_{\Z/\ell^n\Z} \Z/\ell^m\Z\right)
\]
is an isomorphism for any $\F\in D(X\times_s\eta; \Z/\ell^n\Z)$ and $n\geq m$;

\item\label{lemma:shriek-pushforward-pullback-4} $\rm{R}(f\times_s \eta)_!$ carries $D^b_{ctf}(X\times_s\eta; \Z/\ell^n\Z)$ to $D^b_{ctf}(Y\times_s\eta; \Z/\ell^n\Z)$.
\end{enumerate}
\end{lemma}
\begin{proof}
    (\ref{lemma:shriek-pushforward-pullback-1}), (\ref{lemma:shriek-pushforward-pullback-2}), and (\ref{lemma:shriek-pushforward-pullback-new}) follow from \cite[Construction 1.8, Property (2)]{Lu-Zheng}. (\ref{lemma:shriek-pushforward-pullback-3}) follows from \cite[Construction 1.8, Property (3)]{Lu-Zheng}. And (\ref{lemma:shriek-pushforward-pullback-4}) can be proven similarly to Lemma~\ref{lemma:pushforward-pullback}(\ref{lemma:pushforward-pullback-4}). 
\end{proof}

Now we discuss the basic properties of the upper shriek functor:

\begin{lemma}\label{lemma:upper-shriek-pullback} Let $f\colon X\to Y$ be a separated morphism of finite type $k$-schemes, $\ell$ a prime number invertible in $k$, and $n\geq 1$ a positive integer. Then 
\begin{enumerate}
    \item\label{lemma:upper-shriek-pullback-1} the diagram 
    \[
\begin{tikzcd}
    D(Y\times_s \eta, \Z/\ell^n\Z) \arrow{d}{(f\times_s\eta)^!} \arrow{r}{\pi^*_{Y}} & D(Y_{\ov{s}},\Z/\ell^n\Z) \arrow{d}{f^!_{\ov{s}}} \\
    D(X\times_s \eta, \Z/\ell^n\Z) \arrow{r}{\pi^*_{X}} & D(X_{\ov{s}}, \Z/\ell^n\Z),
\end{tikzcd}
\] commutes (up to a canonical isomorphism);
\item\label{lemma:upper-shriek-pullback-2} the diagram
\[
\begin{tikzcd}
    D(Y\times_s \eta; \Z/\ell^n\Z) \arrow{d}{(f\times_s \eta)^!} \arrow{r}{{\sigma}_Y^*} & D(Y; \Z/\ell^n\Z) \arrow{d}{f^!} \\
    D(X\times_s \eta; \Z/\ell^n\Z) \arrow{r}{\sigma_X^*} & D(X; \Z/\ell^n\Z). 
\end{tikzcd}
\]
commutes (up to a canonical isomorphism) for every continuous section $\sigma\colon G_s \to G_\eta$;
\item\label{lemma:upper-shriek-pullback-3} the diagram
\[
\begin{tikzcd}
    D(Y; \Z/\ell^n\Z) \arrow{r}{p_Y^*} \arrow{d}{f^!} & D(Y\times_s \eta; \Z/\ell^n\Z) \arrow{d}{(f\times_s \eta)^!} \\
    D(X; \Z/\ell^n\Z)  \arrow{r}{p_X^*}& D(X\times_s \eta; \Z/\ell^n\Z) 
\end{tikzcd}
\]
commutes (up to a canonical isomorphism), where $p_X\colon X\times_s\eta\to X$ is the natural projection morphism (and the same for $p_Y$);

\item\label{lemma:upper-shriek-pullback-4} If $f$ is smooth of pure relative dimension $d$, there is a natural isomorphism $(f\times_s\eta)^!\simeq (f\times_s\eta)^*(d)[2d]$;

\item\label{lemma:upper-shriek-pullback-5} The natural morphism
\[
c_{\F,n, m}\colon (f\times_s \eta)^!\,\F \otimes^L_{\Z/\ell^n\Z} \Z/\ell^m\Z \to (f\times_s \eta)^!\left(\F \otimes^L_{\Z/\ell^n\Z} \Z/\ell^m\Z\right) 
\]
is an isomorphism for any $\F\in D(Y\times_s\eta; \Z/\ell^n\Z)$ and $n\geq m$;
\item\label{lemma:upper-shriek-pullback-6} $(f\times_s \eta)^!$ carries $D^b_{ctf}(Y\times_s\eta; \Z/\ell^n\Z)$ to $D^b_{ctf}(X\times_s\eta; \Z/\ell^n\Z)$.
\end{enumerate}
\end{lemma}
\begin{proof}
    (\ref{lemma:upper-shriek-pullback-1}) and (\ref{lemma:upper-shriek-pullback-2}) follow from \cite[Proposition 1.24]{Lu-Zheng} (the boundedness assumption can be dropped in our situation by using \cite[Remark 1.18]{Lu-Zheng} in place of \cite[Proposition 1.17]{Lu-Zheng}). (\ref{lemma:upper-shriek-pullback-3}) follows from \cite[Corollary 1.26]{Lu-Zheng} applied to $g=\rm{Id}_{\eta}$ and $M=\ud{\Z/\ell^n\Z}$. (\ref{lemma:upper-shriek-pullback-4}) follows from \cite[Proposition 1.23]{Lu-Zheng}. (\ref{lemma:upper-shriek-pullback-5}) follows from \cite[Proposition 1.25]{Lu-Zheng}. And (\ref{lemma:upper-shriek-pullback-6}) can be proven similarly to Lemma~\ref{lemma:pushforward-pullback}(\ref{lemma:pushforward-pullback-4}).
\end{proof}

For the next definition, we fix a finite type separated $k$-scheme with structure morphism $f\colon X\to \Spec k$.

\begin{defn} The {\it dualizing complex} $\omega_{X\times_s \eta}\in D(X\times_s \eta; \Z/\ell^n\Z)$ is defined to be $\omega_{X\times_s \eta}\coloneqq (f\times_s \eta)^!(\ud{\Z/\ell^n\Z})$. 

The {\it Verdier duality functor} $\bf{D}_{X\times_s \eta}(-)\colon D(X\times_s \eta; \Z/\ell^n\Z)^{\rm{op}} \to D(X\times_s\eta; \Z/\ell^n\Z)$ is defined as
\[
\bf{D}_{X\times_s \eta}(-) \coloneqq \rm{R}\cal{H}om_{\Z/\ell^n\Z}(-, \omega_{X\times_s \eta}).
\]
\end{defn}

\begin{rmk} Lemma~\ref{lemma:upper-shriek-pullback}(\ref{lemma:upper-shriek-pullback-3}), there is a natural isomorphism
\[
\omega_{X\times_s \eta} \simeq p^*_X\omega_X = p^*_X f^!(\ud{\Z/\ell^n\Z}),
\]
where $p_X\colon X\times_s \eta\to X$ is the natural projection morphism.
\end{rmk}

\begin{rmk}\label{rmk:verdier-duality-compatible} By Lemma~\ref{lemma:homs-after-base-change} and Lemma~\ref{lemma:upper-shriek-pullback}, the natural morphisms
\[
\sigma_X^*\bf{D}_{X\times_s\eta}(\F)\to \bf{D}_{X}(\sigma_X^*\F),
\]
\[
\pi^*_X\bf{D}_{X\times_s \eta}(\F)\to \bf{D}_{X_{\ov{s}}}(\pi_X^*\F)
\]
are isomorphisms for $\F\in D^b_{ctf}(X\times_s \eta; \Z/\ell^n\Z)$.
\end{rmk}

\begin{lemma}\label{lemma:biduality} Let $X$ be a finite type $k$-scheme, and let $\ell$ be a prime number invertible in $k$. Then the Verdier duality restricts to an equivalence 
\[
\bf{D}_{X\times_s \eta}\colon D^b_{ctf}(X\times_s \eta; \Z/\ell^n\Z)^{\rm{op}} \to D^b_{ctf}(X\times_s \eta; \Z/\ell^n\Z).
\]
\end{lemma}
\begin{proof}
    We need to show that, for every $\F\in D^b_{ctf}(X\times_s\eta; \Z/\ell^n\Z)$, the complex $\bf{D}_{X\times_s \eta}(\F)$ lies in $D^b_{ctf}(X\times_s\eta; \Z/\ell^n\Z)$ and the natural morphism
    \[
    \F \to \bf{D}_{X\times_s \eta}\left(\bf{D}_{X\times_s \eta}\left(\F\right)\right)
    \]
    is an isomorphism. The first claim follows from Corollary~\ref{cor:projection-hom}. The second claim can be proven after applying $\pi^*_X$ by Lemma~\ref{lemma:properties-Deligne}(\ref{lemma:properties-Deligne-3})). Then Remark~\ref{rmk:verdier-duality-compatible} ensures that it suffices to prove analogous fact for a complex $\G\in D^b_{ctf}(X_{\ov{s}}; \Z/\ell^n\Z)$. This follows from \cite[Theorem 9.6.1]{Lei-Fu}. 
\end{proof}

\subsection{Analytic nearby cycles}\label{section:nearby-cycles-finite-finite-type}

The main goal of this section is to define the functor of nearby cycles for  admissible formal $\O_K$-schemes. For this, we fix a completed algebraic closure $C\coloneqq \wdh{\ov{K}}$ of $K$, the ring of integers $\O_C\subset C$, and the residue field $\ov{k}$. \smallskip

We recall that, for every admissible formal $\O_K$-scheme $\X$, there is a morphism of topoi
\[
\lambda_{\X}\colon \X_\eta \to \X_s
\]
constructed in \cite[Lemma 3.5.1]{H3}. On the level of sites, this morphism sends on \'etale morphism $\sU_s\to \X_s$ to $\sU_\eta \to \X_\eta$, where $\sU \to \X$ is the unique \'etale map of formal schemes lifting $\sU_s \to \X_s$. \smallskip

Now we wish to define the nearby cycles functor. We consider the $(2, 1)$-commutative diagram:
\begin{equation}\label{eqn:nearby}
\begin{tikzcd}
\X_{\eta, \et}\arrow{d} \arrow{r}{\lambda_\X} & \X_{s, \et} \arrow{d}\\
\eta\arrow{r}{r} & s,
\end{tikzcd}
\end{equation}
where vertical arrows are the structure morphisms\footnote{Here, we implicitly identify $\eta$ with the \'etale topos $\Spa(K, \O_K)_\et$.}. By the universal property of the $2$-fiber products, Diagram~(\ref{eqn:nearby}) defines the morphism of topoi
\[
\Psi_\X \colon \X_{\eta, \et} \to (\X_{s} \times_s \eta)_\et.
\]

For the next definition, we fix a prime number $\ell$ and a positive integer $n\geq 1$. 

\begin{defn}\label{defn:nearby-cycles} For a nice formal $\O_K$-scheme, the \emph{nearby cycles functor} is the right derived functor
\[
\rm{R}\Psi_\X \colon \cal{D}(\X_\eta; \Z/\ell^n\Z)\to \cal{D}(\X_s\times_s \eta; \Z/\ell^n\Z).
\]
\end{defn}

\begin{warning} Even though it is not explicitly emphasized in the notation, the nearby cycles functor depends on a choice of a ground field $K$. Even the category $\cal{D}(\X_s\times_s \eta; \Z/\ell^n\Z)$ depends on a choice of $s$ and $\eta$ and not merely on $\X$ as an abstract formal scheme.
\end{warning}

Now we establish some basic properties of this functor. In particular, we will show that $\rm{R}\Psi_\X$ is indeed a canonical ``descent'' of $\rm{R}\lambda_{\X_{\O_C}, *}\circ b_{\X_\eta}^*$, where $b_{\X_\eta}\colon \X_{\wdh{\ov{\eta}}} \to \X_\eta$ is the natural projection morphism. But before we do this, we need to recall the definition of Zariski constructible sheaves:

\begin{defn} Let $X$ be a rigid-analytic space over a non-archimedean field $K$.
\begin{enumerate}
    \item An \'etale sheaf $\cal{F} \in \rm{Shv}(X_\et; \Z/\ell^n\Z)$ is {\it lisse} there exists an étale cover $\{U_i \to X\}_{i\in I}$ such that $\cal{F}|_{U_i}$ is the constant sheaf associated to a finitely generated $\Z/\ell^n\Z$-module.
    \item An étale sheaf $\cal{F} \in \rm{Shv}(X_\et; \Z/\ell^n\Z)$ is {\it Zariski-constructible} if $X$ admits a locally finite stratification $X =
    \sqcup_{i\in I} X_i$ into Zariski locally closed subsets $X_i$ such that $\cal{F}|_{X_i}$ is a lisse sheaf of $\Z/\ell^n\Z$-modules for all $i \in I$. \item A complex $\cal{F}\in \cal{D}(X; \Z/\ell^n\Z)$ is {\it bounded Zariski-constructible} if only finite number of cohomology sheaves $\cal{H}^i(\cal{F})$ are non-zero, and all of them are Zariski-constructible. We denote this category by $\cal{D}_{zc}^b(X; \Z/\ell^n\Z)$ 
    
    \item A complex $\cal{F}\in \cal{D}(X; \Z/\ell^n\Z)$ is {\it Zariski-constructible of finite tor dimension} if $\cal{F}$ is Zariski-constructible and the complex $\cal{F}\otimes^L_{\Z/\ell^n\Z} M$ is bounded for every finitely generated $\Z/\ell^n\Z$-module $M$. We denote this category by $\cal{D}_{zc, ftd}^b(X; \Z/\ell^n\Z)$. 
\end{enumerate}
\end{defn}

\begin{lemma}\label{lemma:compute-nearby-cycles} Let $\mf\colon \X\to \cY$ be a morphism of admissible formal $\O_K$-schemes, let $\ell$ be a prime number, and let $n$ be a positive integer. Then 
\begin{enumerate}
    \item\label{lemma:compute-nearby-cycles-1} the diagram 
    \[
\begin{tikzcd}
    \cal{D}(\X_\eta; \Z/\ell^n\Z) \arrow{d}{b_{\X_\eta}^*} \arrow{r}{\rm{R}\Psi_\X} & \cal{D}(\X_s\times_s\eta; \Z/\ell^n\Z) \arrow{d}{\pi_{\X_s}^*} \\
    \cal{D}(\X_{\wdh{\ov{\eta}}}; \Z/\ell^n\Z) \arrow{r}{\rm{R}\lambda_{\X_{\O_C}, *}} & \cal{D}(\X_{\ov{s}}; \Z/\ell^n\Z)
\end{tikzcd}
\] commutes (up to a canonical isomorphism);
\item\label{lemma:compute-nearby-cycles-2} the diagram
\[
\begin{tikzcd}
    \cal{D}(\X_{\eta}; \Z/\ell^n\Z)\arrow{d}{\rm{R}\mf_{\eta, *}}\arrow{r}{\rm{R}\Psi_\X} & \cal{D}(\X_s\times_s \eta; \Z/\ell^n\Z) \arrow{d}{\rm{R}(\mf_s\times_s \eta)_*} \\
    \cal{D}(\cY_\eta; \Z/\ell^n\Z)\arrow{r}{\rm{R}\Psi_\cY} & \cal{D}(\cY_s\times_s \eta; \Z/\ell^n\Z)
\end{tikzcd}
\]
commutes (up to a canonical isomorphism);
\item\label{lemma:compute-nearby-cycles-3} Let $K\subset K' \subset C$ be an extension of non-archimedean fields inducing an algebraic extension $k\subset k'$ on residue fields, and let $\eta'$ and $s'$ be the classifying topoi of the absolute Galois groups $G_{K'}$ and $G_{k'}$. Then the following diagram
\[
\begin{tikzcd}[column sep = huge]
    \cal{D}(\X_\eta; \Z/\ell^n\Z) \arrow{d}{b_{\eta, \eta'}^*} \arrow{r}{\rm{R}\Psi_\X} & \cal{D}(\X_s\times_s \eta; \Z/\ell^n\Z) \arrow{d}{b_{s, s'}^*}\\
    \cal{D}(\X_{\eta'}; \Z/\ell^n\Z) \arrow{r}{\rm{R}\Psi_{\X_{\O_{K'}}}} & \cal{D}(\X_{s'}\times_{s'} \eta'; \Z/\ell^n\Z)
\end{tikzcd}
\]
commutes (up to a canonical isomorphism), where the vertical functors are the natural pullbacks;
\item\label{lemma:compute-nearby-cycles-4} The natural morphism
\[
c_{\F,n, m}\colon \rm{R}\Psi_\X \F \otimes^L_{\Z/\ell^n\Z} \Z/\ell^m\Z \to \rm{R}\Psi_\X(\F \otimes^L_{\Z/\ell^n\Z} \Z/\ell^m\Z) 
\]
is an isomorphism for any $\F\in \cal{D}(X\times_s\eta; \Z/\ell^n\Z)$ and $n\geq m$;
\item\label{lemma:compute-nearby-cycles-5} If $\ell$ is invertible in $\O_K$ and $\rm{char}\, K=0$, then the functor $\rm{R}\Psi_\X$ carries $\cal{D}^b_{zc, ftd}(\X_\eta; \Z/\ell^n\Z)$ to $\cal{D}^b_{ctf}(\X_s\times_s\eta; \Z/\ell^n\Z)$;
\item\label{lemma:compute-nearby-cycles-6} The nearby cycles $\rm{R}\Psi_\X\colon \cal{D}(\X_\eta;\Z/\ell^n\Z) \to \cal{D}(\X_s\times_s \eta; \Z/\ell^n\Z)$ commutes with colimits. 
\end{enumerate}
\end{lemma}
\begin{proof}
    Before we start the proof, we note that $\rm{R}\lambda_{\X_{\O_C}, *}$ has finite cohomological dimension by \cite[Corollary 2.8.3]{H3}. We will freely use this in the proof. \smallskip
    
    $(\ref{lemma:compute-nearby-cycles-1})$ It can be seen explicitly using the explicit site-theoretic construction of $\X_s\times_s \eta$ from \cite[Exp. XI, \textsection 3]{deGabber}. Alternatively, $(\X_s\times_s \eta)_\et$ is coherent by \cite[Lemma 1.3]{Lu-Zheng} (or \cite[Exp.XI, Lemme 2.5]{deGabber}), and the proof of {\it loc. cit.} implies that $\Psi_\X\colon \X_{\eta, \et}\to (\X_s\times_s \eta)_{\et}$ is coherent. Furthermore, an argument analogous to that of Lemma~\ref{lemma:properties-Deligne}~(\ref{lemma:properties-Deligne-1}) implies that $\X_{\eta}\times_{\eta} \ov{\eta}\simeq \X_{\wdh{\ov{\eta}}}$ (use \cite[Proposition 2.4.4]{H3} in place of \cite[Lemma 8.3]{Morin}). Then the result follows from the base change result, see \cite[Proposition 1.17 and Remark 1.18]{Lu-Zheng}. \smallskip
    
    $(\ref{lemma:compute-nearby-cycles-2})$ This is formal. \smallskip
    
    $(\ref{lemma:compute-nearby-cycles-3})$ By Lemma~\ref{lemma:properties-Deligne}~(\ref{lemma:properties-Deligne-3}), it suffices to show that the natural morphism
    \[
    b_{s,s'}^*\circ \rm{R}\Psi_\X \to \rm{R}\Psi_{\X_{\O_K'}}\circ b_{\eta, \eta'}^*
    \]
    is an isomorphism after applying $\pi_{\X_{s'}}^*\colon \cal{D}(\X_{s'}\times_{s'}\eta';\Z/\ell^n\Z) \to \cal{D}(\X_{\ov{s'}}=\X_{\ov{s}}; \Z/\ell^n\Z)$. But then both compositions are canonically identified with 
    \[
    \rm{R}\lambda_{\X_{\O_C}, *}\circ p_{\X_\eta}^*
    \]
    by $(\ref{lemma:compute-nearby-cycles-1})$. \smallskip
    
    $(\ref{lemma:compute-nearby-cycles-4})$ By (\ref{lemma:compute-nearby-cycles-1}) and Lemma~\ref{lemma:properties-Deligne}~(\ref{lemma:properties-Deligne-3}), cohomological dimension of $\rm{R}\Psi_{\X}$ is bounded by the cohomological dimension of $\rm{R}\lambda_{\X_{\O_C},*}$. Therefore, the result follows from \cite[Corollary 1.20]{Lu-Zheng}.
     
    $(\ref{lemma:compute-nearby-cycles-5})$ By (\ref{lemma:compute-nearby-cycles-1}), it suffices to show that $\rm{R}\lambda_{\X_{\O_C}, *}$ carries $\cal{D}^b_{zc, ftd}(\X_{\wdh{\ov{\eta}}}; \Z/\ell^n\Z)$ to $\cal{D}^b_{ctf}(\X_{\ov{s}}; \Z/\ell^n\Z)$. We first show that $\rm{R}\lambda_{\X_{\O_C}, *}$ carries $\cal{D}^b_{zc}(\X_{\wdh{\ov{\eta}}}; \Z/\ell^n\Z)$ to $\cal{D}^b_{c}(\X_{\ov{s}}; \Z/\ell^n\Z)$. By \cite[Proposition 3.6]{Bhatt-Hansen} and Lemma~\ref{lemma:pushforward-pullback}(\ref{lemma:pushforward-pullback-4}), it suffices to show that $\rm{R}\lambda_{\X_{\O_C}, *}(\ud{M})\in \cal{D}^b_c(\X_{\ov{s}}; \Z/\ell^n\Z)$ for a finitely generated $\Z/\ell^n\Z$-module $M$. This follows from \cite[Proposition 3.11]{Huber-finiteness} or \cite[Theorem 1.1.2]{Berkovich-finiteness}. Now it is easy to see that $\rm{R}\lambda_{\X_{\O_C}, *}$ carries $\cal{D}^b_{zc, ftd}(\X_{\wdh{\ov{\eta}}}; \Z/\ell^n\Z)$ to $\cal{D}^b_{ctf}(\X_{\ov{s}}; \Z/\ell^n\Z)$ using the projection formula.
    
    (\ref{lemma:compute-nearby-cycles-6}) By \cite[Proposition 1.4.4.1(2)]{HA}, it suffices to show that $\rm{R}\Psi_\X$ commutes with (infinite) direct sums. Since $\pi^*_{\X_s}$ commutes with (infinite) direct sums and conservative, it suffices to show that $\rm{R}\lambda_{\X_{\O_C}, *}$ commutes with infinite direct sums. Now this is classical; for example, it follows from \cite[Theorem 1.1(i)]{Hansen-nearby}.
\end{proof}

\begin{rmk}\label{rmk:finiteness-char-p} The proof of Lemma~\ref{lemma:compute-nearby-cycles}(\ref{lemma:compute-nearby-cycles-5}) shows that $\rR\Psi_{\X} \Z/\ell^n\Z$ lies in $\cal{D}^b_{ctf}(\X_s\times_s \eta; \Z/\ell^n\Z)$ for any admissible formal $\O_K$-scheme and any prime number $\ell$ invertible in $\O_K$.
\end{rmk}

\begin{rmk}\label{rmk:action-coincide} Lemma~\ref{lemma:compute-nearby-cycles}(\ref{lemma:compute-nearby-cycles-2}) and Lemma~\ref{lemma:pushforward-pullback}(\ref{lemma:pushforward-pullback-1}) imply that, for an admissible formal $\O_K$-scheme $\X$ with the structure morphism $\mf\colon \X\to \Spf \O_K$ and $\F\in \cal{D}(\X_\eta, \Z/\ell^n\Z)$, we have 
\[
\rm{R}\Gamma\big(\X_{\wdh{\ov{\eta}}}, \F\big) \simeq \left(\rm{R}\mf_{\eta, *}\F\right)_{\wdh{\ov{\eta}}} \simeq \big(\rm{R}\left(\mf_s\times_s \eta\right)_* \rm{R}\Psi_{\X}\F\big)_{\wdh{\ov{\eta}}} \simeq \rm{R}\Gamma\left(\X_{\ov{s}}, \pi^*_{\X_s}\rm{R}\Psi_{\X} \F\right)
\]
compatibly with the $G_\eta$-actions on both sides. 
\end{rmk}

Now we wish to discuss an analogue of Remark~\ref{rmk:action-coincide} for compactly supported cohomology groups. We will not be able to establish a result in such generality, but we will prove a subtitute that is sufficient for all our purposes. The question turns out to be more subtle than Remark~\ref{rmk:action-coincide} since the lower shriek functors do not come from morphisms of topoi, so it is somewhat difficult to control the $G_\eta$-action on compactly supported cohomology groups. Before we discuss this, we record the following preliminary lemma.

\begin{lemma}\label{lemma:open-lower-shriek} Let $j\colon \sU \to \X$ be an open immersion of admissible formal $\O_K$-schemes, and $\ell$ a prime number. Then there is a natural isomorphism of functors
\[
(j_s\times_s \eta)_! \circ \rm{R}\Psi_{\sU}\simeq \rm{R}\Psi_{\X}\circ j_{\eta,!}
\]
as functors $D(\sU_\eta; \Z/\ell^n\Z) \to D(\X_s\times_s \eta;\Z/\ell^n\Z)$.
\end{lemma}
\begin{proof}
    The hard part is to construct the natural tranformation. First we note that there is a natural transformation $j_{\eta, !} \to \rm{R}j_{\eta, *}$ essentially by construction (see \cite[Definition 5.2.1(ii) and Proposition 5.2.4]{H3}). This induces a transformation
    \[
    \rm{R}\Psi_\X \circ j_{\eta, !} \to \rm{R}\Psi_{\X}\circ \rm{R}j_{\eta, *} \simeq \rm{R}(j_s\times_s \eta)_*\circ \rm{R}\Psi_{\sU}
    \]
    where the last isomorphism comes from Lemma~\ref{lemma:compute-nearby-cycles}(\ref{lemma:compute-nearby-cycles-2}). By adjunction, this gives a morphism
    \[
    (j_s\times_s\eta)^*\circ \rm{R}\Psi_X \circ j_{\eta,!} \to \rm{R}\Psi_{\sU}.
    \]
    This morphism is easily seen to be an isomorphism (after applying $\pi^*_{\sU_s}$ as always). By adjunction, its inverse defines a morphism
    \begin{equation}\label{eqn:transformation}
    (j_s\times_s \eta)_!\circ \rm{R}\Psi_{\sU} \to \rm{R}\Psi_{\X}\circ j_{\eta, !}.
    \end{equation}
    It suffices to show that this transformation is an isomorphism after applying $\pi^*_{\X_s}$ by Lemma~\ref{lemma:properties-Deligne}(\ref{lemma:properties-Deligne-3}). Therefore, using Lemma~\ref{lemma:shriek-pushforward-pullback}(\ref{lemma:shriek-pushforward-pullback-1}), Lemma~\ref{lemma:compute-nearby-cycles}(\ref{lemma:compute-nearby-cycles-1}), and \cite[Corollary 3.5.11(ii)]{H3}, one proves that the transformation (\ref{eqn:transformation}) is an isomorphism on $D^+(\sU_\eta, \Z/\ell^n\Z)$. The general case follows from the fact that all functors commute with colimits (see Lemma~\ref{lemma:compute-nearby-cycles}(\ref{lemma:compute-nearby-cycles-6})).
\end{proof}

\begin{lemma}\label{lemma:adjoint-to-nearby-cycles} Let $\X$ be an admissible separated formal $\O_K$-scheme with structure morphism $\mf\colon \X \to \Spf \O_K$, let $\ell$ be a prime number, and let $n\geq 1$ be a positive integer. Then:
\begin{enumerate}
    \item\label{lemma:adjoint-to-nearby-cycles-1} The nearby cycles functor $\rm{R}\Psi_\X\colon \cal{D}(\X_\eta; \Z/\ell^n\Z) \to \cal{D}(\X_s\times_s \eta; \Z/\ell^n\Z)$ has a right adjoint $\Psi_\X^!$;
    \item\label{lemma:adjoint-to-nearby-cycles-2} For $\F\in \cal{D}(\X_\eta; \Z/\ell^n\Z)$ and $\G\in \cal{D}(\X_s\times_s \eta; \Z/\ell^n\Z)$, there is a functorial isomorphism
    \[
    \rm{R}\Psi_{\X} \rm{R}\cal{H}om_{\Z/\ell^n\Z}(\F, \Psi_\X^! \G) \to \rm{R}\cal{H}om_{\Z/\ell^n\Z}(\rm{R}\Psi_\X \F, \G);
    \]
    \item\label{lemma:adjoint-to-nearby-cycles-3} If $\ell$ is invertible in $\O_K$, there is an isomorphism
    \[
    \Psi^!_\X\circ \mf_s^!\left(\ud{\Z/\ell^n\Z}\right) \simeq \mf_{\eta}^! \left(\ud{\Z/\ell^n\Z}\right);
    \]
    \item\label{lemma:adjoint-to-nearby-cycles-4} If $\ell$ is invertible in $\O_K$, there is a natural isomorphism of functors
    \[
    \rm{R}\Psi_{\X}\circ \bf{D}_{\X_\eta} \simeq \bf{D}_{\X_s\times_s\eta}\circ \rm{R}\Psi_{\X}.
    \]
    
\end{enumerate}
\end{lemma}
\begin{proof}
    (\ref{lemma:adjoint-to-nearby-cycles-1}) follows from the fact that $\rm{R}\Psi_\X$ commutes with colimits (see Lemma~\ref{lemma:compute-nearby-cycles}(\ref{lemma:compute-nearby-cycles-6})), the fact that both $\cal{D}(\X_\eta; \Z/\ell^n\Z)$ and $\cal{D}(\X_s\times_s\eta;\Z/\ell^n\Z)$ are presentable $\infty$-categories (see \cite[Proposition 1.3.5.21]{HA}), and the Adjoint Functor Theorem (see \cite[Corollary 5.5.2.9]{HTT}).  \smallskip
    
    (\ref{lemma:adjoint-to-nearby-cycles-2}) is essentially formal from the standard adjunctions and the projection formula for $\rm{R}\Psi_\X$. We refer to \cite[p. 8]{Hansen-nearby} and \cite[Corollary 4.3(2)]{Gaisin} for similar arguments (see \cite[Theorem 2.5]{Hansen-nearby} or \cite[Lemma 4.2]{Gaisin} for the proof of the projection formula). \smallskip
    
    (\ref{lemma:adjoint-to-nearby-cycles-3}) This can be proven similarly to \cite[Corollary 4.3(iii)]{Gaisin} using Lemma~\ref{lemma:open-lower-shriek} in place of \cite[Corollary 3.5.11]{H3}. First, the proof of \cite[Corollary 4.3(iii)]{Gaisin} shows that 
    \[
    \cal{E}xt^i_{\Z/\ell^n\Z}\left(\mf_\eta^!\left(\ud{\Z/\ell^n\Z}\right),\mf_\eta^!\left(\ud{\Z/\ell^n\Z}\right)\right) = 0
    \]
    for $i< 0$. Therefore, using the BBD gluing lemma, it suffices to construct such isomorphism locally (provided that it is compatible with open immersions). In the affinoid case, one reduces first to the case of a ball, where one can embed it into the projective space. Then the isomorphism comes from the combination of Lemma~\ref{lemma:open-lower-shriek} and Lemma~\ref{lemma:compute-nearby-cycles}(\ref{lemma:compute-nearby-cycles-2}). We refer to \cite[Corollary 4.3(iii) and Lemma 2.34]{Gaisin} for more detail. \smallskip
    
    (\ref{lemma:adjoint-to-nearby-cycles-4}) follows formally from (\ref{lemma:adjoint-to-nearby-cycles-2}) and (\ref{lemma:adjoint-to-nearby-cycles-3}). 
\end{proof}

\begin{thm}\label{thm:compact-pushforward-nearby-cycles} Let $K$ be a non-archimedean field of characteristic $0$, let $\mf \colon \X\to \cY$ be a morphism of separated admissible formal $\O_K$-schemes, let $\ell$ be a prime number invertible in $\O_K$, and let $\F\in D^b_{zc, ftd}(\X_\eta, \Z/\ell^n\Z)$ for some integer $n\geq 1$. Then there is a functorial isomorphism
\[
\rm{R}(\mf_{s}\times_s \eta)_{!}\circ \rm{R}\Psi_\X \F \simeq \rm{R}\Psi_{\cY} \circ \rm{R}f_{\eta, !} \F. 
\]
\end{thm}
\begin{proof}
    The claim follows from a sequence of isomorphisms:
    
    \begin{align*}
        \rm{R}\Psi_{\cY}\circ \rm{R}\mf_{\eta, !}\F & \simeq \rm{R}\Psi_{\cY}\circ \rm{R}\mf_{\eta, !} \circ \bf{D}_{\X_\eta}\circ \bf{D}_{\X_\eta} \F \\
        & \simeq \rm{R}\Psi_{\cY}\circ \bf{D}_{\cY_\eta}\circ \rm{R}\mf_{\eta, *}\circ \bf{D}_{\X_\eta} \F \\
        & \simeq \bf{D}_{\cY_s\times_s \eta} \circ \rm{R}\Psi_{\cY} \circ \rm{R}\mf_{\eta, *} \circ \bf{D}_{\X_\eta} \F \\
        & \simeq \bf{D}_{\cY_s\times_s \eta}\circ \rm{R}(\mf_{s}\times_s\eta)_* \circ \rm{R}\Psi_\X \circ \bf{D}_{\X_\eta} \F \\
        & \simeq \bf{D}_{\cY_s\times_s \eta}\circ \rm{R}(\mf_{s}\times_s\eta)_* \circ \bf{D}_{\X_s\times_s \eta}\circ \rm{R}\Psi_{\X} \F \\
        & \simeq \bf{D}_{\cY_s\times_s\eta} \circ \bf{D}_{\cY_s\times_s\eta} \circ \rm{R}(\mf_s\times_s\eta)_!  \circ \rm{R}\Psi_\X \F \\
        & \simeq \rm{R}(\mf_s\times_s\eta)_!  \circ \rm{R}\Psi_\X \F.
    \end{align*}
    Now we explain each isomorphism in more detail. The first isomorphism follows from \cite[Theorem 3.21(3)]{Bhatt-Hansen}. The second isomorphism follows from  \cite[Corollary 4.9(2)]{Gaisin}, the fact that $\bf{D}_{\X_\eta}\F$ is Zariski-constructible (see \cite[Corollary 3.14]{Bhatt-Hansen}), and the fact that Zariski-constructible complexes are constructible in the sense of \cite[Definition 3.1]{Gaisin} (this is not hard to deduce from \cite[Proposition 3.6]{Bhatt-Hansen} and \cite[Remark 3.2]{Gaisin}). The third isomorphism follows from Lemma~\ref{lemma:adjoint-to-nearby-cycles}(\ref{lemma:adjoint-to-nearby-cycles-4}). The fourth isomorphism follows from Lemma~\ref{lemma:compute-nearby-cycles}(\ref{lemma:adjoint-to-nearby-cycles-2}). The fifth isomorphism follows from Lemma~\ref{lemma:adjoint-to-nearby-cycles}(\ref{lemma:adjoint-to-nearby-cycles-4}). The sixth isomorphism follows from (the sheafified version of) the $(\rm{R}(\mf_{s}\times_s\eta)_{!}, (\mf_{s}\times_s \eta)^!)$-adjunction. The seventh isomorphism follows from Lemma~\ref{lemma:compute-nearby-cycles}(\ref{lemma:compute-nearby-cycles-5}) and Lemma~\ref{lemma:biduality}.
\end{proof}

\begin{rmk}\label{rmk:action-coincide-compact} Similarly to Remark~\ref{rmk:action-coincide}, Theorem~\ref{thm:compact-pushforward-nearby-cycles} and Lemma~\ref{lemma:shriek-pushforward-pullback}(\ref{lemma:shriek-pushforward-pullback-1}) imply that, for a non-archimedean field $K$ of characteristic $0$, an admissible separated formal $\O_K$-scheme $\X$ with the structure morphism $\mf\colon \X\to \Spf \O_K$, and $\F\in \cal{D}^b_{zc, ftd}(\X_\eta, \Z/\ell^n\Z)$, we have
\[
\rm{R}\Gamma_c\big(\X_{\wdh{\ov{\eta}}}, \F\big) \simeq  \rm{R}\Gamma_c\big(\X_{\ov{s}}, \pi^*_{\X_s}\rm{R}\Psi_{\X} \F\big)
\]
compatibly with the $G_\eta$-actions on both sides.
\end{rmk}

\begin{rmk} We expect that the methods of \cite{GW2} could be adapted to extend both Theorem~\ref{thm:compact-pushforward-nearby-cycles} and Remark~\ref{rmk:action-coincide-compact} to any non-archimedean field $K$, any prime number $\ell$, and any $\F\in \cal{D}^+(\X_\eta; \Z/\ell^n\Z)$. We do not pursue it as we never need this result in this paper. 
\end{rmk}

\subsection{Comparison of analytic and algebraic nearby cycles}

The main goal of this section is to compare the nearby cycles functor from Section~\ref{section:nearby-cycles-finite-finite-type} to the standard construction of algebraic nearby cycles. 

For the rest of this section, we fix a {\it henselian} rank-$1$ valuation ring $\O_K$ with fraction field $K$ and residue field $k$. We also fix its completed algebraic closure $C\coloneqq \wdh{\ov{K}}$. It is a non-archimedean field with ring of integers $\O_C$ and residue field $\ov{k}$, an algebraic closure of $k$. In what follows, we denote by $S$ the spectrum $\Spec \O_K$. \smallskip

We start by briefly reviewing the construction of the algebraic nearby cycles. Let $X$ be a finitely presented, flat $\O_K$-scheme. We consider the oriented fiber product $X_\et\overleftarrow{\times}_{S_\et}\eta$ (see \cite[Exp. XI, \textsection 1]{deGabber}), where the morphism $\eta \to S_\et$ is induced by a morphism of schemes $\Spec K\to \Spec \O_K$. Thus the $(2,1)$-commutative square
\[
\begin{tikzcd}
X_{\eta,\et} \arrow{r} \arrow{d}& X_\et \arrow{d}\\
\eta \arrow{r}& S_\et
\end{tikzcd}
\]
and the universal property of the oriented fiber products define the morphism of topoi
\[
\Psi^{\rm{alg}}_{X, \eta}\colon X_{\eta,\et} \to X_\et\overleftarrow{\times}_{S_\et}\eta.
\]

However, unlike the analytic situation, this does not finish the construction of the algebraic nearby cycles. To construct the desired nearby cycles, we consider the morphism of topoi 
\[
\pi\colon S_\et\to s
\]
induced by the functor of underlying sites $\pi^*\colon \text{\'Et.qcqs}(\Spec k)\to \text{\'Et}(\Spec \O_K)$ sending $\Spec \ov{A} \to \Spec k$ to the unique (finite \'etale) lift $\Spec A\to \Spec \O_K$. By functoriality of the oriented fiber products, it defines the morphism 
\[
X_{s, \et} \overleftarrow{\times}_{S_{\et}} S_{\et} \to X_{s, \et}\overleftarrow{\times}_s S_{\et} 
\]

\begin{lemma}\label{lemma:smaller-oriented-product} The natural morphism $X_{s, \et} \overleftarrow{\times}_{S_\et} S_{\et} \to X_{s, \et}\overleftarrow{\times}_s S_{\et}$ is an equivalence for any $k$-scheme $X$.
\end{lemma}
\begin{proof}
    Using the adjunction between $\pi\colon S_\et \to s$ and $i\colon s\to S_\et$, one checks that both oriented fiber products satisfy the same universal property. See \cite[Lemma 1.41]{Lu-Zheng} for details. 
\end{proof}

Recall that, for any topos $T$, the category $\rm{Hom}_{\cal{T}}(T, s)$ is a groupoid (see \cite[Remark 1.15]{Lu-Zheng}), so the oriented and $2$-fiber products over $s$ coincide. In particular, $X_{s,\et}\times_s S\simeq X_{s, \et}\overleftarrow{\times}_s S$. We combine it with Lemma~\ref{lemma:smaller-oriented-product} to get a canonical equivalence $X_{s, \et}\times_s S_\et \simeq X_{s, \et} \overleftarrow{\times}_{S_\et} S_{\et}$. We also define the morphism
\[
\overleftarrow{i_\eta}\colon (X_s\times_s\eta)_\et\to X_\et \overleftarrow{\times}_{S_\et} \eta
\]
as the composition
\[
(X_{s}\times_s\eta)_\et\to X_{s, \et}\times_s S_{\et}\simeq X_{s, \et} \overleftarrow{\times}_{S_\et} S_{\et} \to X_\et\overleftarrow{\times}_{S_\et} S_{\et},
\]
where the first and third maps come from functoriality of the $2$-fiber and oriented products respectively, and the middle equivalence is the equivalence discussed above. Finally, we are ready to define the algebraic nearby cycles:

\begin{defn}\label{defn:nearby-algebraic} The \emph{algebraic nearby cycles functor} 
\[
\rm{R}\Psi^{\rm{alg}}_X \colon \cal{D}(X_\eta; \Z/\ell^n\Z)\to \cal{D}(X_s\times_s \eta; \Z/\ell^n\Z)
\]
 is the composition
\[
\rm{R}\Psi^{\rm{alg}}_X \coloneqq \overleftarrow{i_\eta}^*\circ \rm{R}\Psi^{\rm{alg}}_{X, \eta, *}
\]
\end{defn}

\begin{lemma}\label{lemma:compute-nearby-cycles-algebraic} Let $X$ be a flat, finitely presented $\O_K$-scheme, $\ell$ a prime number, and $n\geq 1$ a positive integer. Then the diagram 
    \[
\begin{tikzcd}
    \cal{D}(X_\eta; \Z/\ell^n\Z) \arrow{d}{c_{X_\eta}^*} \arrow{rr}{\rm{R}\Psi^{\rm{alg}}_X} & & \cal{D}(\X_s\times_s\eta; \Z/\ell^n\Z) \arrow{d}{\pi_{X_s}^*} \\
    \cal{D}(X_{\ov{\eta}}; \Z/\ell^n\Z) \arrow{r}{\rm{R}\ov{j}_*}& \cal{D}(X_{\O_{\ov{K}}}; \Z/\ell^n\Z) \arrow{r}{\ov{i}^*}& \cal{D}(X_{\ov{s}}; \Z/\ell^n\Z)
\end{tikzcd}
\] 
commutes (up to a canonical isomorphism), where $X_{\ov{\eta}}$ is the generic fiber of $X_{\O_{\ov{K}}}$ and $X_{\ov{s}}$ is its special fiber. In other words, the algebraic nearby cycles coincide with the other constructions given in \cite[Exp. XIII]{SGA7_2} and \cite[Section 4.2]{H3}
\end{lemma}
\begin{proof}
    The easiest way to show the claim is to use the explicit construction of the oriented fiber product from \cite[Exp. XI, \textsection 1]{deGabber}. If $\O_K$ is discretely valued, this is explained in \cite[(1.2)]{Illusie-Thom}. In general the same argument applies.
\end{proof}

Now we assume that $\O_K$ is {\it complete} with a choice of a pseudo-uniformizer $\varpi\in \O_K$. We would like to compare the algebraic nearby cycles for a flat, finitely presented $\O_K$-scheme $X$ with the analytic nearby cycles for its $\varpi$-adic completion $\wdh{X}$ considered as an admissible $\O_K$-scheme.  \smallskip

The first step is to construct the comparison morphism. For this, we recall that there are two different analytic generic fibers associated to $X$. The first one $X_\eta^{\rm{an}}$ is obtained by taking the analytification of algebraic generic fiber, this comes with the natural morphism of \'etale topoi
\[
\iota\colon X_{\eta, \et}^{\rm{an}} \to X_{\eta, \et}.
\]
The other generic fiber $\wdh{X}_\eta$ is the adic generic fiber of the admissible formal scheme $\wdh{X}$. This comes with the natural morphism 
\[
\wdh{X}_\eta\to X_\eta^{\rm{an}}
\]
that is an open immersion for a separated $X$ (see \cite[Theorem 5.3.1]{conrad}). By passing to the associated \'etale topoi, we get the morphism
\[
\gamma_X\colon \wdh{X}_{\eta, \et}\to X_{\eta, \et}^{\rm{an}}.
\]
By composing it with $\iota$, we get the morphism
\[
\alpha\colon \wdh{X}_{\eta, \et} \to X_{\eta, \et}. 
\]
We note the diagram of topoi 
\begin{equation}\label{eqn:does-not-commute}
\begin{tikzcd}
\wdh{X}_{\eta, \et}\arrow{r}{\Psi_{\wdh{X}}}\arrow{d}{\a} & (X_s\times_s \eta)_\et \arrow{d}{\overleftarrow{i_\eta}}\\
X_{\eta, \et} \arrow{r}{\Psi_{X, \eta}^{\rm{alg}}} & X_\et\overleftarrow{\times}_{S_\et} \eta. 
\end{tikzcd}
\end{equation}
does not commute. However, there is a non-invertible $2$-tranformation
\[
\gamma\colon \Psi_{X, \eta}^{\rm{alg}}\circ \a\to \overleftarrow{i_{\eta}} \circ \Psi_{\wdh{X}}.
\]
To construct it, we consider the natural projections $q_X\colon X_\et\overleftarrow{\times}_{S_\et} \eta \to X_\et$ and $q_\eta\colon X_\et\overleftarrow{\times}_{S_\et} \eta \to \eta$. By the universal property of oriented fiber products, it suffices to define the transformation $\gamma$ after applying $q_\eta$ and $q_X$ (in a compatible way). One sees that $q_\eta \circ \Psi_{X, \eta}^{\rm{alg}}\circ \a$ is canonically identified with $q_\eta \circ \overleftarrow{i_{\eta}} \circ \Psi_{\wdh{X}}$. And the transformation 
\[
\gamma_X \colon q_X \circ \Psi_{X, \eta}^{\rm{alg}}\circ \a \to q_X \circ \overleftarrow{i_{\eta}} \circ \Psi_{\wdh{X}}
\]
is induced (on the level of sites) by the natural transformation
\[
\gamma_U\colon \wdh{U}_\eta=(q_X \circ \overleftarrow{i_{\eta}} \circ \Psi_{\wdh{X}})^*\left(U\right) \to (q_X \circ \Psi_{X, \eta}^{\rm{alg}}\circ \a)^*(U) = U_\eta^{\rm{an}}\times_{X_\eta^{\rm{an}}} \wdh{X}_\eta.
\] 

The $2$-morphism $\gamma$ defines the natural transformation of functors
\[
\rm{R}\Psi_{X, \eta, *}^{\rm{alg}}\circ \rm{R}\a_*\to \rm{R}\overleftarrow{i_{\eta}}_{\ast}\circ \rm{R}\Psi_{\wdh{X}}
\]
that, by adjunction, defines the following natural transformation of functors
\[
c\colon \overleftarrow{i_{\eta}}^*\circ  \rm{R}\Psi_{X, \eta, *}^{\rm{alg}} \to \rm{R}\Psi_{\wdh{X}}\circ \a^*.
\]
Note that the source of $c$ is by definition equal to $\rm{R}\Psi_{X}^{\rm{alg}}$, so $c$ can be rewritten as the natural tranformation
\[
c\colon \rm{R}\Psi_{X}^{\rm{alg}}\to \rm{R}\Psi_{\wdh{X}}\circ \a^*.
\]

\begin{thm}\label{thm:comparison} Let $\O_K$ be complete rank-$1$ valuation ring, $X$ a flat, finitely presented $\O_K$-scheme, $\ell$ a prime number, and $n$ an integer $\geq 1$. Then the natural morphism
\[
c\colon \rm{R}\Psi_X^{\rm{alg}} \left(\F\right) \to \rm{R}\Psi_{\wdh{X}} \left(\a^* \F\right) 
\]
is an isomorphism for any $\F\in \cal{D}(X_\eta; \Z/\ell^n\Z)$.
\end{thm}
\begin{proof}
    By Lemma~\ref{lemma:properties-Deligne}(\ref{lemma:properties-Deligne-3}), it suffices to show that $c$ is an isomorphism after applying $\pi_{X_s}^*$. Now using Lemma~\ref{lemma:compute-nearby-cycles}, the question boils down to the following one: for a flat, finitely presented $\O_{\ov{K}}$-scheme $X$, the natural morphism
    \[
    d\colon i^*\rm{R}j_* \F \to \rm{R}\nu_*(\a^*\F)
    \]
    is an isomorphism for any $\F\in \cal{D}(X; \Z/\ell^n\Z)$, $j\colon X_{\ov{K}} \to X$ the natural open immersion of the generic fiber of $X$ into $X$, $i\colon X_{s} \to X$ the natural closed immersion of the special fiber of $X$, and $\nu\colon \wdh{X}_{\eta, \et} \to X_{s, \et}$ is the natural morphism between the \'etale topoi of the adic generic fiber of a formal scheme $\wdh{X}$ to its special fiber. Now $d$ is an isomorphism by \cite[Theorem 3.5.13]{H3} for bounded below complexes. The result extends formally to the unbounded case since both functors are of finite cohomological dimension. 
\end{proof}

\section{Adic and rational coefficients}\label{appendix:adic}

The main goal of this Appendix is to review the theory of ``derived categories with $\Z_\ell$ and $\Q_\ell$-coefficients'' in the generality needed for the purposes of this paper. We pay extra attention to the categories of $\Z_\ell$ and $\Q_\ell$ complexes on Deligne's topos $X\times_s \eta$. \smallskip

Our approach is based on the theory of $\infty$-categories. For the rest of the section, we fix a prime number $\ell$. In this section, we freely identify $(2, 1)$-categories with their Duskin nerves considered as $\infty$-categories (see \cite[\href{https://kerodon.net/tag/00AC}{Tag 00AC}]{kerodon}). We will also freely use the notions of $\infty$-categorical limit and colimit (see \cite[\href{https://kerodon.net/tag/02H0}{Tag 02H0}]{kerodon} for some general discussion). We denote by $\cal{T}$ the $2$-category of topoi and by $\rm{Pith}(\cal{T})$ the associated $(2, 1)$-category, see \cite[\href{https://kerodon.net/tag/00AL}{Tag 00AL}]{kerodon}. \smallskip

\subsection{Adic complexes on a general topos}

The main goal of this section is to discuss the general notion of ``complexes of $\Z_\ell$ and $\Q_\ell$ sheaves'' on a topos. \smallskip

\begin{defn} The {\it $\infty$-derived category of $\Z_\ell$ sheaves} $\cal{D}(T; \Z_\ell)$ on a topos $T$ is the limit
\[
\cal{D}(T; \Z_\ell) \coloneqq \lim_n \cal{D}(T; \Z/\ell^n\Z).
\]
We denote its homotopy category by $D(T; \Z_\ell)\coloneqq \rm{h}\cal{D}(T; \Z_\ell)$.

The {\it $\infty$-derived category of sheaves of $\Q_\ell$-modules} $\cal{D}(T; \Q_\ell)$ on a topos $T$ is the localization\footnote{See \cite[\href{https://kerodon.net/tag/01ME}{Tag 01ME}]{kerodon}
 for the notion of a localization in the $\infty$-categorical context.} $\cal{D}(T, \Q_\ell)\coloneqq \cal{D}(T; \Z_\ell)\left[\frac{1}{\ell}\right]$. We denote its homotopy category by $D(T; \Q_\ell)\coloneqq \rm{h}\cal{D}(T; \Q_\ell)$
\end{defn}

\begin{rmk} An object $\F\in \cal{D}(T; \Z_\ell)$ is a sequence of objects $\F_n\in \cal{D}(T; \Z/\ell^n\Z)$ equipped with isomorphisms $\F_n\otimes^{L}_{\Z/\ell^n\Z} \Z/\ell^{n-1}\Z \simeq \F_{n-1}$. We informally denote the object $\F$ as $``\lim_n\text{''} \F_n$. 
\end{rmk}

Now we wish to show that the formation of $\cal{D}(T; \Z_\ell)$ and $\cal{D}(T; \Q_\ell)$ are $\infty$-functorial in $T$. For this, it will be convenient to identify $\cal{D}(T; \Z_\ell)$ with a subcategory of $\cal{D}(T; \Z)$.

\begin{defn} An object $\F\in \cal{D}(T; \Z)$ is $\ell$-adically derived complete if the natural morphism
\[
\F \to \lim_n \left(\F\otimes^{L}_\Z \Z/\ell^n\Z\right)
\]
is an isomorphism. We denote by $\cal{D}_{\ell}(T; \Z)$ the full subcategory of $\cal{D}(T; \Z)$ that consists of $\ell$-adic derived complete objects.
\end{defn}

\begin{lemma}\label{lemma:ell-adic=derived-complete} Let $T$ be a topos. Then the natural morphism
\[
\cal{D}_{\ell}(T; \Z) \to \cal{D}(T; \Z_\ell)
\]
is an equivalence.
\end{lemma}
\begin{proof}
    The proof is completely analogous to \cite[Proposition 4.3.9]{GL}.
\end{proof}

Now we recall that the assignment of the $\infty$-category $\cal{D}(T, \Z)$ to a topos $T\in \cal{T}$ can be made into an $\infty$-functor
\[
\cal{D}(-; \Z)_*\colon \rm{Pith}(\cal{T}) \to \cal{C}at_{\infty}
\]
that, on vertices, associates to a topos $T$ the $\infty$-category $\cal{D}(T; \Z)$ and, on edges, sends a morphism $f\colon T'\to T$ to $\rm{R}f_*\colon \cal{D}(T'; \Z) \to \cal{D}(T; \Z)$. Since $\rm{R}f_*$ preserves $\ell$-adically derived complete objects by \cite[\href{https://stacks.math.columbia.edu/tag/099J}{Tag 099J}]{stacks-project}, we conclude that $\cal{D}(-; \Z)_*$ restricts to an $\infty$-functor
\[
\cal{D}(-; \Z_\ell)_*\colon \rm{Pith}(\cal{T}) \to \cal{C}at_{\infty}
\]
that sends a topos $T$ to $\cal{D}_\ell(T; \Z)\simeq \cal{D}(T; \Z_\ell)$ (see Lemma~\ref{lemma:ell-adic=derived-complete}). By passing to adjoints, we get an $\infty$-functor
\[
\cal{D}(-; \Z_\ell)^*\colon \rm{Pith}(\cal{T})^{\rm{op}} \to \cal{C}at_{\infty}.
\]
After localizing at $\ell$, we also get the $\infty$-functor
\[
\cal{D}(-; \Q_\ell)^*\colon \rm{Pith}(\cal{T})^{\rm{op}} \to \cal{C}at_{\infty}.
\]

\begin{rmk}\label{rmk:functor-adic} Lemma~\ref{lemma:ell-adic=derived-complete} and \cite[\href{https://stacks.math.columbia.edu/tag/0B54}{Tag 0B54}]{stacks-project} formally imply that, for a morphism of topoi $f\colon T \to T'$ and objects  $\F=``\lim_n\F_n\text{''}\in \cal{D}(T; \Z_\ell)$ and $\G=``\lim_n \G_n\text{''}\in \cal{D}(T'; \Z_\ell)$, there are formulas
\[
\rm{R}f_*\F = ``\lim_n\rm{R}f_* \F_n\text{''}  \in \cal{D}(T'; \Z_\ell),
\]
\[
f^*\G = ``\lim_n f^*\G_n\text{''}  \in \cal{D}(T; \Z_\ell).
\]
\end{rmk}

\begin{rmk}\label{rmk:2-functor} By composing the $\infty$-functor
\[
\cal{D}(-; \Z_\ell)^*\colon \rm{Pith}(\cal{T})^{\rm{op}} \to \cal{C}at_{\infty}
\]
with the functor $\rm{h}(-)\colon \cal{C}at_{\infty} \to \rm{Pith}(\cal{C}at)$ that sends an $\infty$-category $\cal{C}$ to its homotopy category $\rm{h}\cal{C}$, we get a $2$-functor
\[
D(-; \Z_\ell)^*\colon  \rm{Pith}(\cal{T})^{\rm{op}} \to \rm{Pith}(\cal{C}at)
\]
that sends a topos $T$ to the triangulated category $D(T; \Z_\ell)$. The same applies to $D(-; \Q_\ell)$. 
\end{rmk}

\subsection{Adic complexes on Deligne's topos} 

The main goal of this section is to apply the constructions of the previous section to Deligne's topos defined in Appendix~\ref{appendix:deligne}. \smallskip

For the rest of the section, we fix a non-archimedean field $K$ with the residue field $k$, a finite type $k$-scheme $X$, and a prime number $\ell$ {\it invertible} in $k$.\smallskip 

We start with the observation that essentially all the results of Apendix~\ref{appendix:deligne} formally generalize to the case of adic coefficients:

\begin{rmk1}\label{rmk:adic-good} Using Remark~\ref{rmk:functor-adic}, we extend the functors $\pi^*_{X}$, $\sigma_X^*$, $p_X^*$, $\rm{R}(f\times_s\eta)_*$, $\rm{R}\Psi_\X$ and $\rm{R}\Psi^{\rm{alg}}_{X}$ to the setting of $\Z_\ell$ and $\Q_\ell$ coefficients. By passing to the limit, one can easily check that the results of Lemma~\ref{lemma:pushforward-pullback}, Lemma~\ref{lemma:compute-nearby-cycles}, Lemma~\ref{lemma:compute-nearby-cycles-algebraic}, and Theorem~\ref{thm:comparison} hold with $\Z_\ell$ and $\Q_\ell$ coefficients. 
\end{rmk1}

Now we show that any sheaf $\F\in D(X\times_s \eta; \Z_\ell)$ (resp. $\F\in D(X\times_s \eta; \Q_\ell)$) admits an ``action''of $G_\eta$ after applying the pullback functor $\pi^*_X\colon D(X\times_s \eta; \Z_\ell) \to D(X_{\ov{s}}; \Z_\ell)$:

\begin{construction}\label{construction:action-adic} Using $2$-functoriality of $D(T; \Z_\ell)$ (resp.\,$D(T; \Q_\ell)$) established in Remark~\ref{rmk:2-functor}, we can repeat Construction~\ref{construction:action} for $\Z_\ell$-coefficients (resp. $\Q_\ell$-coefficients). More precisely, for an object $\F\in D(X\times_s\eta; \Z_\ell)$ (resp. $\F\in D(X\times_s\eta; \Q_\ell)$), we get a family of isomorphisms $\rho_g\colon \ov{g}^*\pi_X^*\F \to \pi_X^*\F$ such that $\rho_e=\rm{Id}$ and the diagram
\[
\begin{tikzcd}
\ov{g}^*\ov{h}^*\pi^*_X\F \arrow{d}{\rm{iso}} \arrow{r}{\ov{g}^*(\rho_h)}& \ov{g}^*\pi^*_X\F \arrow{d}{\rho_g}\\
(\ov{gh}^*) \pi^*_X \F \arrow{r}{\rho_{gh}}& \pi^*_X\F
\end{tikzcd}
\]
commutes for every $g, h\in G_\eta$. By restricting to the inertia subgroup $I\subset G_\eta$, we get a homomorphism
\[
\rho\colon I \to \rm{Aut}(\pi^*_X\F)
\]
for any $\F\in D(X\times_s \eta; \Z_\ell)$ (resp. $\F\in D(X\times_s \eta; \Q_\ell)$). 
\end{construction}

\begin{defn}\label{defn:constructible-deligne} An object $\F\in \cal{D}(X\times_s \eta; \Z_\ell)$ is called {\it constructible} if $\F\otimes^L_{\Z_\ell} \bf{F}_\ell \in \cal{D}^b_{ctf}(X\times_s \eta; \bf{F}_\ell)$. We denote by $\cal{D}^b_c(X\times_s \eta; \Z_\ell)$ the full $\infty$-subcategory of $\cal{D}(X\times_s \eta; \Z_\ell)$ consisting of constructible objects, and by $D^b_c(X\times_s \eta; \Z_\ell)$ its homotopy category.

We define {\it the bounded derived category of constructible $\Q_\ell$-sheaves} $\cal{D}^b_c(X\times_s \eta; \Q_\ell)\coloneqq \cal{D}^b_c(X\times_s \eta; \Z_\ell)[\frac{1}{\ell}]$ as the evident localization of $\cal{D}^b_c(X\times_s \eta; \Z_\ell)$. We denote by $D^b_c(X\times_s \eta; \Q_\ell)$ the homotopy category of $\cal{D}^b_c(X\times_s \eta; \Q_\ell)$. 
\end{defn}

\begin{rmk} It is straighforward to check that $\F\in D(X\times_s\eta; \Z_\ell)$ is constructible if and only if $\pi^*_X\F\in D(X_{\ov{s}}; \Z_\ell)$ is constructible (in the usual sense). 
\end{rmk}

\begin{rmk} Lemma~\ref{lemma:homs-after-base-change} easily generalizes to $\Z_\ell$ and $\Q_\ell$-sheaves. 
\end{rmk}

We denote by $f\colon X \to k$ the structure morphism of $X$. Then for $\Lambda=\Z_\ell$ or $\Q_\ell$ and $\F, \G \in D^b_c(X\times_s \eta; \Lambda)$, we set
\[
\rm{RHom}_{/\eta, \Lambda}(\F, \G) \coloneqq \rm{R}(f\times_s \eta)_*\rm{R}\cal{H}om_{\Lambda}(\F, \G)\in D^b_c(\eta; \Lambda).
\]

\begin{lemma}\label{lemma:hom-formula-rational} Let $X$ be a finite type $k$-scheme, let $\ell$ be a prime number invertible in $\O_K$,  let $\Lambda=\Z_\ell$ or $\Q_\ell$, and let $\F, \G \in D^b_{c}(X\times_s\eta; \Lambda)$. Then
\[
\rm{RHom}_{\Lambda}\left(\F, \G\right) \simeq \rm{R}\Gamma_{\rm{cont}}\left(G_\eta, \rm{RHom}_{/\eta, \Lambda}\left(\F, \G\right)\right).
\]
\end{lemma}
\begin{proof}
    It follows from Lemma~\ref{lemma:Hom-formula} by passing to a cofiltered limit, and then filtered colimit.
\end{proof}

\begin{cor}\label{cor:compute-hom-Q-ell} In the notation of Lemma~\ref{lemma:hom-formula-rational}, assume that $\rm{RHom}_{\Lambda}(\pi^*_X \F, \pi^*_X\G) \in D^{\geq 0}_c(X_{\ov{s}}; \Lambda)$. Then
\[
\rm{Hom}_{\Lambda}(\F, \G) \simeq \rm{Hom}_{\Lambda}(\pi^*_X\F, \pi^*_X\G)^{G_\eta}.
\]
\end{cor}

\begin{cor}\label{cor:section-finite-morphism} Let $f\colon X \to Y$ be a finite surjective morphism of finite type $k$-schemes. Then there is a dense open $U\subset Y$ such that the natural morphism $\Q_{\ell}|_{U\times_s \eta} \to \big((f\times_s \eta)_* \Q_\ell\big)|_{U\times_s \eta}$ admits a splitting.
\end{cor}
\begin{proof}
    Since the universal homeomorphisms do not change the \'etale topos, we can assume that $X$ and $Y$ are reduced. In this case, we can find a dense open subscheme $U\subset Y$ such that $f|_{f^{-1}(U)}$  factors as the composition of a universal homeomorphism followed by a surjective finite étale map. For the purpose of proving this lemma, we can assume that $U=Y$. Using that universal homeomorphisms do not change the \'etale topos, we can reduce to the case that $f$ is a surjective finite \'etale map. In this case, the natural morphism $(f\times_s \eta)_! \to (f\times_s \eta)_*$ is an isomorphism because it could be checked after applying $\pi_Y^*$. Therefore, the counit of the adjunction defines a map $\tr_f\colon (f\times_s\eta)_*\Q_\ell \to \Q_\ell$ such that the composition 
    \[
    \Q_\ell \to (f\times_s\eta)_*\Q_\ell \xr{\tr_f} \Q_\ell
    \]
    is equal to $\deg f\neq 0$ (this claim can be checked after applying $\pi_Y^*$ due to Corollary~\ref{cor:compute-hom-Q-ell}). Therefore, $\frac{\tr_f}{\deg f}$ gives the desired splitting. 
\end{proof}

We next discuss the standard $t$-structure for sheaves on $X\times_s \eta$. This is a little bit subtle, because the individual categories $\cal{D}^b_{ctf}(X\times_s \eta; \Z/\ell^n\Z)$ do not admit natural $t$-structures when $n>1$.\footnote{As usual, this is ``because" $\mathrm{Perf}(\mathbf{Z}/\ell^n \mathbf{Z})$ does not admit any natural t-structure for $n>1$.} 

\begin{lemma}\label{lemma:constructible-t-structure} Let $X$ be a finite type $k$-scheme, and $\ell$ a prime number invertible in $k$. Then the triangulated category $D^b_c(X\times_s \eta; \Z_\ell)$ admits a \emph{standard} $t$-structure:
\begin{enumerate}
    \item $D^{\leq 0}_c(X\times_s \eta; \Z_\ell)$ is the full subcategory of $D^b_c(X\times_s \eta; \Z_\ell)$ consisting of objects $\F$ such that $\pi^*_X\F\in D^{\leq 0}_c(X_{\ov{s}}; \Z_\ell)$;
    \item $D^{\geq 0}_c(X\times_s \eta; \Z_\ell)$ is the full subcategory of $D^b_c(X\times_s \eta; \Z_\ell)$ consisting of objects $\F$ such that $\pi^*_X\F\in D^{\geq 0}_c(X_{\ov{s}}; \Z_\ell)$.
\end{enumerate}
\end{lemma}
\begin{proof}
    We note that Corollary~\ref{cor:compute-hom-Q-ell} implies that $\rm{Hom}_{\Z_\ell}(\F, \G)=0$ for $\F\in D^{\leq 0}_c(X\times_s \eta; \Z_
    \ell)$ and $\G\in D^{\geq 1}_c(X\times_s \eta; \Z_\ell)$. Thus the only non-trivial part of the definition of a $t$-structure one needs to verify is that every object $\F\in D^b_c(X\times_s \eta; \Z_\ell)$ fits into an exact triangle
    \[
    \F' \to \F \to \F''
    \]
    with $\F'\in D^{\leq 0}_c(X\times_s\eta; \Z_\ell)$ and $\F''\in D^{\geq 1}_c(X\times_s \eta; \Z_\ell)$. For this, we note that the proof of the analogous fact in \cite[Proposition 2.3.6.1]{GL} goes through with little changes; we leave details to the interested reader.    
\end{proof}
    
\begin{cor}\label{cor:constructible-t-structure-rationally} Let $X$ be a finite type $k$-scheme, and $\ell$ a prime number invertible in $k$. Then the triangulated category $D^b_c\left(X\times_s \eta; \Q_\ell\right)$ admits a \emph{standard} $t$-structure:
\begin{enumerate}
    \item $D^{\leq 0}_c\left(X\times_s \eta; \Q_\ell\right)$ is the full subcategory of $D^b_c(X\times_s \eta; \Q_\ell)$ consisting of objects $\F$ such that $\pi^*_X\F\in D^{\leq 0}_c(X_{\ov{s}}; \Q_\ell)$;
    \item $D^{\geq 0}_c(X\times_s \eta; \Q_\ell)$ is the full subcategory of $D^b_c(X\times_s \eta; \Q_\ell)$ consisting of objects $\F$ such that $\pi^*_X\F\in D^{\geq 0}_c(X_{\ov{s}}; \Q_\ell)$.
\end{enumerate}
\end{cor}

\begin{cor}\label{cor:splitting-after-field-extension} Let $K \subset K'$ be a finite extension of non-archimedean fields, let $X$ be a finite type $k$-scheme, let $b\colon X_{s'} \times_{s'} \eta' \to X_s \times_s \eta$ be the morphism of topoi from Notation~\ref{notation:field-extension}, and let $\varphi \colon \F \to \G$ be a morphism in $D^b_c(Y\times_s \eta; \Q_\ell)^{\heartsuit}$. Then there exists a morphism $\psi \colon \G \to \F$ such that $\psi \circ \varphi = \id$ if and only if there exists a morphism $\psi'\colon b^*\G \to b^*\F$ such that $\psi' \circ b^*(\varphi) = \id$. 
\end{cor}
\begin{proof}
    Let $G_K$ and $G_{K'}$ be the absolute Galois groups of $K$ and $K'$ respectively. Then $G_{K'}$ can be identified with a finite index subgroup of $G_K$. Let us denote by $g_1, \dots, g_n \in G_K$ a choice of representatives for every residue element in $G_K/G_{K'}$. Then Lemma~\ref{lemma:hom-formula-rational} implies that $\psi'$ can be identified with a $G_{K'}$-invariant element in $\Hom(\pi_X^* \F, \pi_X^*\G)$ such that $\pi_X^*(\psi')\circ \pi_X^*(\varphi) = \id$. Our goal is to construct a $G_K$-invariant element $\psi$ in $\Hom(\pi_X^* \F, \pi_X^*\G)$ such that $\pi_X^*(\psi)\circ \pi_X^*(\varphi) = \id$. Using that $\varphi$ is $G_K$-invariant, one easily checks that $\psi \coloneqq \frac{1}{n}\sum_{i=1}^n g_i(\psi')$ does the job.
\end{proof}

We next discuss ``local systems'' on the topos $X\times_s \eta$. For the next definition, we fix a finite type $k$-scheme $X$ and a ring $\Lambda\in \{\Q_\ell, \Z_\ell, \Z/\ell^n\Z\}$ for a prime number $\ell$ invertible in $k$. 

\begin{defn}\label{defn:lisse-objects} An object $\F\in D^b_c(X\times_s \eta; \Lambda)$ is {\it lisse} if $\pi^*_X\F\in D^b_c(X_{\ov{s}}; \Lambda)$ has lisse cohomology sheaves.

An object $\F\in D^b_c(X\times_s \eta; \Lambda)$ is a {\it $\Lambda$-local system} if it lies in the heart of the standard $t$-structure, lisse, and all stalks of $\pi^*_X\F$ are finite flat $\Lambda$-modules. 
\end{defn}

\begin{lemma}\label{lemma:lisse-stratification} Let $X$ a finite type $k$-scheme, $\ell$ a prime number invertible in $k$, $\Lambda\in \{\Q_\ell, \Z_\ell, \Z/\ell^n\Z\}$, and $\F\in D^b_c(X\times_s \eta;\Lambda)$. Then there is a finite stratification $X=\bigsqcup_{i\in I} X_i$ such that $\F|_{X_i\times_s \eta}$ is lisse and $\big(X_{i, \ov{k}}\big)_{\rm{red}}$ is smooth for each $i\in I$.
\end{lemma}
\begin{proof}
    The case of $\Q_\ell$-coefficients easily reduces to the case of $\Z_\ell$-coefficients by choosing a $\Z_\ell$-lattice. Now, for any $\F\in D^b_c(X\times_s \eta; \Z_\ell)$, $\pi^*_X\F$ has lisse cohomology groups if and only if 
    \[
    \left(\pi^*_X \F\right)\otimes^L_{\Z_\ell} \bf{F}_\ell \simeq \pi^*_X \left(\F\otimes^L_{\Z_\ell}\bf{F}_\ell\right)
    \]
    has lisse cohomology sheaves. Therefore, it suffices to prove the claim for $\Lambda=\Z/\ell^n\Z$. \smallskip
    
    By noetherian induction, it suffices to show that, for each generic point $\eta\in X$, there is an open $\eta\in U \subset X$ such that $\pi^*_X\F|_{U_{\ov{k}}}\in D^b_c\left(X_{\ov{s}}; \Z/\ell^n\Z\right)$ has locally constant cohomology sheaves and $\left(U_{\ov{k}}\right)_{\rm{red}}$ is smooth. \smallskip
    
    The proof of Lemma~\ref{lemma:homs-after-base-change} ensures that cohomology sheaves of $\F\in D^b_c\left(X\times_s \eta; \Z/\ell^n\Z\right)$ are constructible in the sense of \cite{Lu-Zheng} (see the discussion after \cite[Corollary 1.26]{Lu-Zheng}). Using the definition of constructible sheaves in \cite{Lu-Zheng} and boundedness of $\F$, we conclude that there is an open $\eta\in U\subset X$ such that $\F|_U$ has locally constant cohomology sheaves (in particular, the same holds for $\pi^*_X\F$). Then a standard argument shows that, after possibly shrinking $U$, one can also achieve that $U_{\ov{k}, \rm{red}}$ is smooth. 
\end{proof}

\begin{lemma}\label{lemma:existence-of-a-lattice} Let $X$ be a geometrically normal (i.e. $X_{\ov{k}}$ is normal) finite type $k$-scheme, $\ell$ a prime number invertible in $k$, and $\F$ a $\Q_\ell$-local system on $X\times_s \eta$. Then there is a $\Z_\ell$-local system $\G$ and an isomorphism $\G\left[\frac{1}{\ell}\right]\simeq \F$.
\end{lemma}
\begin{proof}
    The standard $t$-structure on $D^b_c(X\times_s \eta; \Q_\ell)$ is induced from the standard $t$-structure on $D^b_c(X\times_s \eta; \Z_\ell)$, so there is a sheaf $\G\in D^b_c(X\times_s \eta; \Z_\ell)^{\heartsuit}$ with an isomorphism 
    \[
    \F\simeq \G\left[\frac{1}{\ell}\right].
    \]
    Without loss of generality, we may and do assume that $\G$ is $\ell$-torsionfree. Furthermore, we can pass to connected components of $X$ to assume that $X$ is connected and, therefore, irreducible due to normality of $X$. Therefore, there is an open dense subset $U\subset X$ such that $\G|_{U\times_s \eta}$ is a $\Z_\ell$-local system (it suffices to check the same claim for $\G\otimes_{\Z_\ell} \bf{F}_\ell$ that follows from the proof of Lemma~\ref{lemma:lisse-stratification}). We denote by $j\colon U \to X$ the open immersion of $U$ into $X$.  Then the result follows from the following two claims: \smallskip
    
    {\it Claim~$1$. The natural morphism $\F\to \cal{H}^0\left(\rm{R}\left(j\times_s \eta\right)_*\F|_{U\times_s \eta}\right)$ is an isomorphism.} \smallskip
    
    {\it Claim~$2$. The object $\cal{H}^0\left(\rm{R}\left(j\times_s \eta\right)_*\G|_{U\times_s \eta}\right)$ is a $\Z_\ell$-local system.} \smallskip
    
    Now we discuss the proofs of both claims. In what follows we use Remark~\ref{rmk:adic-good} without saying, so we give references to the facts about torsion coefficients and freely apply them to the adic coefficients. \smallskip
    
    With that in mind, we recall that $\pi^*_X$ is conservative by Lemma~\ref{lemma:properties-Deligne}(\ref{lemma:properties-Deligne-3}) and $\pi^*_X\rm{R}(j\times_s \eta)_{*}$ is canonically isomorphic to $\rm{R}j_{\ov{k}, *}\pi^*_X$ by Lemma~\ref{lemma:pushforward-pullback}(\ref{lemma:pushforward-pullback-1}). Therefore, it suffices to prove analogous claims for a $\Q_\ell$-local system $\F$ on a normal, finite type $\ov{k}$-scheme $X_{\ov{k}}$ and a constructible $\Z_\ell$-lattice $\G$. This is standard and left to the reader. 
\end{proof}

\begin{cor}\label{cor:trivial-action-descend-rationally} Let $X$ be a geometrically normal finite type $k$-scheme, $\ell$ a prime number invertible in $k$, and $\F$ a $\Lambda$-local system on $X\times_s \eta$ for $\Lambda\in \{\Q_\ell, \Z_\ell, \Z/\ell^n\Z\}$. Suppose that the action of $I$ is trivial on $\pi^*_X\F$. Then the natural morphism
\[
\F \to p_X^*\cal{H}^0(\rm{R}p_{X, *} \F)
\]
is an isomorphism, and $\cal{H}^0(\rm{R}p_{X, *} \F)$ is a $\Lambda$-local system on $X$. 
\end{cor}
\begin{proof}
    The case of $\Lambda=\Z/\ell^n\Z$ follows from Lemma~\ref{lemma:trivial-action-descend} and the trivial observation that a lisse sheaf $\G\in \rm{Shv}_{\rm{lisse}}(X, \Z/\ell^n\Z)$  is a local system if and only if $c^*_X \G \in \rm{Shv}_{\rm{lisse}}(X_{\ov{s}}; \Z/\ell^n\Z)$ is. The case of $\Lambda=\Z_\ell$ follows by passing to the limit. The case of $\Lambda=\Q_\ell$ follows from Lemma~\ref{lemma:existence-of-a-lattice} by taking a $\Z_\ell$-local system $\G$ with an isomorphism $\G\left[\frac{1}{\ell}\right] \simeq \F$ (so the action of $I$ on $\G$ is automatically trivial). 
\end{proof}

Finally, we discuss the ``perverse'' $t$-structure on $D^b_c(X\times_s \eta; \Q_\ell)$. The idea of the construction is similar to that of constructible $t$-structure on $D^b_c(X\times_s \eta; \Q_\ell)$: we descend it from the perverse $t$-structure from $D^b_c(X_{\ov{s}}; \Q_\ell)$ (see \cite[Section III.1]{KW}). 

\begin{lemma}\label{lemma:perverse-t-structure-rationally} Let $X$ be a finite type $k$-scheme. Then the category $D^b_c(X\times_s \eta; \Q_\ell)$ admits a ``perverse'' $t$-structure:
\begin{enumerate}
    \item ${}^pD^{\leq 0}_c(X\times_s \eta; \Q_\ell)$ is the full subcategory of $D^b_c(X\times_s \eta; \Q_\ell)$ consisting of objects $\F$ such that $\pi^*_X\F\in {}^pD^{\leq 0}_c(X_{\ov{s}}; \Q_\ell)$;
    \item ${}^{p}D^{\geq 0}_c(X\times_s \eta; \Q_\ell)$ is the full subcategory of $D^b_c(X\times_s \eta; \Q_\ell)$ consisting of objects $\F$ such that $\pi^*_X\F\in {}^{p}D^{\geq 0}_c(X_{\ov{s}}; \Q_\ell)$.
\end{enumerate}
\end{lemma}
\begin{proof}
    As in the proof of Lemma~\ref{lemma:constructible-t-structure}, the only hard part is to show that the object ${}^{p}\tau^{\leq 0} \pi^*_X\F\in D^b_c(X_{\ov{s}}; \Q_\ell)$ and the morphism 
    \[
    ^{p}\tau^{\leq 0} \pi^*_X\cal{E}\to \pi^*_X\cal{E}
    \]
    in $D^b_c(X_{\ov{s}}; \Q_\ell)$ descends to $D^b_c(X\times_s \eta; \Q_\ell)$ for each $\cal{E}\in D^b_c(X\times_s \eta; \Q_\ell)$. \smallskip
    
    We prove it by induction on $\dim X$. If $\dim X=0$, then the constructible and perverse $t$-structures on $X_{\ov{s}}$ coincide, so the result follows from Corollary~\ref{cor:constructible-t-structure-rationally}. Now we suppose that the claim is known for all finite type $k$-schemes of dimension $< d$, and deduce it for $X$ of dimension $d$.\smallskip

    For brevity, we denote $\pi^*_X\cal{E}$ simply by $E$. Lemma~\ref{lemma:lisse-stratification} implies that there is a dense open $U\subset X$ such that $U_{\ov{k},\rm{red}}$ is smooth and $E|_{U_{\ov{k}}}$ has lisse cohomology sheaves. Let us denote by $F\in D^b_c(X_{\ov{s}}; \Q_\ell)$ the shifted cone\footnote{In the formula below, we treat $\dim U$ as a locally constant function $\dim U \colon |U| \to \Z_{\geq 0}$.}:
    \[
    F\coloneqq \rm{cone}(E \to \rm{R}j_{\ov{s}, *}\tau^{\geq 1 + \dim U}j_{\ov{s}}^*E)[-1],
    \]
    where $j\colon U\to X$ is the open immersion and $\tau^{\geq 1}$ is the truncation functor for the standard $t$-structure on $D^b_c(X_{\ov{s}}; \Q_\ell)$. Let us also denote by $A$ the shifted cone 
    \[
    A\coloneqq \rm{cone}(F \to i_{\ov{s}, *}{}^p\tau_Z^{\geq 1}i^*_{\ov{s}}E)[-1]
    \]
    where $i\colon Z=X\setminus U \to X$ is the complementary closed immersion and ${}^p\tau_Z^{\geq 1}$ is the perverse truncation on $D^b_c(Z_{\ov{s}}; \Q_\ell)$. This comes with a natural morphism $A \to E$, and the construction of ${}^p\tau^{\leq 0}$ in the proof of \cite[Lemma III.1.1]{KW} (in particular, see \cite[p.140 and Claim on p.141]{KW}) guarantees that this morphism is isomorphic to 
    \[
    ^{p}\tau^{\leq 0} E \to E. 
    \]
    Therefore, in order to descend the morphism $^{p}\tau^{\leq 0} E \to E$ it suffices to descend $\rm{R}j_{\ov{s}, *}$, $i_{\ov{s}, *}$, $j_{\ov{s}}^*$, $i^*_{\ov{s}}$, $\tau^{\geq 1}$, and ${}^p\tau_Z^{\geq 1}$. The first two functors descend by Lemma~\ref{lemma:pushforward-pullback}(\ref{lemma:pushforward-pullback-1}) (and Remark~\ref{rmk:adic-good}), the next two functors clearly descend, the truncation functor for the standard $t$-structure descends by Corollary~\ref{cor:constructible-t-structure-rationally}, and the perverse truncation ${}^p\tau_Z^{\geq 1}$ descends by the induction assumption.
\end{proof}

\begin{defn} A complex $\F\in D^b_c(X\times_s \eta; \Q_\ell)$ is {\it perverse} if it lies in the heart of the perverse $t$-structure. We denote the category of perverse sheaves by $\rm{Perv}(X\times_s \eta; \Q_\ell)$.
\end{defn}

\begin{lemma}\label{lemma:perverse-artinian} Let $X$ be a finite type $k$-scheme. Then $\rm{Perv}(X\times_s \eta; \Q_\ell)$ is an Artinian and Noetherian category.
\end{lemma}
\begin{proof}
    This follows from the facts that $\pi^*_X$ is conservative (see Lemma~\ref{lemma:properties-Deligne}(\ref{lemma:properties-Deligne-3}) and Remark~\ref{rmk:adic-good}) and that $\rm{Perv}(X_{\ov{s}}; \Q_\ell)$ is Artinian and Noetherian (see \cite[Corollary III.5.7]{KW}).
\end{proof}

\subsection{Six functors over an arithmetic field}

The main goal of this section is to develop a $6$-functor formalism for the $\Q_\ell$-constructible complexes on the Deligne's topoi $X\times_s \eta$. We develop this formalism under the additional assumption that the ground field $K$ is arithmetic (see Definition~\ref{defn:arithmetic}). Most likely, one can avoid this assumption by using the categorical gluing formalism from \cite{enhanced-operations}, but we do not pursue it in this paper. \smallskip

For the rest of the section, we assume that $K$ is an arithmetic field. 

\begin{lemma}\label{lemma:homotopy-2-limit} Let $K$ be an arithmetic non-archimedean field, $X$ a finite type $k$-scheme, and $\ell$ a prime number invertible in $\O_K$. Then the natural morphism
\[
D^b_c(X\times_s \eta; \Z_\ell) \to 2-\lim_n D^b_{ctf}(X\times_s \eta;\Z/\ell^n\Z)
\]
is an equivalence, where $2-\lim_n$ stands for the projective $2$-limit in the $2$-category of categories. 
\end{lemma}
\begin{proof}
    First, we note that $\F\otimes^L_{\Z_\ell} \Z/\ell^n\Z\in D^b_{ctf}(X\times_s \eta; \Z/\ell^n\Z)$ for any $\F\in D^b_c(X\times_s\eta; \Z_\ell)$ and an integer $n\geq 1$. Therefore, there is a natural functor
    \[
    \gamma\colon D^b_c(X\times_s \eta; \Z_\ell) \to 2-\lim_n D^b_{ctf}(X\times_s \eta;\Z/\ell^n\Z).
    \]
    Essentially by construction, this functor is essentially surjective. Thus we only need to show that it is also fully faithful. \smallskip
    
    Let $\F=``\lim_n\text{''} \F_n\in D^b_c(X\times_s \eta; \Z_\ell)$ and $\G=``\lim_n\text{''} \G_n\in D^b_c(X\times_s \eta; \Z_\ell)$. Since $D^b_c(X\times_s \eta; \Z_\ell)$ is the homotopy category of a full $\infty$-subcategory of $\lim_n \cal{D}(X\times_s \eta; \Z/\ell^n\Z)$, we see that there is Milnor's short exact sequence computing Hom groups:
    \[
    0 \to \rm{R}^1\lim_n \rm{Ext}^{-1}_{\Z/\ell^n\Z} (\F_n, \G_n) \to \rm{Hom}_{\Z_\ell}(\F, \G) \to \lim_n \rm{Hom}_{\Z/\ell^n\Z}(\F_n, \G_n) \to 0.
    \]
    Corollary~\ref{cor:finiteness-Deligne-topos} implies that $\rm{Ext}^{-1}_{\Z/\ell^n\Z} (\F_n, \G_n)$ are finite group. Thus the Mittag-Leffler criterion implies that the $\rm{R}^1\lim_n$-term vanishes. In other words, 
    \[
    \rm{Hom}_{\Z_\ell}(\F, \G) \simeq \lim_n \rm{Hom}_{\Z/\ell^n\Z}(\F_n, \G_n).
    \]
    This exactly means that $\gamma$ is fully faithful. 
\end{proof}

\begin{lemma}\label{lemma:functors} Let $K$ be an arithmetic non-archimedean field, $f\colon X \to Y$ a separated morphism of finite type $k$-schemes, and $\ell$ a prime number invertible in $\O_K$. Let $\F=``\lim_n\text{''} \F_n$ and $\G = ``\lim_n\text{''}\G_n$ be objects in $D^b_c(X\times_s \eta; \Z_\ell)$, and $\cal{H}=``\lim_n\text{''} \cal{H}_n$ an object in $D^b_c(Y\times_s \eta; \Z_\ell)$. Define 
\[
\rm{R}(f\times_s \eta)_* \F \coloneqq ``\lim_n\text{''} \rm{R}(f\times_s \eta)_* \F_n,
\]
\[
(f\times_s \eta)^* \cal{H} \coloneqq ``\lim_n\text{''} (f\times_s \eta)^* \cal{H}_n,
\]
\[
\rm{R}(f\times_s \eta)_! \F \coloneqq ``\lim_n\text{''} \rm{R}(f\times_s \eta)_! \F_n,
\]
\[
(f\times_s \eta)^! \cal{H} \coloneqq ``\lim_n\text{''} (f\times_s \eta)^! \cal{H}_n,
\]
\[
\F\otimes^L_{\Z_\ell} \G \coloneqq ``\lim_n\text{''} \F_n \otimes^L_{\Z/\ell^n\Z} \G_n,
\]
\[
\rm{R}\cal{H}om_{\Z_\ell}(\F, \G)\coloneqq ``\lim_n\text{''} \rm{R}\cal{H}om_{\Z/\ell^n\Z}(\F_n, \G_n).
\]
Then $\rm{R}f_*\F$ and $\rm{R}f_! \F$ are objects $D^b_c(Y\times_s \eta; \Z_\ell)$, while $(f\times_s \eta)^* \cal{H}$, $(f\times_s \eta)^! \cal{H}$, $\F\otimes^L_{\Z_\ell} \G$, and $\rm{R}\cal{H}om_{\Z_\ell}(\F, \G)$ are objects in $D^b_c(X\times_s \eta; \Z_\ell)$.
\end{lemma}
\begin{proof}
    Lemma~\ref{lemma:homotopy-2-limit} implies it suffices to show that all these functors satisfy base change with respect to the morphisms $\Z/\ell^n\Z \to \Z/\ell^{n-1}\Z$. The claim is essentially obvious for $(f\times_s \eta)^*$ and $-\otimes^L_{\Z/\ell^n\Z} -$. For other functors, this follows from Lemma~\ref{lemma:pushforward-pullback}(\ref{lemma:pushforward-pullback-3}), Lemma~\ref{lemma:shriek-pushforward-pullback}(\ref{lemma:shriek-pushforward-pullback-3}), Lemma~\ref{lemma:upper-shriek-pullback}(\ref{lemma:upper-shriek-pullback-5}), and Corollary~\ref{cor:projection-hom}. 
\end{proof}

\begin{defn}\label{defn:functors-rationally} Let $f\colon X \to Y$ be a separated morphism of finite type $k$-schemes. We define the {\it six functors}
\[
\rm{R}(f\times_s \eta)_*, \rm{R}(f\times_s \eta)_! \colon D^b_c(X\times_s \eta; \Z_\ell) \to D^b_c(X\times_s \eta; \Z_\ell),
\]
\[
(f\times_s \eta)^*, (f\times_s \eta)^! \colon D^b_c(Y\times_s \eta; \Z_\ell) \to D^b_c(X\times_s\eta; \Z_\ell),
\]
\[
-\otimes^L_{\Z_\ell} - \colon D^b_c(X\times_s\eta; \Z_\ell) \times D^b_c(X\times_s \eta; \Z_\ell) \to D^b_c(X\times_s \eta; \Z_\ell),
\]
\[
\rm{R}\cal{H}om_{\Z_\ell}(-, - )\colon D^b_c(X\times_s \eta; \Z_\ell)^{\rm{op}}\times D^b_c(X\times_s \eta; \Z_\ell)\to D^b_c(X\times_s \eta; \Z_\ell)
\]
as in Lemma~\ref{lemma:functors}. All these functors formally induce functors 
\[
\rm{R}(f\times_s \eta)_*, \rm{R}(f\times_s \eta)_! \colon D^b_c(X\times_s \eta; \Q_\ell) \to D^b_c(X\times_s \eta; \Q_\ell),
\]
\[
(f\times_s \eta)^*, (f\times_s \eta)^! \colon D^b_c(Y\times_s \eta; \Q_\ell) \to D^b_c(X\times_s\eta; \Q_\ell),
\]
\[
-\otimes^L_{\Q_\ell} - \colon D^b_c(X\times_s\eta; \Q_\ell) \times D^b_c(X\times_s \eta; \Q_\ell) \to D^b_c(X\times_s \eta; \Q_\ell),
\]
\[
\rm{R}\cal{H}om_{\Q_\ell}(-, - )\colon D^b_c(X\times_s \eta; \Q_\ell)^{\rm{op}}\times D^b_c(X\times_s \eta; \Q_\ell)\to D^b_c(X\times_s \eta; \Q_\ell). 
\]
\end{defn}

\begin{rmk}\label{rmk:rational-good} By a standard limit argument, one easily checks that all results from Appendix~\ref{appendix:deligne} stays true for the objects of $D^b_c(X\times_s \eta; \Z_\ell)$ and $D^b_c(X\times_s \eta; \Q_\ell)$.
\end{rmk}

\begin{rmk} Using \cite[Proposition 2.2.5]{BBD}, Lemma~\ref{lemma:pushforward-pullback}(\ref{lemma:pushforward-pullback-1}), and Lemma~\ref{lemma:shriek-pushforward-pullback}(\ref{lemma:shriek-pushforward-pullback-1}) (and Remark~\ref{rmk:rational-good}), we see that, for every quasi-finite morphism $f\colon X \to Y$ of finite type $k$-schemes, the functor
\[
\rm{R}(f\times_s \eta)_! \colon D^b_c(X\times_s \eta; \Q_\ell) \to D^b_c(Y\times_s \eta; \Q_\ell)
\]
is right perverse exact (see Lemma~\ref{lemma:perverse-t-structure-rationally}), and
\[
\rm{R}(f\times_s \eta)_* \colon D^b_c(X\times_s \eta; \Q_\ell) \to D^b_c(Y\times_s \eta; \Q_\ell)
\]
is left perverse exact. 
\end{rmk}

\begin{defn}\label{defn:intermediate-extension-Deligne} For a locally closed immersion $j\colon X\to Y$ between finite type $k$-schemes and a perverse sheaf $\F\in \rm{Perv}(X\times_s \eta; \Q_\ell)$, we define the {\it intermediate extension}
\[
(j\times_s \eta)_{!*} \F \coloneqq {}^p\rm{Im}\left({}^p\cal{H}^0\left(\left(j\times_s\eta\right)_!\F\right) \to {}^p\cal{H}^0\left(\rm{R}\left(j\times_s\eta\right)_*\F\right)\right) \in \rm{Perv}\left(Y\times_s \eta; \Q_\ell\right).
\]
\end{defn}

\begin{lemma}\label{lemma:extension-pullback} Let $K$ be an arithmetic field, $j\colon X\to Y$ be a locally closed immersion of finite type $k$-schemes, $\ell$ a prime number invertible in $k$, and $n\geq 1$ a positive integer. Then 
\begin{enumerate}
    \item\label{lemma:extension-pullback-1} the diagram 
    \[
\begin{tikzcd}
    \rm{Perv}(X\times_s \eta, \Q_\ell) \arrow{d}{(j\times_s\eta)_{!*}} \arrow{r}{\pi^*_{X}} & \rm{Perv}(X_{\ov{s}},\Q_\ell) \arrow{d}{j_{\ov{s}, !*}} \\
    \rm{Perv}(Y\times_s \eta, \Q_\ell) \arrow{r}{\pi^*_{Y}} & \rm{Perv}(Y_{\ov{s}}, \Q_\ell),
\end{tikzcd}
\] commutes (up to a canonical isomorphism);
\item\label{lemma:extension-pullback-2} the diagram
\[
\begin{tikzcd}
    \rm{Perv}(X\times_s \eta; \Q_\ell) \arrow{d}{(j\times_s\eta)_{!*}} \arrow{r}{{\sigma}_X^*} & \rm{Perv}(X; \Q_\ell) \arrow{d}{j_{!*}} \\
    \rm{Perv}(Y\times_s \eta; \Q_\ell) \arrow{r}{\sigma_Y^*} & \rm{Perv}(Y; \Q_\ell). 
\end{tikzcd}
\]
commutes (up to a canonical isomorphism) for every continuous section $\sigma\colon G_s \to G_\eta$;
\item\label{lemma:extension-pullback-3} the diagram
\[
\begin{tikzcd}
    \rm{Perv}(X; \Q_\ell) \arrow{r}{p_X^*} \arrow{d}{j_{!*}} & \rm{Perv}(X\times_s \eta; \Q_\ell) \arrow{d}{(j\times_s\eta)_{!*}} \\
    \rm{Perv}(Y; \Q_\ell)  \arrow{r}{p_Y^*}& \rm{Perv}(Y\times_s \eta; \Q_\ell) 
\end{tikzcd}
\]
commutes (up to a canonical isomorphism).
\end{enumerate}
\end{lemma}
\begin{proof}
    The proof is an easy consequence of Lemma~\ref{lemma:pushforward-pullback}, Lemma~\ref{lemma:shriek-pushforward-pullback}, and Lemma~\ref{lemma:perverse-t-structure-rationally}. Details are left to the reader. 
\end{proof}

\begin{lemma}\label{lemma:classification-simple-perverse-objects} Let $X$ be a finite type $k$-scheme, and $\F$ a simple perverse sheaf on $\rm{Perv}(X\times_s \eta; \Q_\ell)$. Then there is an irreducible subscheme $Y\subset X$, an open dense $U\subset Y$, and an irreducible local system $\G$ on $U$ such that $U_{\rm{red}}$ is smooth, and $\F\simeq (j\times_s \eta)_{!*}(\G[\dim Y])$. 
\end{lemma}
\begin{proof}
    The proof is identical to that of \cite[Corollary 5.5]{KW} using Lemma~\ref{lemma:lisse-stratification} and the usual properties of the six functors. 
\end{proof}

\section{Some facts from rigid-analytic geometry}\label{appendix:rigid-geometry}

In this section, we collect some facts about rigid-analytic spaces. These facts play an important role in our proof of \cite[Conjecture 4.15]{Bhatt-Hansen}. These results are probably well-known to the experts, but we find it difficult to extract them from the existing literature. For this reason, we decided to include the proofs in this appendix.

Throughout this section, we fix a non-archimedean field $K$. We denote by $\O_K$ its ring of integers and by $k$ its residue field. For an $\O_K$-algebra $A$, we set $A_K\coloneqq A\otimes_{\O_K} K$ and $A_k\coloneqq A\otimes_{\O_K} k$. 

\begin{lemma1} Let $\X$ be a quasi-compact admissible formal $\O_K$-scheme with generic fiber $X=\X_\eta$ and special fiber $\X_s$, let $\sp_\X \colon |X| \to |\X_s|$ be the specialization morphism, and let $\zeta\in \X_s$ be a generic point in the special fiber. Then $\sp_{\X}^{-1}(\zeta)$ consists of finitely many points of rank-$1$.
\end{lemma1}
\begin{proof}
    First, we can assume that $\X$ is a reduced formal scheme. Furthermore, the question is Zariski-local around the point $\zeta\in \X$, so we can assume that $\X=\Spf A$ is an irreducible reduced affine admissible formal $\O_K$-scheme and $\zeta\in \X_s$ is the unique generic point. For brevity, we denote by $k$ the residue field of $\O_K$ and by $A_k$ the tensor product $A\otimes_{\O_K} k$. Then \cite[Theorem 0.9.3.6]{FujKato} implies that there is a finite injective morphism $\O_K\langle T_1, \dots, T_d\rangle \to A$ such that the induced morphism $k[T_1, \dots, T_d] \to A_k$ is (finite) injective as well. Let $g\colon \Spf A \to \Spf \O_K\langle T_1, \dots, T_d\rangle$ be the induced map and let $\xi \in \Spec k[T_1, \dots, T_d]$ be the generic point. Since $\Spec A_k$ is irreducible, the dimension formula (see \cite[\href{https://stacks.math.columbia.edu/tag/02IJ}{Tag 02IJ}]{stacks-project}) applied to $(A_k)_{\rm{red}}$ implies that $g_{s}^{-1}(\{\xi\})=\{\zeta\}$. Therefore, \cite[Lemma~2.3.1]{LRZ} implies that it suffices to prove the claim for $\X = \Spf \O_K\langle T_1, \dots, T_d\rangle$. In this case, the result follows immediately from \cite[Lemma 4.3.2]{LRZ}. 
\end{proof}

\begin{cor1}\label{cor:good-opens} Let $K$ be a perfect non-archimedean field and let $\X$ be a reduced admissible formal $\O_K$-scheme. Then there is a dense open formal subscheme $\mathcal{U} \subset \X$ such that $\mathcal{U}_\eta$ is smooth over $\Spa(K, \O_K)$.
\end{cor1}
\begin{proof}
    For the purpose of proving this corollary, we can replace $\X$ with any quasi-compact dense open $\cal{U}\subset \X$. Therefore, we can assume that $\X$ is a disjoint union of its irreducible components. Dealing with one component at a time, we can assume that $\X$ is irreducible. Let $\zeta\in \X_s$ be the unique generic point, and let $\sp_{\X}^{-1}(\{\zeta\}) = \{\xi_1, \dots, \xi_n\}$. Then \cite[Lemma 2.8]{BLR3} implies that the regular locus of $X_\eta$ coincides with the smooth locus because $K$ is perfect. Therefore, \cite[Proposition 2.9 and Corollary 2.10]{Bhatt-Hansen} imply that there is a quasi-compact open $V\subset \X_\eta$ such that $\xi_1, \dots, \xi_n\in V$ and $V$ is smooth over $\Spa(K,\O_K)$. Now we note that $\cap_{i\in I} U_i = \{\zeta\}$, where the intersection is taken over all (quasi-compact) open non-empty subsets $U_i\subset \X_s$. Therefore, 
    \[
    \cap_{i\in I} \sp_{\X}^{-1}(U_i)= \sp_{\X}^{-1}(\cap_{i\in I} U_i) = \{\xi_1, \dots, \xi_n\}  \subset V. 
    \]
    Thus, \cite[\href{https://stacks.math.columbia.edu/tag/0A2W}{Tag 0A2W}]{stacks-project} implies that there is a (quasi-compact) open non-empty subset $U\subset \X_s$ such that $\sp_{\X}^{-1}(U)\subset V$. Let us denote by $\cal{U}\subset \X$ the associated open formal subscheme of $\X$. Then $\cal{U}$ does the job. 
\end{proof}


\begin{lemma1}\label{lemma:irreducible-components-hit-associated-points} Let $A$ be a flat, topologically finitely presented $\O_K$-algebra. Let $\p$ be a minimal prime ideal of $A$. Then every generic point in $\rV(\p) \cap \Spec A_k$ is an associated point of $\Spec A_k$. 
\end{lemma1}
\begin{proof}
    Choose a pseudo-uniformizer $\varpi\in A$. Then \cite[Proposition 3.4.3]{RG} and \cite[\href{https://stacks.math.columbia.edu/tag/058A}{Tag 058A}]{stacks-project} imply that $|\WeakAss(A/\varpi^n)|=|\Ass(A_k)|$ for any integer $n\geq 1$, so it suffices to show that every generic point of $\rV(\p) \cap \Spec A/\varpi^n$ lies in $\WeakAss(A/\varpi^n)$ for some $n\geq 1$. 

    Since $A$ is flat over $\O_K$, we conclude that $\varpi\notin \p$. This implies that $\p$ is saturated (see \cite[\textsection 0.8.1(c)]{FujKato}). In particular, we conclude that $\p$ is a finitely generated ideal by virtue of \cite[Proposition 0.8.5.3 and Corollary 0.9.2.7]{FujKato}. Therefore, \cite[\href{https://stacks.math.columbia.edu/tag/05C4}{Tag 05C4}]{stacks-project} and \cite[\href{https://stacks.math.columbia.edu/tag/058A}{Tag 058A}]{stacks-project} imply that there is an element $a\in A$ such that $\Ann_A(a)=\p$. Set $I\coloneqq (a)$, so \cite[\href{https://stacks.math.columbia.edu/tag/00L2}{Tag 00L2}]{stacks-project} ensures that $\Supp(I) = \rV(\p)$. Then \cite[\href{https://stacks.math.columbia.edu/tag/00L3}{Tag 00L3}]{stacks-project} implies that 
    \[
    \Supp(I/\varpi^n I) = \rV(\p)\cap \Spec A/\varpi^n.
    \]
    Then the Artin--Rees Lemma (see \cite[Proposition 0.8.2.13]{FujKato}) implies that there is an integer $c\geq 1$ such that $\varpi^c I \subset \varpi^c A\cap I \subset \varpi I$. Then $|\Supp(I/\varpi^c I)| = |\Supp(I/(\varpi^c A\cap I))| = |\Supp(I/\varpi I)|$. Therefore, this formula, the last sentence of the first paragraph, and \cite[\href{https://stacks.math.columbia.edu/tag/05C4}{Tag 05C4}]{stacks-project} imply that it suffices to show that $\WeakAss(I/(\varpi^c A\cap I)) \subset \WeakAss(A/\varpi^c)$. This follows directly from \cite[\href{https://stacks.math.columbia.edu/tag/0548}{Tag 0548}]{stacks-project}. 
\end{proof}

We now discuss the relation between the dimension of a rigid-analytic space $X$ and of its formal model $\X$. We use \cite[Definition 1.8.1]{H3} as our definition of the dimension in the rigid-analytic situation. However, we freely use that this definition coincides with the definition in \cite[Definition II.10.1.1]{FujKato} by virtue of \cite[Corollary II.10.1.10]{FujKato} and \cite[Lemma 1.8.6]{H3}. Furthermore, both definitions agree with the Krull dimension $\dim R$ for an affinoid rigid-analytic space $X=\Spa(R, R^\circ)$.

\begin{prop1}\label{prop:equi-dimension} Let $\X$ be an admissible formal $\O_K$-scheme. Then 
\begin{enumerate}
   \item\label{prop:equi-dimension-1} the scheme $\X_s$ is of pure dimension $d$ if the rigid-analytic space $\X_\eta$ is of pure dimension $d$;
   \item\label{prop:equi-dimension-2} the rigid-analytic space $\X_\eta$ is of pure dimension $d$ if the scheme $\X_s$ is of pure dimension $d$ and has no embedded points (see \cite[\href{https://stacks.math.columbia.edu/tag/05AK}{Tag 05AK}]{stacks-project}).
\end{enumerate}
\end{prop1}
    We note that \cite[\href{https://stacks.math.columbia.edu/tag/0346}{Tag 0346}]{stacks-project} ensures that $\X_s$ has no embedded points if and only if it is $(S_1)$. 
\begin{proof}
    Part~(\ref{prop:equi-dimension-1}) was already proven in \cite[Corollary B.4]{Z-thesis}. Therefore, we only need to show part~(\ref{prop:equi-dimension-2}). The claim is local on $\X$, so we can assume that $\X = \Spf A$ is affine. Furthermore, we note that \cite[Lemma 2.1.5]{conrad} together with \cite[Theorem II.10.1.8 and Proposition II.10.1.9]{FujKato} imply that it suffices to show that each irreducible component of $\X_\eta$ has dimension $d$. In other words, we need to show that $\dim A_K/\p' =d$ for any minimal prime ideal $\p'\subset A_K$. By the $\O_K$-flatness of $A$, any such minimal prime ideal corresponds to a minimal prime ideal $\p\subset A$. Then $A/\p$ is a flat, topologically finitely presented $\O_K$-algebra, so Lemma~\ref{lemma:irreducible-components-hit-associated-points} ensures that $\Spec (A/\p)_k$ meets a generic point of $\Spec A_k$. Since $\Spec A_k$ is of pure dimension $d$, we conclude that $\dim (A/\p)_k$ is of dimension $d$. Therefore, \cite[Lemma B.3]{Z-thesis} implies that $\dim A_K/\p' = \dim (A/\p)_K =d$. This finishes the proof. 
\end{proof}

The assumption in Proposition~\ref{prop:equi-dimension}(\ref{prop:equi-dimension-2}) is optimal. The following example was communicated to us by Shizhang Li.

\begin{example1} Let $\varpi\in \O_K$ be a pseudo-uniformizer, let $A\coloneqq \O_K\langle X\rangle$, let $B\coloneqq \O_K$, let $C\coloneqq \O_K\langle Y\rangle/(Y^2-\varpi Y)$, and let $f\colon A\to B$ and $g\colon C \to B$ be the unique continuous $\O_K$-linear homomorphisms that send $X$ and $Y$ to $0$. Set $D \coloneqq A\times_B C$ to be the fiber product of these rings. Then $D$ is a flat, topologically finitely presented $\O_K$-algebra such that $(\Spf D)_s$ is pure of dimension $1$, but $(\Spf D)_\eta$ is not of pure dimension. 
\end{example1}
\begin{proof}
    First, we note that $D$ is a subring of $A\times C$, therefore it is flat over $\O_K$. Then one see that $I\coloneqq \ker(D \twoheadrightarrow C) \simeq \ker(A \twoheadrightarrow B)$. Since $A$ and $B$ are $\varpi$-adically complete,  \cite[\href{https://stacks.math.columbia.edu/tag/091T}{Tag 091T}]{stacks-project} and \cite[\href{https://stacks.math.columbia.edu/tag/091U}{Tag 091U}]{stacks-project} imply that $I$ is derived $\varpi$-adically complete. Now using that $C$ is $\varpi$-adically complete, we also conclude that $D$ is derived $\varpi$-adically complete. In particular, $D\simeq \rR\lim_n [D/\varpi^n]$. Since $D$ has no $\varpi$-torsion, we conclude that $D\simeq \lim_n D/\varpi^n D$, i.e, $D$ is $\varpi$-adically complete. Now we note that the natural morphism $D \hookrightarrow A\times C$ realizes $A\times C$ as a finite $D$-module because both projections $D \to A$ and $D\to C$ are surjective. Since $B$ is $\O_K$-flat, we conclude that $D\subset A\times C$ is saturated. Therefore, \cite[Lemma 3.2.5]{Z-quotient} implies that $D$ is topologically finite type over $\O_K$. Finally, \cite[Corollary 7.3/5]{B} implies that $D$ is topologically finitely presented over $\O_K$. 

    Now we need to show that $\Spf D$ has special fiber of pure dimension $1$ and that it has generic fiber which is not of pure dimension. Using that the short exact sequences $0 \to I \to D \to C \to 0$ and $0 \to I \to A \to B \to 0$ remain exact after applying $-\otimes_{\O_K} K$ and $-\otimes_{\O_K} k$, we conclude that $D_k \simeq A_k\times_{B_k} C_k$ and $D_K \simeq A_K \times_{B_K} C_K$. Therefore, one easily checks that 
    \[
    D_k \simeq k[X] \times_k k[Y]/(Y^2) \simeq k[X, Y]/(XY, Y^2) \text{ and }\]
    \[
    D_K \simeq K\langle X\rangle \times_K \big(K\langle Y\rangle/(Y) \times K\langle Y\rangle/(Y-\varpi)\big) \simeq K\langle X\rangle \times K. 
    \]
    Therefore, $\Spec D_k$ is of pure dimension $1$, but $\Spa(D_K, D_K^\circ)$ has an irreducible component of dimension $0$. 
\end{proof}




\printbibliography

\end{document}